\documentclass[twoside,a4paper,11pt]{amsart}

\usepackage{MyStyle}

\begin{document}
\title{Higher covering spaces in an $\infty$-topos}
\author{Virgile Constantin}
\begin{abstract}
   We develop a systematic theory of $(n-1)$-truncated maps, called $n$-covering maps, in a fixed $\infty$-topos $\E$, guided by the analogy with classical covering spaces. We prove an equivalence of $n$-categories between $n$-coverings over a pointed connected object $(X,x)$ and $\infty$-actions of the fundamental $n$-group $\Pi_n(X,x)$ on $(n-1)$-truncated objects, which restricts to a classification of pointed connected $n$-coverings in terms of sub-$n$-groups of $\Pi_n(X,x)$. We study the $n$-group of deck transformations $\Deck(p)$, identifying it with $\Pi_n(X,x)$-equivariant autoequivalences of the fiber $F$. For normal $n$-coverings, it is further described as a quotient of $\Pi_n(X,x)$, yielding a classification of such coverings in terms of normal subgroups of $\pi_n(X,x)$. 
   For an arbitrary $n$-covering, the deck $n$-group arises as a quotient of a suitable normalizer. 
   Our approach relies on a careful study of $n$-groups and their $\infty$-actions, on the use of univalent universes, and on an internal Yoneda embedding. When $n=1$ and $\E$ is the $\infty$-category of homotopy types, our results recover the classical theory of covering spaces. We further illustrate the theory in sheaf and étale $\infty$-topoi, where the external deck group recovers cohomology of the base, and in cohesive $\infty$-topoi, where it recovers the $1$-covering theory of manifolds.
\end{abstract}
\maketitle
\vspace{-\baselineskip}
\tableofcontents
\section*{Introduction}
\renewcommand{\thetheorem}{\Alph{theorem}}
The theory of covering spaces, namely of fibrations $p\colon E\to X$ whose fibers are sets, studies their relationship with the fundamental group $\pi_1(X)$ of the pointed base space. A wide range of results can be established, from classification theorems in terms of $\pi_1(X)$-actions, or in terms of subgroups of $\pi_1(X)$, to descriptions of the group $\Deck(p)$ of deck transformations, that is, fiberwise homeomorphisms of the total space $E$. Since the theory is invariant under homotopy equivalences, it is natural to regard it as a theory of $0$-truncated maps. Indeed, different authors \cite{HouHarper_CoveringSpacesInHoTT, HigherGroupsInHoTT_BuchholtzVanDoornRijke, WemmenhoveManeaPortegies_ClassificationCoveringSpacesAndChangeOfBasePoint, MimramOleon_ClassifyingCoveringHoTT} have in this perspective developed aspects of the theory within the syntactic framework of Homotopy Type Theory (HoTT) \cite{hottbook}, whose models are inherently homotopy theoretic. It is therefore natural to ask whether analogous results hold for higher truncation. The aim of this article is to propose a systematic study of $(n-1)$-truncated maps in a fixed $\infty$-topos $\E$ with the guiding analogy of covering space theory. Concretely, we show that the main structural results of classical covering space theory admit natural generalizations to $(n-1)$-truncated maps in this context. We therefore call them \textit{$n$-covering maps}. The key principle is that the mentioned relationships with the fundamental group $\pi_1(X)$ hold in general with the fundamental $n$-group $\Pi_n(X)$ of the base object. When $n=1$, and $\E=\Sp$ is the $\infty$-topos of homotopy types, our results recover the classical theory.

\medskip
 
Higher topoi are $\infty$-categories with several complementary facets, studied in the context of model categories by Rezk in \cite{Rezk_HomotopyTopos} and in the framework of quasicategories by Lurie in \cite{LurieHTT}. They can be viewed as generalized spaces, as illustrated by the example of $\infty$-groupoid-valued sheaves on a topological space, and are commonly defined as left exact localizations of presheaf $\infty$-categories $\Fun(\C^\op,\Sp)$. While this presentation is succinct and efficient, it is not the perspective we adopt here. Instead we think of an $\infty$-topos $\E$ as an $\infty$-category of generalized homotopy types, and our reasoning and constructions are guided by intuitions coming from homotopy theory. Concretely we work from the Giraud-Rezk-Lurie axioms (\cref{def: infinity topos}), which characterize $\infty$-topoi internally through categorical properties such as universality of colimits, descent, and effectiveness of groupoid objects, without reference to a presheaf $\infty$-category. These properties reflect well known phenomena familiar from the $\infty$-category $\Sp$ of homotopy types. For instance, universality and descent for pushouts recover Mather’s cube theorems \cite{Mather_PullsbackInHomotopyTheory}, while the fact that group-like $A_\infty$-spaces $G$ arise as loop spaces of pointed connected spaces $\B G$ \cite{Stasheff_HomotopyAssociativityOfHSpaces, May_GeometryOfIteratedLoopSpaces} reflects the effectiveness of groupoids. Crucially, this axiom also captures the classical equivalence between $G$-spaces and fibrations over $\B G$ \cite{FarjounDwyerKan_EquivariantMapsWhichAreSelfHomotopyEquivalences}. This approach forces us to argue from first principles rather than importing results from spaces, leading to proofs that are more flexible and robust under changes of hypotheses. While many results in $\Sp$ can be transported to presheaf $\infty$-categories and then along left exact localizations $\Psh(\C)\to\E$, we deliberately avoid this route in order to isolate the intrinsic structural features at play. In particular, many arguments rely solely on general principles expected to hold beyond the presentable setting, for instance in elementary $\infty$-topoi as developed by Rasekh \cite{Rasekh_TheoryElementaryHigherToposes}. The interpretation of objects of $\E$ as generalized spaces is further supported by the fact that $\infty$-topoi provide models of Homotopy Type Theory, as established by Shulman in \cite{Shulman_InfinityToposesStrictUnivalentUniverse}. Motivated by this connection, we make extensive use of univalent universes, and rely on the alternative characterization of $\infty$-topoi as presentable $\infty$-categories with universal colimits and sufficient univalent universes (\cref{Higher topoi have enough universes}).

\medskip

As the fundamental $n$-group plays a central role in our results, we recall the idea behind higher groups. An $\infty$-group $G$ can be thought of as a space equipped with a multiplication that is associative, unital, and invertible only up to coherent higher homotopies. Concretely, instead of having strict equalities like $(ab)c=a(bc)$, we have specified paths between the two sides, together with higher paths expressing the compatibility of these choices, and so on indefinitely. From this viewpoint, a classical group is the special case where all higher homotopies are trivial. An $n$-group is obtained by truncating this structure so that no nontrivial homotopical information remains above level~$n$. These higher homotopies can be conveniently encoded using simplicial objects. The classical example is that of loop spaces $\Omega X$ whose multiplication is path concatenation. In fact, the effectiveness of groupoids implies that every $\infty$-group $G$ arises as the loop object of its classifying object $\B G$, which is $n$-truncated when $G$ is an $n$-group. The fundamental $n$-group $\Pi_n(X,x)$ of a pointed space $X$ is the main example in this work. Intuitively it records all loops based at $x$, $1$-homotopies between such loops, $2$-homotopies between such $1$-homotopies, and so on, up to $(n-1)$-homotopy \textit{classes} between $(n-2)$-homotopies. When $n=1$, this reduces to the fundamental group $\pi_1(X,x)$. When $n=2$, it additionally remembers how homotopies between loops can be deformed, giving a richer invariant that detects two-dimensional features of the space. In general, $\Pi_n(X,x)$ can be viewed as the object governing $n$-truncated phenomena over $X$, and indeed can be defined as the loop space $\Omega_x(\tau_nX)$ of the $n$-truncation of $X$.

\medskip

\textbf{Main results of the paper.}
We summarize here the main results concerning $n$-covering maps. As noted above, these results recover the classical theory when $n=1$ and $\E=\Sp$. A preliminary property asserts that for a fixed pointed object $X$, the $\infty$-category of pointed $n$-coverings admits an initial object (\cref{prop: characterization of universal cover}), characterizing the universal $n$-cover of $X$. This result was established in the syntactic setting of HoTT in \cite{MimramOleon_ClassifyingCoveringHoTT}. 

\medskip

\hspace{0.5cm}\textbf{Classification via the fundamental $n$-group.} 
Our first main result generalizes the classical correspondence between coverings and $\pi_1(X)$-actions, usually called a Galois correspondence.
\begin{theorem}[{\cref{Theorem: representation theorem for n-covers}}]\label{theorem intro: representation theorem for n-covers}
     Let $(X,x)\in\E_\ast$ be a pointed connected object, and let $1\leq n<\infty$ be an integer. Then there is an equivalence of $n$-categories
    \begin{equation*}
        \Cov_n(X)\simeq \Act_{\Pi_n(X,x)}(\E^{\sleq{{n-1}}})
    \end{equation*}
    between $n$-covering objects over $X$ and $\infty$-actions of the fundamental $n$-group of $X$ on $(n-1)$-truncated objects. It restricts to an equivalence \[\Cov_{\ast,n}(X)\simeq \Act_{\Pi_n(X,x)}(\E_\ast^{\sleq n-1})\]
    between pointed $n$-covers and pointed $\infty$-actions. For each $1\leq k\leq n$, these restrict to equivalences between $(k-1)$-connected (pointed) $n$-covers and $k$-transitive (pointed) $\infty$-actions.
\end{theorem}

The appropriate notion of $\infty$-actions of an $\infty$-group $G$ on an object $X$, as studied by Nikolaus, Schreiber, and Stevenson in \cite{NSS_PrincipalBundlesGeneralTheory}, follows the same underlying principle as $\infty$-groups. For example, the basic equality $g_1\cdot(g_2\cdot x)=(g_1g_2)\cdot x$ does not hold strictly, but there is a path in $X$ between them, and those paths satisfy higher coherences. While classically connected covers correspond to the transitive actions $\pi_1(X,x)\curvearrowright F$, we extend this result to an appropriate notion of $k$-transitive $\infty$-actions (\cref{def: n transitive free and faithful actions}).

\medskip

It is worth noting one advantage of working externally, as opposed to within HoTT. The case $n=1$ of \cref{theorem intro: representation theorem for n-covers} has been established in HoTT by Hou and Harper in \cite{HouHarper_CoveringSpacesInHoTT}. However, a syntactic equivalence is interpreted in an $\infty$-topos $\E$ as an equivalence between the internal object of $n$-covering maps and the internal object of actions of the fundamental $n$-group. Passing to the external viewpoint, using global sections, therefore yields an equivalence only between the maximal subgroupoids (the cores) of the external $n$-categories appearing in \cref{theorem intro: representation theorem for n-covers}. This limitation is inherent to the internal perspective, and applies to the following \cref{theorem intro: pointed coverings corresponds to subgroups}, which for $n=1$ identifies pointed connected covering maps with subgroups of the fundamental group. Internal versions were established for $n=1$ by Wemmenhove, Manea, and Portegies in \cite{WemmenhoveManeaPortegies_ClassificationCoveringSpacesAndChangeOfBasePoint}, and independently, using different techniques, by Mimram and Oleon in \cite{MimramOleon_ClassifyingCoveringHoTT}. When interpreted in an $\infty$-topos, they yield an equivalence between the sets of isomorphism classes of pointed connected coverings and of subgroups, whereas we obtain an equivalence of posets, capturing the lattice structure of $1$-covering spaces.
\begin{theorem}[{\cref{corollary: pointed coverings correspond to subgroups}}]\label{theorem intro: pointed coverings corresponds to subgroups}
    Let $(X,x)\in\E_\ast$ be a pointed connected object and $1\leq k\leq n$ be integers. Then there is an equivalence of $\infty$-categories: 
    \[\Cov^{\sg{k-1}}_{\ast,n}(X)\simeq\Sub_{(n-k+1,k)}\left(\Pi_n(X,x)\right)\]
    between $(k-1)$-connected pointed $n$-covering maps over $X$, and $k$-symmetric sub-$(n-k+1)$-groups of $\Pi_n(X,x)$.
\end{theorem}
The statement can be seen as a corollary of \cref{theorem intro: representation theorem for n-covers} together with a characterization of pointed $\infty$-actions (\cref{cor: characterization of pointed action}).
Recall that an $n$-group $G$ is the loop object of an $n$-truncated classifying object $\B G$. More generally, if there exists a pointed $(k-1)$-connected object $\B ^kG$ equipped with an $n$-group equivalence $\Omega^{k}\B ^kG\simeq  G$, we call the $n$-group $G$ \textit{$k$-symmetric} (\cref{def: k-symmetric n-groups}). For $k\geq 2$, the $n$-group structure on $G$ exhibits additional commutativity: the Eckmann–Hilton argument, which proves that the $1$-group $\pi_0(G)$ is abelian, is one instance of this phenomenon. In fact, $k$-symmetric $n$-groups are $n$-groups that admit an $E_k$-algebra structure, where $E_k$ is the little $k$-discs operad. They have been studied in HoTT by Buchholtz, van Doorn, and Rijke in \cite{HigherGroupsInHoTT_BuchholtzVanDoornRijke} under the name of $k$-tuply groupal $n$-groupoids, and more recently in \cite{BuchholtzRijke_LESofHomotopynGroups}, from which we borrow the terminology. Extending this concept, we introduce an appropriate notion of $k$-symmetric sub-$m$-groups (\cref{def: sub n groups}). For any pointed $n$-covering map $p\colon E\to X$, the induced map $\Pi_n(p)\colon\Pi_n(E,e)\to\Pi_n(X,x)$ defines a $1$-symmetric sub-$n$-group. However, to extract meaningful results about deck transformations (\cref{theorem intro :action of Deck is n-free} \& \ref{theorem intro: group deck is the quotient of fundamental n-group}), and normality (\cref{theorem intro:characterization normal n-covering maps}), we must restrict our attention to $(n-1)$-connected $n$-covering maps, analogous to restricting to connected coverings when $n=1$. Under the connectivity assumption, the fundamental $n$-group simplifies to $\Pi_n(E,e)\simeq\B^{n-1}\pi_n(E,e)$, which defines an $n$-symmetric sub-$1$-group of $\Pi_n(X,x)$. These specific subgroups are of central importance in this paper. In \cref{theorem: n-symmetric subgroups of the fundamental n-group are 1-subgroups of n-th homotopy group}, we prove that $n$-symmetric sub-$1$-groups of $\Pi_n(X,x)$ are identified with sub-$1$-groups of $\pi_n(X,x)$. Hence when $k=n$, \cref{theorem intro: pointed coverings corresponds to subgroups} specializes to a clean classification of pointed $(n-1)$-connected $n$-coverings in terms of subgroups of the $n$-th homotopy group (\cref{corollary: pointed n-1 connected coverings corresponds to subgroups of homotopy n-th group}):
\begin{equation}\label{theorem intro: pointed coverings corresponds to subgroups eq}
    \Cov_{\ast,n}^{\sg{n-1}}(X)\simeq \Sub_1(\pi_n(X,x)).
    \end{equation}
Specifically, the $n$-covering map $E_A\to X$ associated to a subgroup $A\leq\pi_n(X,x)$ is $(n-1)$-connected, satisfies $\pi_n(E_A)\cong A$, and has its higher homotopy groups canonically identified with those of $X$.

\medskip
 
\hspace{0.5cm}\textbf{Deck transformations.}
Let $p\colon E\to X$ be a map with fiber $F$, where $X$ is pointed and connected. The $\infty$-group of deck transformations $\Deck(p)$ encodes autoequivalences of $E$ over $X$: its points are equivalences $\varphi\colon E\xrightarrow{\sim}E$ together with a homotopy $p\circ\varphi\simeq p$. If $F\to E\to X$ is a fiber sequence, there is an $\infty$-action $\Omega X\curvearrowright F$, which allows us to define the $\infty$-group $\Aut_{\Omega X}(F)$ of $\Omega X$-equivariant autoequivalences of $F$. We show that there is an equivalence of $\infty$-groups \[\Deck(p)\simeq \Aut_{\Omega X}(F).\] If $p$ is an $n$-covering map, so that $F$ is $(n-1)$-truncated, the induced $\infty$-action factors through the fundamental $n$-group $\Pi_n(X,x)$ as in \cref{theorem intro: representation theorem for n-covers}. The following result relates the $n$-group of $\Pi_n(X,x)$-equivariant autoequivalences of $F$ with the $n$-group of deck transformations.
\begin{theorem}[{\cref{corollary: deck transformations as Omega X-equivariant autoequivalences} \& \cref{prop: group deck of a covering is pi_1-equivariant automorphism of the fiber}}]\label{theorem intro: deck transformations as Omega X-equivariant autoequivalences}
   Let $(X,x)\in\E_\ast$ be a pointed connected object, and let $p\colon E\to X$ be an $n$-covering map with fiber $F$. Then there is an equivalence of $n$-groups:  \[\Deck(p)\simeq \Aut_{\Omega X}(F) \simeq \Aut_{\Pi_n(X,x)}(F).\]
    Moreover, this equivalence respects the respective $\infty$-actions on $F$.
\end{theorem}

The $\infty$-group $\Deck(p)$, being defined by fiberwise equivalences, acts naturally on the fiber $F$ by evaluation $\varphi\cdot e\coloneq\varphi(e)$. We extend the classical notion of freeness of group action to $\infty$-groups and introduce the notion of $n$-freeness.  
\begin{theorem}[{\cref{theorem: action of Deck is n-free}}]\label{theorem intro :action of Deck is n-free}
  Let $(X,x)\in\E_\ast$ be a pointed connected object, and let $p\colon E\to X$ be an $(n-1)$-connected $n$-covering map with fiber $F$. Then the $\infty$-action $\Deck(p)\curvearrowright F$ is $n$-free. 
\end{theorem}
Via the identification of \cref{theorem intro: deck transformations as Omega X-equivariant autoequivalences},
the $\infty$-actions of $\Deck(p)$ and $\Pi_n(X,x)$ on $F$ commute. 
The former is $n$-faithful, while the latter is $n$-transitive when $E$ is $(n-1)$-connected. In the classical case $n=1$, this immediately implies freeness. 
In the higher setting, the argument requires a substantially more involved analysis of $\infty$-actions and their properties, carried out in \cref{subsection: action from a product} and culminating in \cref{Theorem: action of a product restricted to H is free}.

\medskip

The following theorem is the central result of this article. It identifies the $n$-group of deck transformations of certain $n$-coverings with a quotient of the fundamental $n$-group, generalizing the classical fact that for a normal covering, the deck group is the quotient $\pi_1(X)/\pi_1(E)$. We call a morphism of $\infty$-groups $f\colon A\to G$ \textit{normal} if it fits into a fiber sequence $A\to G\to Q$ of $\infty$-groups, that is, if it arises as the kernel of another morphism of $\infty$-groups. The quotient $Q$ is denoted $G\sslash A$. For a subgroup $A\leq\pi_n(X,x)$, we write $p_A\colon E_A\to X$ for the corresponding $(n-1)$-connected $n$-covering map under the classification \eqref{theorem intro: pointed coverings corresponds to subgroups eq}.
\begin{theorem}[{\cref{cor: group deck is the quotient of fundamental n-group}}]\label{theorem intro: group deck is the quotient of fundamental n-group}
        Let $(X,x)$ be a pointed connected object of an $\infty$-topos $\E$, and let $A\to \pi_n(X,x)$ be a normal sub-$1$-group object. Then there is an equivalence of $n$-groups \[\Deck(p_A)\simeq \Pi_{n}(X,x){\sslash \B ^{n-1}A}.\]
        Moreover this equivalence respects the respective $\infty$-actions on the fiber $F_A$ of $p_A$.
\end{theorem}
This result was established in HoTT in the special case where 
$A=1\leq\pi_n(X,x)$ is the trivial sub-$1$-group 
\cite{HigherGroupsInHoTT_BuchholtzVanDoornRijke}, and our proof is inspired by their approach. The idea is to construct both sides as the same connected component of the same internal mapping object. Given a univalent universe $\U$, its connected components are classifying objects $\B\Aut(F)$ for objects $F$ it classifies. In a similar fashion, if $p\colon E\to X$ is classified by a map $X\to\U$, one can identify the connected component of the adjoint map $\ast\to\U^X$ with the classifying object $\B\Deck(p)$ of fiberwise equivalences of $E$. On the other hand, one realizes that this global point of $\U^X$ can be viewed as the "representable" object $\B \big(\Pi_n(X,x)\sslash \B^{n-1}A\big)(-,x)$. Carrying this out rigorously required an internal version of the Yoneda embedding \cref{Yoneda Embedding}.

\medskip

\cref{theorem intro: group deck is the quotient of fundamental n-group} is the essential input for the characterization of normal $n$-coverings (\cref{theorem intro:characterization normal n-covering maps}) discussed below, as well as for the identification of the deck transformation $n$-group of a normal $n$-covering map \eqref{intro: normal covering eq}.

\medskip

A subtle point accompanies the statement of \cref{theorem intro: group deck is the quotient of fundamental n-group}. For general $\infty$-groups, a quotient $G\sslash A$ need not carry a unique compatible $\infty$-group structure: there may be distinct deloopings $\B Q$ fitting into a fiber sequence $\B A \to \B G \to \B Q$ \cite{BeardsleyFoxHigherGroupsHigherNormality}. Our setting, however, bypasses this ambiguity: the proof identifies the quotient, independently of the chosen $n$-group structure, with the $n$-group of deck transformations $\Deck(p_A)$, whose group structure is intrinsic. As a consequence, when $\B^{n-1}A \to G$ is a normal $n$-symmetric sub-$1$-group, the quotient $G\sslash \B^{n-1}A$ carries a unique $n$-group structure  (\cref{cor: quotient of n-symmetric sub-1-groups are unique}). When $n=1$, this recovers the classical fact that the quotient of a group by a normal subgroup admits a unique compatible group structure.

\medskip

\hspace{0.5cm}\textbf{Normal $n$-coverings.}
An $n$-covering map $p\colon E\to X$ is called \emph{normal} if its total space $E$ is $(n-1)$-connected and the induced $\Deck(p)$-action on the fiber $F$ is $n$-transitive. The central result expresses the $n$-group of deck transformations in terms of fundamental $n$-groups: for a normal $n$-covering map $p$, the induced morphism on fundamental $n$-groups $\B^{n-1}\pi_n(E,e)\to\Pi_n(X,x)$ is normal, and one has a canonical equivalence of $n$-groups:
\begin{equation}\label{intro: normal covering eq}
    \Deck(p)\simeq \Pi_n(X,x)\sslash \B^{n-1}\pi_n(E,e),
\end{equation}
where $E$ is canonically pointed (see \cref{corollary: group deck of a normal covering is a quotient of fundamental group}).
Because of the connectivity assumption on $E$, its fundamental $n$-group records only its $n$-th homotopy group, so that $\Pi_n(E,e)\simeq\B^{n-1}\pi_n(E,e)$. The proof of \eqref{intro: normal covering eq} relies on \cref{theorem intro: group deck is the quotient of fundamental n-group}, which allows the identification of the deck $n$-group as a quotient of fundamental $n$-groups, provided the covering corresponds to a \textit{normal} morphism of $\infty$-groups. The crucial step is therefore to connect the geometric normality condition with the algebraic normality of the associated morphism of fundamental $n$-groups. The key observation is that for a normal $n$-covering $p$, the $\Deck(p)$-action on the fiber $F$ is not only $n$-transitive by definition, but also $n$-free (by \cref{theorem intro :action of Deck is n-free}). Together, these properties force the quotient $F\sslash\Deck(p)$ to be contractible. From this, a short manipulation recovers the base space $X$ as the quotient of the total space $E$ by the deck action: $E\sslash \Deck(p)\simeq X.$ This identification implies that the $\Deck(p)$-action on $E$ is classified by a fiber sequence $E\xrightarrow{p}X\to\B\Deck(p)$. Since $\B\Deck(p)$ is $n$-truncated, the second map factors through $X\to\B\Pi_n(X,x)$, resulting in a pasting of Cartesian squares:
\[\begin{tikzcd}
	E & {\B^n\pi_n(E,e)} & \ast \\
	X & {\B\Pi_n(X,x)} & {\B\Deck(p)}
	\arrow[from=1-1, to=1-2]
	\arrow["p"', from=1-1, to=2-1]
	\arrow["\lrcorner"{anchor=center, pos=0.125}, draw=none, from=1-1, to=2-2]
	\arrow[from=1-2, to=1-3]
	\arrow[from=1-2, to=2-2]
	\arrow["\lrcorner"{anchor=center, pos=0.125}, draw=none, from=1-2, to=2-3]
	\arrow[from=1-3, to=2-3]
	\arrow[from=2-1, to=2-2]
	\arrow[from=2-2, to=2-3]
\end{tikzcd},\]
in which the right-hand square exhibits $\B^{n-1}\pi_n(E,e)\to\Pi_n(X,x)$ as a normal morphism of $\infty$-groups. Using the correspondence between pointed $n$-coverings and sub-$1$-groups \eqref{theorem intro: pointed coverings corresponds to subgroups eq}, the left-hand square means that $p$ corresponds to the associated normal sub-$1$-group. One then concludes that \eqref{intro: normal covering eq} is an equivalence by applying \cref{theorem intro: group deck is the quotient of fundamental n-group}. A parallel argument establishes the converse direction (algebraic normality implies geometric normality), leading to the following characterization.
\begin{theorem}[{\cref{theorem: characterization normal n-covering maps}}]\label{theorem intro:characterization normal n-covering maps}
Let $\E$ be an $\infty$-topos, let $(X,x)$ be a pointed connected object, and let $p\colon E\to X$ be an $(n-1)$-connected $n$-covering map with fiber $F$. The following are equivalent:
\begin{enumerate}
    \item $p$ is a normal $n$-covering map;
    \item there is an equivalence $E\sslash \Deck(p)\simeq X$ under $E$;
    \item $p$ corresponds via \cref{corollary: pointed n-1 connected coverings corresponds to subgroups of homotopy n-th group} to a normal sub-$1$-group $A\hookrightarrow \pi_n(X,x)$.
\end{enumerate}
If one of these conditions is satisfied, there exists a basepoint $e\colon \ast\to F\to E$ rendering $p$ pointed. Moreover, for any such basepoint, the $1$-group $A$ is identified with both the $n$-th homotopy group $\pi_n(E,e)$, and the stabilizer $n$-symmetric $1$-group of the action $\pi_n(X,x)\curvearrowright\pi_{n-1}(F,e)$: \[A\simeq \pi_n(E,e)\simeq \Stab_{\pi_n(X,x)}(1_e).\]
\end{theorem}

Consequently, the classification of $(n-1)$-connected $n$-covering maps established in \eqref{theorem intro: pointed coverings corresponds to subgroups eq} restricts to an equivalence between normal $n$-coverings and normal 
sub-$1$-groups (\cref{cor: classification normal n-covering maps}):
   \begin{equation}\label{theorem intro: classification normal n-covering maps}
        \Cov^{\nor}_{\ast,n}(X)\simeq \Sub_1^{\nor}(\pi_n(X,x)).
    \end{equation}
\begin{remark}
    When $n>1$, every sub-$1$-group is automatically normal; therefore, in this range, every $(n-1)$-connected $n$-covering is normal.
\end{remark}

\medskip

\hspace{0.5cm}\textbf{Normalizers.}
We now turn to the deck transformations of an $n$-covering map that are not necessarily normal. The following theorem expresses the $n$-group of deck transformations as a quotient of a suitable normalizer. Classically, the normalizer $\N_G(H)$ is defined for a subgroup inclusion 
$H\leq G$. It applies to ordinary ($1$-)coverings, because the induced map 
$\pi_1(E,e)\to\pi_1(X,x)$ is injective. For $n$-coverings, however, the induced morphism of $n$-groups $\Pi_n(E,e)\longrightarrow \Pi_n(X,x)$ need not be a monomorphism. To address this, we extend the notion of normalizer $\N(f)$ of an arbitrary group homomorphism $f\colon H\to G$, introduced by Farjoun and Segev in \cite{FarjounSegev_NormalClosureAndInjectiveNormalizers}, to the setting of $\infty$-groups. Similar constructions appear in other homotopical contexts, for instance in Dwyer-Wilkerson \cite{DwyerWilkerson_HomotopyFixedPointsMethods}.
\begin{theorem}[{\cref{theorem: Deck p of a covering is a quotient of the normalizer}}]\label{theorem intro: Deck p of a covering is a quotient of the normalizer}
   Let $\E$ be an $\infty$-topos, let $(X,x)$ be a pointed connected object, let $p\colon E\to X$ be a connected pointed $n$-covering map in $\E_\ast$, and write $p_n\colon \Pi_n(E,e)\to\Pi_n(X,x)$ for the induced morphism of $n$-groups. Then there is an equivalence of $n$-group objects: 
    \begin{equation*}
        \Deck(p)\simeq \N(p_n)\sslash \Pi_n(E,e).
    \end{equation*}
\end{theorem}
 At first glance, one might hope that \cref{theorem intro: group deck is the quotient of fundamental n-group} follows formally from this \cref{theorem intro: Deck p of a covering is a quotient of the normalizer}: for $1$-groups, a subgroup inclusion $H\leq G$ is normal precisely when $\N_G(H)=G$. However, this property fails for arbitrary morphisms $f\colon H\to G$, even for ordinary groups. Instead, the logical relationship is reversed: we use \cref{theorem intro: group deck is the quotient of fundamental n-group} to prove that for an $n$-symmetric sub-$1$-group $f\colon \B^{n-1}A\longrightarrow G$, normality is equivalent to the canonical morphism $\N(f)\longrightarrow G$ being an equivalence (\cref{cor: characterization normality of n symmetric sub 1 groups with normalizers}). 

\medskip
\textbf{Examples.}
We illustrate these results in \cref{section: examples}, computing deck $n$-groups through the quotient formula $\Deck(p)\simeq\Pi_n(X,x)\sslash\B^{n-1}A$ of \cref{corollary: group deck of a normal covering is a quotient of fundamental group}. In homotopy types the formula unwinds to concrete computations, from the deck groups of universal covers to deck $3$-groups with non-trivial Postnikov invariants. Beyond spaces, the classification brings out a phenomenon with no counterpart there: the internal deck group and its \emph{external} image $\Gamma\Deck(p)=\E(\ast,\Deck(p))$ can differ sharply. For an Eilenberg--Mac Lane covering $\B^nA\to\B^nG$ with $n\geq2$, the internal deck group $\B^{n-1}(G/A)$ is connected: as an object of $\E$ it has a single component, so internally every deck transformation is trivial. Global sections need not preserve this connectedness, and what they record is exactly the cohomology of the base:
\[\pi_0\Gamma\Deck(p)\simeq H^{n-1}(\E;G/A).\]
For sheaves on a space $B$ this is singular cohomology $H^{n-1}(B;G/A)$; over the small étale site of a scheme $X$ it is étale cohomology $H^{n-1}_{\et}(X;G/A)$. For instance, for the inclusion $\mu_{k,X}\hookrightarrow\bG_{m,X}$ of the $k$-th roots of unity into the multiplicative group of $X$, the external deck $2$-group recovers the Picard group (\cref{example: picard deck group}): \[\pi_0\Gamma\Deck\Big(\B^2\mu_{k,X}\to\B^2\bG_{m,X}\big)\simeq H^1_{\et}(X;\bG_m)\simeq\operatorname{Pic}(X),\] so that every line bundle on $X$ appears as an external deck transformation, even though the internal deck group of this covering is connected. Finally, we adapt the theory to cohesive $\infty$-topoi, the main examples being the $\infty$-topoi of topological and smooth $\infty$-groupoids, into which the corresponding manifolds embed as $0$-truncated objects. Because a manifold is then $0$-truncated, the naive covering theory sees nothing; using the shape modality to single out the maps determined by their shape, we recover internally a theory of $1$-coverings of topological (resp.\ smooth) manifolds by manifolds. 

\medskip

\textbf{Organization of the paper.}
All results of this paper concerning $n$-coverings are collected in \cref{section : Coverings}. Their proofs rely in varying degrees on the machinery developed in the preceding sections. Some, such as \cref{theorem intro: representation theorem for n-covers}, are immediate corollaries of a general equivalence of slice $\infty$-categories established earlier (\cref{prop: characterization n-connected maps through pullback functor} \& \ref{cor: characterization n-connected maps through pullback functor in the pointed category}), which identifies the $n$-category of (pointed) $(n-1)$-truncated  maps over $X$ with that of objects over $\B\Pi_n(X,x)$. Others depend on more involved ingredients: \cref{theorem intro :action of Deck is n-free}  builds on the theory of $\infty$-actions developed throughout \cref{section: infinty group actions}. The identification of the deck $n$-group with a quotient of $\Pi_n(X,x)$ depends crucially on the internal Yoneda embedding and the use of universes \cite{Rasekh_YonedaLemmaElementaryTopos}. We now outline the content of each section in turn.

\medskip

In \cref{section : preliminaries}, we recall higher categorical preliminaries. In particular, we review the Giraud-Rezk-Lurie axioms for an $\infty$-topos and their equivalent formulation in terms of descent and universes. We also investigate which properties remain valid in the pointed $\infty$-category associated to an $\infty$-topos, as this will be useful for the study of pointed covering maps and pointed $\infty$-actions.

\medskip

In \cref{section : truncations and connectivity}, we review definitions and basic properties concerning truncations and connectivity. While many of the results are known, we sometimes provide new proofs in order to avoid relying on a presentation $\E\simeq\Sh(\C)$. In particular, we construct a generalized form of $k$-invariants which will be crucial in the proof of \eqref{theorem intro: classification normal n-covering maps}. The results of this section are used extensively throughout the paper; notably, \cref{theorem intro: representation theorem for n-covers} follows as a simple corollary.

\medskip

In \cref{section : group objects}, we review $\infty$-group objects and their characterization via pointed connected objects. Restricting to $n$-groups, we define $k$-symmetric $n$-groups and an appropriate notion of subobjects for these structures. We introduce normality and prove a higher analogue of the fact that subgroups of abelian groups are normal. We define homotopy groups and fundamental $n$-groups, which we use to characterize connectivity. Finally, we relate (normal) $n$-symmetric sub-$1$-groups of the fundamental $n$-group $\Pi_n(X)$ to (normal) sub-$1$-groups of the homotopy group $\pi_n(X)$, which will be used in the classification of $(n-1)$-connected (normal) $n$-covering maps, see \eqref{theorem intro: pointed coverings corresponds to subgroups eq} \& \eqref{theorem intro: classification normal n-covering maps}.

\medskip

In \cref{section: infinty group actions}, we study $\infty$-group actions. We analyze the behavior of an $\infty$-action $G\curvearrowright X$ under the $n$-truncation functor $\tau_n\colon\E\to\E$ and under the homotopy group functors $\pi_k\colon\E_\ast\to\E$. We establish properties of stabilizers and define the notions of $n$-transitive, $n$-free, and $n$-faithful $\infty$-actions, together with suitable characterizations (\cref{theorem: free transitive regular action characterization}). We then prove the crucial \cref{Theorem: action of a product restricted to H is free}, from which we deduce \cref{theorem intro :action of Deck is n-free}. Finally, we study fixed points, on which the proof of \cref{theorem intro: deck transformations as Omega X-equivariant autoequivalences} relies.

\medskip

In \cref{section : Coverings}, we assemble the results of the previous sections and translate them into the language of $n$-covering maps. This section contains all the \crefrange{theorem intro: representation theorem for n-covers}{theorem intro: Deck p of a covering is a quotient of the normalizer}.

\medskip

Finally, in \cref{section: examples}, we compute the deck $n$-groups of normal $n$-coverings in a range of $\infty$-topoi. We first record the universal and Eilenberg--Mac Lane coverings, together with a non-example in the $\infty$-topos of parametrized spectra. We then specialize to homotopy types, the arrow $\infty$-topos, sheaves on a topological space and the étale $\infty$-topos of a scheme, and to cohesive $\infty$-topoi.
\medskip

\textbf{Relation to other work.}
For $n=1$, this work contributes to the broader program of internalizing classical algebraic topology within abstract homotopical settings. Such developments have been pursued extensively in the synthetic framework of HoTT, for example \cite{hottbook, Christensen_HurewiczInHoTT, BuchholtzHou_CellularCohomologyInHoTT,Warn_EMSpacesInHoTT, HouShulman_SvKinHoTT, QuirinTabareau_LawvereTierneySheafificationInHoTT}, as well as in the context of $\infty$-topoi and higher category theory \cite{Lavenir_HiltonMilnorsTheoremInftyTopoi, ABFJ_GeneralizedBlakersMassyTheorem, NSS_PrincipalBundlesGeneralTheory, LurieHTT, DevalapurkarHaine_James/HiltonMilnorSplittings}. A detailed review of the literature specifically concerning covering spaces is provided in the following paragraphs.

\medskip

The theory developed here is related to earlier work in both strict higher-categorical settings and in Homotopy Type Theory. 

\medskip

Prior to the systematic use of $\infty$-categories, Roberts \cite{Roberts_2CoveringSpaces} studied $2$-covering spaces in a strict context. His definition consists of a locally trivial functor $E \to X$ of topological groupoids, where $X$ is a topological space, so that the fibers are discrete groupoids modeling $1$-types. One of his main results identifies the induced morphism of $2$-groups $\Pi_2(E)\to\Pi_2(X)$ as a sub-$2$-group, i.e. a $0$-truncated morphism. Our notion of sub-$n$-group is inspired by this perspective.

\medskip

Several aspects of the $n=1$ case have also been developed in Homotopy Type Theory. Hou and Harper established an internal form of \cref{theorem intro: representation theorem for n-covers} in \cite{HouHarper_CoveringSpacesInHoTT}. The classification of pointed coverings as subgroups (\cref{theorem intro: pointed coverings corresponds to subgroups}) was obtained internally for $n=1$ by Wemmenhove, Manea, and Portegies \cite{WemmenhoveManeaPortegies_ClassificationCoveringSpacesAndChangeOfBasePoint}, using classical lifting properties, and independently by Buchholtz, van Doorn, and Rijke \cite{HigherGroupsInHoTT_BuchholtzVanDoornRijke} via the identification of groups with pointed connected types. The latter authors also prove a special case of \cref{theorem intro: group deck is the quotient of fundamental n-group}, corresponding to the trivial subgroup $1 \leq \pi_n(X,x)$. Lastly Mimram and Oleon show that the universal $n$-cover is initial among pointed $n$-covering maps \cite{MimramOleon_ClassifyingCoveringHoTT}.

\medskip

The theory of $\infty$-groups has been investigated both from an external categorical perspective (for example \cite{SylowTheoremsForHigherGroups_PrasmaSchlank}) and within the internal framework of homotopy type theory \cite{BuchholtzRijke_LESofHomotopynGroups, HigherGroupsInHoTT_BuchholtzVanDoornRijke}. Our development of $\infty$-group theory within an $\infty$-topos aims to further advance this context, specifically by addressing the subtle notion of normality, as discussed below.

\medskip

The notion of normality in the context of $\infty$-groups behaves quite differently than for ordinary groups. For instance, it was pointed out in \cite{FarjounSegev_CrossedModulesasHomotopyNormalMaps}, where the notion was first defined for $1$-groups, that a $2$-group structure on the quotient $G\sslash A$ associated to a normal map $f\colon A\to G$ need not be unique. This phenomenon persists for morphisms of $\infty$-groups, as defined in \cite{Prasma_HomotopyNormalMaps}. The study of normality was further extended to $\infty$-topoi in \cite{BeardsleyFoxHigherGroupsHigherNormality}, where they point out another significant difference: the cofiber of a normal map $f \colon A \to G$ in the category of $\infty$-groups may not yield a quotient group.
Our work explores how these properties are refined when restricting to $n$-groups. We show that, under certain conditions, classical behaviors are recovered: specifically, we prove that certain $n$-group quotients are unique (\cref{cor: quotient of n-symmetric sub-1-groups are unique}) and that others may be obtained via a truncated pushout construction (\cref{prop: k-symmetric groups are k+1-symmetric for large k}). Furthermore, while a sub-$1$-group inclusion $f \colon A \hookrightarrow G$ is normal if and only if the canonical map $\N(f) \to G$ is an equivalence, the case for arbitrary morphisms is more subtle. In the classical setting, $f \colon A \to G$ is normal if and only if $\N(f) \to G$ admits a section \cite{FarjounSegev_NormalClosureAndInjectiveNormalizers}. While it remains open whether this characterization holds for $\infty$-groups, we prove that for $n$-symmetric sub-$1$-groups $\B^{n-1}A \to G$, normality is indeed equivalent to $\N(f) \to G$ being an equivalence (\cref{cor: characterization normality of n symmetric sub 1 groups with normalizers}). Notably, our approach inverts the classical relationship between group theory and topology. While $1$-group theory is traditionally employed as a tool to study covering spaces, we instead use the theory of $n$-covering maps to establish the properties of $n$-groups described above.

\medskip
Galois theory for $1$-topoi originated in the study of locally constant objects, a concept pioneered by Grothendieck for Galois categories in \cite{SGA1}, and extended to $1$-topoi \cite{JoyalTierney_AnExtensionOfGaloisTheoryOfGrothendeick}, $2$-categories \cite{HigherMonodromy_PoleselloWaschkies}, and $\infty$-topoi \cite{LurieHigherAlgebra}. In an $\infty$-topos $\mathcal{E}$, the notion of locally constant objects naturally extends to morphisms by passing to the slice $\infty$-topos $\slice{\E}{X}$. In his development of higher Galois theory, Hoyois \cite{Hoyois_HigherGaloisTheory} classifies $n$-coverings that are locally constant over $X$, assuming $\slice{\E}{X}$ is locally $(n-1)$-connected, via representations of the \textit{external} fundamental $n$-groupoid $\Pi_n(X) \in \mathcal{S}^{\sleq{n}}$ in $(n-1)$-truncated homotopy types. The two frameworks are related but not directly comparable in general: where Hoyois works externally and restricts to locally constant coverings under a local connectivity hypothesis, \cref{theorem intro: representation theorem for n-covers} classifies \emph{all} $(n-1)$-truncated maps via the \emph{internal} fundamental $n$-group $\Pi_n(X,x)\in\Grp_n(\E)$, with no connectivity hypothesis on the topos. In the cohesive setting \cite{Schreiber_DifferentialCohomologyCohesiveInfinityTopos}, the two can be directly compared: locally constant $n$-coverings correspond to the subclass classified by the subuniverse $\mathrm{Disc}((\mathcal{S}^{\sleq{n-1}}_\kappa)^\simeq)\hookrightarrow\mathcal{U}_\kappa^{\sleq{n-1}}$ of cohesively discrete objects, where $\mathrm{Disc}$ is the left adjoint of the terminal geometric morphism, and the shape adjunction recovers Hoyois's external classification from ours.

\medskip

\textbf{Use of AI assistance.}
The writing of this paper was assisted by Claude, a large language model made by Anthropic. The AI contributed in two ways. Editorially, it suggested trimming certain passages for conciseness, and assisted with English phrasing and overall smoothness of the exposition. Mathematically, it suggested the examples of \cref{subsection: example: sheaf topoi}: $n$-coverings in the $\infty$-topoi of sheaves over a topological space $B$ and over a scheme $X$. All AI-produced suggestions were verified, and where necessary corrected, by the author.

\medskip

\textbf{Acknowledgments.}
We warmly thank Jérôme Scherer and Nima Rasekh for their helpful feedback on earlier versions of this manuscript.
\setcounter{theorem}{0}
\renewcommand{\thetheorem}{\thesection.\arabic{theorem}}
 \section{Higher categorical preliminaries}\label{section : preliminaries}
\subsection{Conventions}
\paragraph{\textbf{$\infty$-Categorical}}
Throughout this paper, we employ the language of higher category theory, considering only homotopy-invariant constructions. By an $\infty$-category, we mean an $(\infty, 1)$-category. Our arguments are essentially model-independent, as they rely solely on the manipulation of $\infty$-limits, $\infty$-colimits, and the internal properties supported by the definition of an $\infty$-topos. Consequently, a reader who is not strictly familiar with the technical foundations of $\infty$-categories should be able to follow the majority of this text by interpreting our constructions in the sense of homotopy limits and colimits. To be precise, however, one could interpret all our results within the framework of \emph{quasi-categories} as developed by Joyal \cite{Joyal_QuasiCategoriesAndKanComplexes} and Lurie \cite{LurieHTT}, which provides a complete and rigorous theory of $\infty$-categories. Other references are \cite{Cisinski_HigherCategoriesAndHomotopicalAlgebra, RiehlVerity_ElementsOfInftyCategoryTheory}.

\medskip

\paragraph{\textbf{Notational}}
We adhere to the following terminological and notational conventions. We use the word \emph{space} to refer to an $\infty$-groupoid or abstract homotopy type, and denote the $\infty$-category of spaces by $\mathcal{S}$. An arrow $f\colon X\to Y$ is called an \emph{equivalence} if it has a two-sided inverse $g\colon Y\to X$ such that $gf\simeq 1_X$ and $fg\simeq 1_Y$. We write $X=Y$ to mean that $X$ is canonically equivalent to $Y$, when a homotopy equivalence is clear or explicit (usually defined through a universal property). The adjective \emph{unique} is always understood to mean unique up to a contractible space of choices. A diagram \emph{commutes} if there are coherent homotopies filling all of its simplices, specified as part of the data. We follow the standard convention of leaving these homotopies implicit when they are uniquely determined (as in the case of compositions) or provided by a universal property (as in the case of (co)limits). All categorical constructions, including limits, colimits, mapping objects, and adjunctions, are to be understood in their $\infty$-categorical (homotopical) sense. When necessary, we implicitly view ordinary 1-categories as $\infty$-categories via the nerve embedding. An initial object is denoted by $\emptyset$, while $\ast$ denotes terminal objects. We write $(X,x)\in\C_\ast$ for an object $X\in\C$ endowed with a global element $x\colon\ast\to X$. For any two objects $X$ and $Y$ in an $\infty$-category $\C$, we write $\C(X,Y)\in\mathcal{S}$ for the space of maps between $X$ and $Y$. If $\C$ is Cartesian closed, we write $\Map_\C(X,Y)\in\C$ (dropping the subscript when the category is clear) for the internal hom object. We write $\Eq_\C(X,Y)\in\C$ for the subobject of equivalences, which represents the functor 
\[T \mapsto \slice{\C}{T}^\simeq(X\times T,Y\times T)\subseteq \slice{\C}{T}(X\times T, Y\times T)\] 
(see \cite[§2.2]{Vergura_LocalizationInTopos} for an explicit construction). Working in the slice over an object $Z\in\C$, we use the notation $\Eq_{/Z}(X,Y)\subseteq\Map_{/Z}(X,Y)$.
A commutative square 
\begin{equation}\label{eq: Cartesian square}
  \begin{tikzcd}
	P & X \\
	Y & Z
	\arrow[from=1-1, to=1-2]
	\arrow[from=1-1, to=2-1]
	\arrow[from=1-2, to=2-2]
	\arrow[from=2-1, to=2-2, "f"]
\end{tikzcd}
\end{equation}
is \textit{Cartesian} if $P$ is the pullback of the diagram $Y\to Z\leftarrow X$, and \textit{coCartesian} if $Z$ is the pushout of $Y\leftarrow P\rightarrow X$. 

We record a few basic categorical properties for future reference:
\begin{enumerate}
    \item\label{item: Hom C vs slice} For $Y,Z\in \C$ and $X\in\slice{\C}{Z}$ we have an equivalence of homotopy types $\C(X,Y)\simeq\slice{\C}{Z}(X,Y\times Z)$.
    \item \label{item: double slice is slice} For $f\colon A\to B$ there is a canonical equivalence of $\infty$-categories $\slice{\left(\slice{\C}{B}\right)}{f}\simeq\slice{\C}{A}$.
    \item\label{item: pasting law} Consider the following commutative diagram:
    \[\begin{tikzcd}
    {P} \arrow[r] \arrow[d]  & Y\arrow[dr, phantom, "\lrcorner"{anchor=center, pos=0.125}] \arrow[r] \arrow[d] & T \arrow[d] \\
    X \arrow[r, "i"'] & Z \arrow[r] & W
\end{tikzcd},\]
in which the right square is Cartesian. Then the left square is Cartesian if and only if the total rectangle is Cartesian. This is known as the \textit{pasting law for pullbacks}.
    \item\label{item: pullback of a product with X} For any map $X\to A$ the square \eqref{eq: Cartesian square} is Cartesian if and only if the following induced square is Cartesian:
\[\begin{tikzcd}
	P & X \\
	{Y\times A} & {Z\times A}
	\arrow[from=1-1, to=1-2]
	\arrow[from=1-1, to=2-1]
	\arrow[from=1-2, to=2-2]
	\arrow[from=2-1, to=2-2, "f\times A"]
\end{tikzcd}.\]
\end{enumerate}
\subsection{$\infty$-Topoi}\label{section: preliminaries subsection: higher topoi}
The following definition, which generalizes the Giraud axioms for $1$-topoi, is a characterization due to Toën and Vezzosi for model topoi \cite[Thm.~4.9.2]{Toen&Vezzosi_HomotopicalAG_ToposTheory}, and to Lurie for $\infty$-topoi \cite[Thm~6.1.0.6]{LurieHTT}. Rather than adopting the \textit{external} viewpoint, which views an $\infty$-topos as a left exact localization of a presheaf $\infty$-category, we take these Giraud-style axioms as our primary definition. We will subsequently introduce two related \textit{internal} characterizations, based on the notions of descent and universes, which will be central to our approach.
\begin{definition}\label{def: infinity topos}
    An $\infty$-topos is an $\infty$-category $\E$ satisfying the following conditions:
    \begin{enumerate}
        \item the $\infty$-category $\E$ is presentable;
        \item colimits are universal;
        \item coproducts are disjoint;
        \item every groupoid object is effective.
    \end{enumerate}
\end{definition}
Presentability ensures that $\E$ admits all small limits and colimits and provides access to the adjoint functor theorems. For instance, while truncation functors $\tau_n\colon \E\to\E$ are often constructed using the fact that the inclusion $\E^{\sleq{n}}\to\E$ preserves limits \cite{LurieHTT}, they can also be defined without the presentability assumption \cite{Rasekh_ElementaryApproachTruncations}.

\medskip

\textbf{Universality of Colimits.}
Condition (2) states that for every morphism $f\colon X\to Y$, the pullback functor $f^\ast\colon \slice{\E}{Y} \to \slice{\E}{X}$ preserves small colimits. When $\E$ is presentable, this is equivalent to $\E$ being locally Cartesian closed. In that case we write $\prod_f\colon\slice{\E}{X}\to\slice{\E}{Y}$ for the right adjoint to $f^\ast$, called the \textit{dependent product}. We simply write $\prod_X\colon\slice{\E}{X}\to\E$ if $Y\simeq\ast$ and $\prod_XA$ is then the \textit{object of sections} of $A\to X$.

\medskip
 
\textbf{Disjointness of Coproducts.}
This condition stipulates that for any objects $A, B \in \E$, the following square is Cartesian:
\[\begin{tikzcd}
    \emptyset \arrow[r] \arrow[d] \arrow[dr, phantom, "\lrcorner"{anchor=center, pos=0.125}] & B \arrow[d, "{\iota_B}"] \\
    A \arrow[r, "{\iota_A}"'] & {A \amalg B}
\end{tikzcd}.\]
\textbf{Effectiveness of Groupoids.}
We adopt the following characterization of groupoid objects from Bunk \cite[Lemma~A.6.3]{SeverinBunk_ooBundlesAndSmthStringModels}.
\begin{definition}\label{def: groupoid object}
    Let $\C$ be an $\infty$-category with pullbacks. 
    \begin{enumerate}
    \item A \textit{category object} in $\C$ is a simplicial object $\G:\Delta^\op\to\C$ such that the morphism \[\G_n\xrightarrow{\simeq} \G_1\times_{\G_0}\cdots \times_{\G_0}\G_1\]
    induced by the spine decomposition $[n]\cong [1]\amalg_{[0]}\ldots\amalg_{[0]}[1]$ is an equivalence.
        \item A \textit{groupoid object} in $\C$ is a category object $\G\colon \Delta^\op\to \C$ such that the following commutative square is Cartesian
\[\begin{tikzcd}
    {\G_2} \arrow[r, "{d_2}"] \arrow[d, "{d_1}"'] \arrow[dr, phantom, "\lrcorner"{anchor=center, pos=0.125}] & {\G_1} \arrow[d, "{d_1}"] \\
    {\G_1} \arrow[r, "{d_1}"'] & {\G_0}
\end{tikzcd}.\]
        \item If $\C$ admits $\Delta^\op$-indexed colimits, the colimit of $\G$ is its \textit{realization}, denoted by $\vert\G\vert$, with corresponding quotient map $\pi\colon \G_0\to\vert\G\vert$.
\end{enumerate}
The full subcategory of $\Fun(\Delta^\op, \C)$ consisting of groupoid objects is denoted $\Gpd(\C)$.
\end{definition}
   
    \begin{definition}
        Let $\C$ be an $\infty$-category with pullbacks and realizations, and let $f\colon A\to B$ be a morphism in $\C$.
        \begin{enumerate}
            \item The \textit{\v{C}ech nerve} of $f$, denoted $\check{C}(f)\colon \Delta^\op_+\to\C$, is the right Kan extension of $f\colon \Delta^\op_{+,\leq0}\to\C$ along $\Delta^\op_{+,\leq0}\subseteq\Delta^\op_+$:
\[\begin{tikzcd}
	{\Delta^\op_{+,\leq0}} & \C \\
	{\Delta^\op_{+}}
	\arrow["f", from=1-1, to=1-2]
	\arrow[hook, from=1-1, to=2-1]
	\arrow["{\check{C}(f)}"', dashed, from=2-1, to=1-2]
\end{tikzcd}.\] Informally it is given by 
\[\begin{tikzcd}
	\dots & {A\times_BA\times_BA} & {A\times_BA} & A & {B.}
	\arrow[shift right, from=1-1, to=1-2]
	\arrow[shift left, from=1-1, to=1-2]
	\arrow[shift right=3, from=1-1, to=1-2]
	\arrow[shift left=3, from=1-1, to=1-2]
	\arrow[from=1-2, to=1-1]
	\arrow[shift right=2, from=1-2, to=1-1]
	\arrow[shift left=2, from=1-2, to=1-1]
	\arrow[from=1-2, to=1-3]
	\arrow[shift right=2, from=1-2, to=1-3]
	\arrow[shift left=2, from=1-2, to=1-3]
	\arrow[shift right, from=1-3, to=1-2]
	\arrow[shift left, from=1-3, to=1-2]
	\arrow[shift right, from=1-3, to=1-4]
	\arrow[shift left, from=1-3, to=1-4]
	\arrow[from=1-4, to=1-3]
	\arrow["f", from=1-4, to=1-5]
\end{tikzcd}\]
    \item A groupoid object $\G$ in $\C$ is \textit{effective} if it is equivalent to the restriction to $\Delta^\op$ of the \v{C}ech nerve of its realization, i.e. if $\G\simeq \check{C}(\pi)\vert_{\Delta^\op}$ in $\Gpd(\C)$ where $\pi\colon \G_0\to\vert\G\vert$. 
    \item The realization map $\pi\colon \G_0\twoheadrightarrow\vert\G\vert$ of an effective groupoid is called an \textit{effective epimorphism}.
        \end{enumerate}
    \end{definition}
    \begin{remark}\label{remark: definition groupoid object}
        Combining the category object property with the \v{C}ech nerve condition shows that a groupoid object is effective if and only if the following square is Cartesian (\cite[Prop.~6.1.2.11]{LurieHTT}):
\[\begin{tikzcd}
	{\G_1} \arrow[r, "d_1"] \arrow[d, "d_0"'] \arrow[dr, phantom, "\lrcorner"{anchor=center, pos=0.125}] & {\G_0} \arrow[d] \\
	{\G_0} \arrow[r] & {\vert\G\vert}
\end{tikzcd}\]
    \end{remark}
    Effectiveness of groupoid objects is equivalent to the \v{C}ech nerve functor defining an equivalence of $\infty$-categories $\Eff(\C) \xrightarrow{\simeq} \Gpd(\C)$. Effective epimorphisms, denoted by a double-headed arrow $\twoheadrightarrow$, satisfy several pleasant properties in a semitopos \cite[§6.2.3]{LurieHTT} which we shall use throughout this work.
    \begin{definition}
        A presentable $\infty$-category $\E$ with universal colimits is a \textit{semitopos} if for every morphism $f\colon A\to B$ the underlying groupoid of its \v{C}ech nerve is effective.
    \end{definition}
    For example, the following proposition generalizes the classical fact that a commutative square is Cartesian if and only if the homotopy fibers are equivalent. It is a slightly more general version than \cite[Prop.~I.5.7]{Samuel_Thesis}, but the same proof can be carried out.
\begin{prop}\label{prop: square is Cartesian if fibers are equivalent}
        Consider a commutative diagram in a semitopos $\E$:
\[\begin{tikzcd}
    {F_f} \arrow[r, two heads] \arrow[d] \arrow[dr, phantom, "\lrcorner"{anchor=center, pos=0.125}] & A \arrow[r] \arrow[d] & B \arrow[d] \\
    X \arrow[r, "i"', two heads] & C \arrow[r] & D
\end{tikzcd},\]
    where the left square and the total rectangle are Cartesian, and $i$ is an effective epimorphism. Then the right square is Cartesian.
\end{prop}
We end this subsection by noting that the slice $\infty$-categories of $\infty$-topoi are themselves $\infty$-topoi. This is sometimes called the \textit{fundamental theorem of $\infty$-topos theory:}
\begin{theorem}
    Let $\E$ be an $\infty$-topos and let $X\in\E$ be an object. Then the slice $\infty$-category $\slice{\E}{X}$ is an $\infty$-topos.
\end{theorem}

\subsection{Descent}
The notion of descent is central in the theory of $\infty$-topoi, providing one of their fundamental characterizations. Let $\I$ be a small $\infty$-category and $\C$ an $\infty$-category with $\I$-shaped colimits. A natural transformation $\alpha\colon A\to B$ between functors $A,B\colon \I\to\C$ is \textit{Cartesian} if for every morphism $i\to j$ in $\I$, the commutative naturality square 
\[\begin{tikzcd}
    {A_i} \arrow[r] \arrow[d, "{\alpha_i}"'] & {A_j} \arrow[d, "{\alpha_j}"] \\
    {B_i} \arrow[r] & {B_j}
\end{tikzcd}\]
is Cartesian. We write $\Cart\left(\Fun(\I,\C)\right)\subset \Fun(\I,\C)$ for the wide subcategory on the Cartesian transformations. We say that $\C$ satisfies \textit{descent along $\I$-shaped colimits} if, for every such Cartesian natural transformation $\alpha\colon A\to B$, the following square is Cartesian for every object $i\in \I$:
\[\begin{tikzcd}
    {A_i} \arrow[r] \arrow[d, "{\alpha_i}"'] \arrow[dr, phantom, "\lrcorner"{anchor=center, pos=0.125}, near start] & {\colim_{\mathcal{I}}A} \arrow[d, "{\colim_{\mathcal{I}}\alpha}"'] \\
    {B_i} \arrow[r] & {\colim_{\mathcal{I}}B}
\end{tikzcd}.\]
Equivalently, $\C$ satisfies descent along $\I$-shaped colimits if $\Cart\left(\Fun([1],\C)\right)$ has $\I$-shaped colimits and the inclusion $\Cart\left(\Fun([1],\C)\right)\subset \Fun([1],\C)$ preserves them.
\begin{remark}
    Some authors use the terminology \textit{effective} for what we call descent, saying that an $\infty$-category satisfies descent if it has both universal and effective colimits. See, for example, \cite{ABFJLeftExactLocalizationsI}. 
\end{remark}
\begin{theorem}
    Let $\E$ be a presentable $\infty$-category. Then $\E$ is an $\infty$-topos if and only if the following conditions are satisfied:
    \begin{enumerate}
        \item Colimits in $\E$ are universal;
        \item $\E$ satisfies descent along all small colimits.
    \end{enumerate}
\end{theorem}

Another point of view on descent is through the lens of universes, which we explore now.

\subsection{Universes}\label{section: preliminaries subsection: universes}
Universes play a central role in this paper, formalizing the correspondence between maps $E\to X$ and $X$-indexed families of objects $\corner{E}\colon X\to \U$. These ideas trace back to Voevodsky's formulation of the univalence axiom in Homotopy Type Theory \cite{hottbook}, together with the simplicial model as constructed in \cite{Kapulkin/Lumsdaine_UnivalentModelHoTT}. We follow the categorical interpretation of univalence developed by Gepner and Kock \cite{GepnerKock_UnivalenceInLCCC}, and refer to Rasekh \cite{Rasekh_UnivalenceHigherCategoryTheory} for universes of truncated maps. For the rest of this section, let $\C$ be a presentable locally Cartesian closed $\infty$-category.

\medskip

Given a map $p\colon \U_\ast\to \U$ and an object $X$, there exists a functor of $\infty$-groupoids $\C(X,\U)\to \slice{\C}{X}^\simeq$ given informally by $(\corner{E}\colon X\to \U)\mapsto (\corner{E}^\ast p\colon E \xrightarrow{} X)$, where 
\[\begin{tikzcd}
    {E} \arrow[r] \arrow[d, "{\corner{E}^\ast p}"'] \arrow[dr, phantom, "\lrcorner"{anchor=center, pos=0.125}] & {\U_\ast} \arrow[d, "p"] \\
    X \arrow[r, "\corner{E}"'] & \U
\end{tikzcd}\]
is a Cartesian square.

\begin{definition}[{\cite[Prop. 3.8]{GepnerKock_UnivalenceInLCCC}}]\label{definition: univalent universe} 
    A map $p\colon \U_\ast\to\U$ in a presentable locally Cartesian closed $\infty$-category $\C$ is a \textit{univalent universe} if the functor 
    \begin{equation*}
        (-)^\ast(p)\colon \C(X,\U)\to \slice{\C}{X}^\simeq
    \end{equation*} 
    is fully faithful for all $X\in\C$. 
    In this case, we say that $p$ \textit{classifies} the class $\F$ defined by the essential image of these functors.
\end{definition}
\begin{remark}
    Via the unstraightening equivalence $\Fun(\C^\op,\Sp)\xrightarrow{\simeq}\mathscr{RF}\mathrm{ib}(\C)$, the definition of a univalent universe can be rephrased: $p$ is univalent if and only if the induced functor of right fibrations $\slice{\C}{\U}\to \Cart(\Fun([1],\C))$ is fully faithful. Another similar characterization is that $p$ is univalent if and only if it is a subterminal object of $\Cart(\Fun([1],\C))$ \cite[Thm.~6.1.3.9]{LurieHTT}. This latter condition is the precise bridge between descent and the univalence condition.
\end{remark}

\textbf{The Univalence Axiom.} 
Intuitively, \cref{definition: univalent universe} means that families $\corner{E}\colon X\to\U$ correspond to objects $E\to X$, and paths $\corner{E}\simeq \corner{E'}$ in $\U$ correspond exactly to equivalences $E \simeq E'$ over $X$, and so on. This translates to the geometric internal property that the diagonal $\U \to \U\times\U$ classifies the object of equivalences. That is, $p$ is a univalent universe if and only if the canonical map over $\U\times\U$
    \begin{equation*}
        \begin{tikzcd}[ampersand replacement=\&]
        \U \&\& {\Eq_{/\U\times\U}(\U_\ast\times\U,\U\times\U_\ast)} \\
        \& {\U\times\U}
        \arrow[from=1-1, to=1-3]
        \arrow[from=1-1, to=2-2]
        \arrow[from=1-3, to=2-2]
    \end{tikzcd}
       \end{equation*}
(which represents the functor $(X\xrightarrow{\corner{E}} \U)\mapsto 1_E\in \Eq_{/X}(E,E)$) is an equivalence. This intuitive characterization makes the internal Univalence Axiom essentially a tautology. It stipulates that paths in a universe $\U$ between objects $A,B\in\C$ correspond to equivalences $A\simeq B$. 
\begin{definition}\label{definition: path objects}
    Let $\E$ be an $\infty$-topos, $X$ an object, and $a,b\colon T\to X$ generalized points. The \textit{path object} $X(a,b)$ in $\slice{\E}{T}$ is the generalized fiber
\[\begin{tikzcd}
    {X(a,b)} \arrow[r] \arrow[d] \arrow[dr, phantom, "\lrcorner"{anchor=center, pos=0.125}] & X \arrow[d, "\Delta"] \\
    T \arrow[r, "{(a,b)}"'] & {X \times X}
\end{tikzcd}.\]
\end{definition}
Note that $X(a,a)\simeq\Omega^{/T}_aX$ (\cref{definition: suspension and loop space}) by a Fubini argument.

\begin{prop}[Univalence Axiom]\label{Univalence Axiom in a topos}
    Let $A,B$ be objects over $T$ classified by maps $\corner{A}, \corner{B} \colon T \to \U$ into a univalent universe. Then we have a Cartesian square in $\E$:
\[\begin{tikzcd}
    {\Eq_{/T}(A,B)} \arrow[r] \arrow[d] \arrow[dr, phantom, "\lrcorner"{anchor=center, pos=0.125}] & \U \arrow[d] \\
    T \arrow[r, "{(\corner{A},\corner{B})}"'] & {\U \times \U}
\end{tikzcd}.\]
    In particular, $\Aut(A)\simeq\Omega_{\corner{A}} \U$.
\end{prop}
\begin{example}\label{example: universe in spaces}
    In the large $\infty$-category $\Hat{\Sp}$ of small homotopy types, the universe is given by the forgetful functor $\Sp^\simeq_\ast\to\Sp^\simeq$. By definition a loop at $A$ is precisely an autoequivalence, which illustrates \cref{Univalence Axiom in a topos}. Thus, the connected components of $\U$ are of the form $\B \Aut(A)$ for some $A\in\Sp$. This was proven classically in \cite{Allaud_ClassificationOfFiberSpaces}, and arises here from the internal univalence axiom (see \cref{ex: BAut(F) as the image in the universe}).
\end{example}
\medskip

\textbf{A Characterization of $\infty$-Topoi.}
Classes of maps $\F$ which arise from univalent maps are called \textit{bounded local classes}, and the poset of bounded local classes is in bijection with the poset of equivalence classes of univalent morphisms \cite[Thm.~3.9]{GepnerKock_UnivalenceInLCCC}. The poset structure is given by monomorphisms: if $p\colon \U_\ast\to\U$ is univalent and $q$ is a pullback of $p$ along $f$, then $q$ is univalent if and only if $f$ is a monomorphism \cite[Prop.~2.5]{Rasekh_UnivalenceHigherCategoryTheory}. The \say{bounded} condition ensures we do not need to restrict to small $\infty$-categories, stipulating that the class consists of relatively $\kappa$-compact morphisms for a fixed regular cardinal $\kappa$ \cite[Def.~6.1.6.4]{LurieHTT}. We generally leave this implicit and refer to such maps as \textit{small}. The \say{local} condition on the class $\F$ is a descent condition. Hence, having enough universes to classify bounded local classes becomes a defining feature of higher topoi. 
\begin{theorem}\label{Higher topoi have enough universes}
    Let $\E$ be a presentable $\infty$-category. Then $\E$ is an $\infty$-topos if and only if the following conditions are satisfied:
    \begin{enumerate}
        \item Colimits in $\E$ are universal;
        \item For all sufficiently large cardinals $\kappa$, there exists a classifying object $\U_\kappa\in \E$ for the class of relatively $\kappa$-compact morphisms.
    \end{enumerate}
\end{theorem}
\begin{convention}
    Since there are sufficient universes in an $\infty$-topos, we assume, by working in a larger universe, that each universe $\U$ we work with is closed under the constructions we undertake.
\end{convention}

\begin{example}[Truncated Maps]
    Our main focus is the bounded local class of small $n$-truncated maps for a fixed $-2\leq n<\infty$, as defined in \cref{subsection: truncations}. There exists a subuniverse $\U^{\sleq n}\hookrightarrow\U$ for this class, equipped with a truncation map $\tau_n\colon\U\to \U^{\sleq n}$ representing the truncation functor 
    \[\tau_n\colon\slice{\C}{X}^\simeq\to \left(\slice{\C}{X}^\simeq\right)^{\sleq n}.\]
    If $f\colon E\to X$ is classified by $X\to\U$, then the composition $X\to\U\to\U^{\sleq n}$ classifies $\tau_n^XE\to X$, the universal map with $n$-truncated fibers. See \cref{subsection: truncations} for more details about truncations.
\end{example}

\medskip

\textbf{Universes in Slice Categories.}
Some of our arguments involve universes in slice $\infty$-categories. The following property ensures that universes are stable under the product functor $-\times X\colon\C\to\slice{\C}{X}$. Other conditions on preservation or reflection of univalent universes include fully faithful functors preserving pullbacks \cite[Prop.~2.9]{Rasekh_UnivalenceHigherCategoryTheory}, or locally Cartesian closed localizations \cite[Thm.~3.12]{GepnerKock_UnivalenceInLCCC}.
\begin{prop}[{\cite[1.19]{Rasekh_TheoryElementaryHigherToposes}}]\label{lemma: universe in slice categories}
    Let $\C$ be a finitely complete $\infty$-category, and let $X \in \C$ be an object. Suppose $p \colon \U_\ast \to \U$ is a univalent universe classifying a class $\F$ of maps. Then $p \times X \colon \U_\ast \times X \to \U \times X$ is a univalent universe classifying $\F_X$, the class of triangles over $X$ whose underlying map in $\C$ belongs to $\F$.
\end{prop}

\medskip

\textbf{The Internal Yoneda Embedding.}
Let $X$ be an object of an $\infty$-topos $\E$, and let $\U$ be a universe that classifies the diagonal $\Delta\colon X\to X\times X$.
For any two generalized points $(a,b)\colon T\to X\times X$, the fiber of $\Delta$ at $(a,b)$ is the internal path object $X(a,b)$. Consequently, the classifying map of $\Delta$ is the \textit{internal path map} $X(-,-)\colon X\times X\to\U$.
\begin{definition}
    The \textit{Yoneda map} $\yo_X\colon X\to\U^X$ is the map adjoint to  $X(-,-)\colon X\times X\to\U$.
\end{definition}
When we regard $X$ as an internal $\infty$-groupoid and $\U$ as the internal $\infty$-groupoid of small internal $\infty$-groupoids, this map plays the role of the Yoneda embedding. That it is a monomorphism, the internal counterpart of fully faithfulness, was first shown by Rasekh \cite{Rasekh_YonedaLemmaElementaryTopos}. In \cite[Theorem~2.2]{Constantin_YonedaInLCCC}, we give a different proof in the context of finitely complete locally Cartesian closed $\infty$-categories.
\begin{theorem}[Yoneda Embedding]\label{Yoneda Embedding}
    Let $\E$ be an $\infty$-topos and $X$ an object of $\E$. The Yoneda map \[\yo_X\colon X\to\U^X\] is a monomorphism.
\end{theorem}
This fact is a key ingredient in the proof of \cref{theorem: Deck p is loop space of the base B}.

\subsection{Pointed $\infty$-categories}\label{subsection: pointed categories}
 Let $\C$ be an $\infty$-category with a terminal object. We denote the associated \textit{pointed} $\infty$-category by $\C_\ast = \C_{\ast/}$ and the forgetful projection by $U\colon \C_\ast \to \C$. In this subsection, we record the properties of $\infty$-topoi $\E$ that persist in their associated pointed $\infty$-categories. We use them in \cref{section: infinty group actions} to characterize pointed $\infty$-actions. Those are central in the study of pointed covering maps, as we shall see in \cref{Theorem: representation theorem for n-covers,corollary: pointed coverings correspond to subgroups}.
\begin{prop}\label{prop: pointed category has effective groupoids}
Let $\C$ be an $\infty$-category admitting a terminal object, pullbacks and realizations.
     A simplicial object $F:\Delta^\op\to\C_\ast$ is an (effective) groupoid in $\C_\ast$ if and only if $U F:\Delta^\op\to\C$ is an (effective) groupoid in~$\C$. Consequently, $U$ preserves and reflects effective epimorphisms.
\end{prop}
\begin{proof}
         The projection $U\colon\C_\ast\to\C$ preserves and reflects limits, so $F$ is a groupoid object if and only if $UF$ is so. Moreover $U$ preserves and reflects weakly contractible colimits \cite[Cor.~I.5.2]{Samuel_Thesis}, such as those indexed by $\Delta^\op$. This implies that $F$ is effective if and only if $UF$ is.
    \end{proof}
\begin{corollary}[{\cite[Remark~I.5.6]{Samuel_Thesis}}]\label{cor: groupoid object in pointed category are effective}
    If groupoid objects in $\C$ are effective, then so are groupoid objects in~$\C_\ast$:
    \[
    \Eff(\C_\ast)\xrightarrow[\simeq]{\check{C}}\Gpd(\C_\ast).
    \]
\end{corollary}
\begin{prop}[{\cite[Corollary~5.4.5.16]{LurieHTT}}]
    Suppose that $\C$ is a presentable $\infty$-category. Then $\C_\ast$ is presentable.
\end{prop}
\begin{prop}
    Let $\I$ be a weakly contractible $\infty$-category and suppose that $\C$ has universal $\I$-shaped colimits. Then $\C_\ast$ has universal $\I$-shaped colimits.
\end{prop}
\begin{proof}
    Pullbacks and colimits in a slice $\infty$-category $\slice{\C}{Y}$ are computed in $\C$, and pointed $\I$-shaped colimits in $\C$ coincide with their unpointed counterparts \cite[Cor.~I.5.2]{Samuel_Thesis}. Thus, for $f:X\to Y$ and a diagram $A:\I\to\slice{(\C_\ast)}{Y}$, we have: \[f^\ast(\colim_{\ast,\I}A_i)\simeq f^\ast (\colim_{\I} A_i)\simeq\colim_\I f^\ast(A_i)\simeq\colim_{\ast,\I}f^\ast(A_i).\]
\end{proof}
\begin{prop}[{\cite[Prop.~I.5.5]{Samuel_Thesis}}]
    Let $\I$ be a weakly contractible $\infty$-category and suppose that $\C$ has descent along $\I$-shaped colimits. Then $\C_\ast$ has descent along $\I$-shaped colimits.
\end{prop}
Let us pack the previous results into a single one for future reference.
\begin{corollary}\label{cor: pointed topos has nice properties}
    Let $\E$ be an $\infty$-topos and $\I$ a weakly contractible $\infty$-category. Then the pointed $\infty$-category $\E_\ast$ is presentable, has effective groupoids, has universal $\I$-shaped colimits, and satisfies descent along $\I$-shaped colimits.
\end{corollary}
\section{Truncation and connectedness}\label{section : truncations and connectivity}
In this section we recall and further explore the notions of $n$-truncated and $n$-connected maps. These concepts have been extensively studied across higher category theory: in the context of presentable $\infty$-categories by Lurie \cite{LurieHTT}, within higher topoi by Rezk and Lurie \cite{Rezk_HomotopyTopos, LurieHTT}, and more recently in a non-presentable setting by Rasekh \cite{Rasekh_ElementaryApproachTruncations}. While we will make use of these properties in the context of higher topoi, Lurie's proofs often depend on choosing a presentation $\E \simeq \mathrm{Sh}(\C,\Sp)$. To ensure our arguments remain independent of such models, we primarily refer to Rasekh's work, which develops the theory for elementary $\infty$-topoi, hence from first principles. These results will be used heavily during the remainder of this article, particularly for studying $k$-symmetric $n$-groups and their $\infty$-actions in \cref{section : group objects} and \ref{section: infinty group actions}, as well as for studying $n$-covering maps in \cref{section : Coverings}.

\subsection{Truncations}\label{subsection: truncations}
In this section, let $\C$ be a presentable locally Cartesian closed $\infty$-category.
\begin{definition}\label{definition: suspension and loop space}
\begin{enumerate}
    \item The \textit{suspension} $\Sigma X$ of an object $X\in \C$ is the pushout
\[\begin{tikzcd}
    X \arrow[r] \arrow[d] & \ast \arrow[d] \\
    \ast \arrow[r] & {\Sigma X} \arrow[ul, phantom, "\lrcorner"{anchor=center, pos=0.125, rotate=180}]
\end{tikzcd}.\]
It assembles into a functor $\Sigma \colon \C\to\C_\ast$.
\item The \textit{loop object} $\Omega_xX$ of a pointed object $(X,x)\in\C_\ast$ is the pullback
\begin{equation}
    \begin{tikzcd}\label{pullback diagram of loop space}
    {\Omega_x X} \arrow[r] \arrow[d] \arrow[dr, phantom, "\lrcorner"{anchor=center, pos=0.125}] & \ast \arrow[d] \\
    \ast \arrow[r] & {X}
\end{tikzcd}.
\end{equation}
It assembles into a functor $\Omega \colon \C_\ast\to\C_\ast$, where the canonical basepoint of $\Omega_x X$ is denoted by $\cst_x\colon \ast\to\Omega_xX$.
\end{enumerate}
\end{definition}
These functors form an adjunction
\begin{tikzcd}[ampersand replacement=\&]
	{\Sigma\colon \C_\ast} \& {\C_\ast\colon \Omega}
	\arrow[shift left, from=1-1, to=1-2]
	\arrow[shift left, from=1-2, to=1-1]
\end{tikzcd} of $\infty$-categories.
\begin{remark}\label{rmk: fiber of a nullhomtopic map is the product of the loop of the base}
    The fiber of a nullhomotopic map $X\xrightarrow{\ast}Y$ into a pointed object is $\Omega Y\times X$.
\end{remark}
\begin{definition}
     Let $n\geq0$. The \textit{$n$-sphere object} $S^n\in\C_\ast$ is defined inductively by $S^0\coloneq\ast\amalg\ast\in\C$ and $S^{n}\coloneq\Sigma S^{n-1}$, the suspension of the $(n-1)$-sphere. It is equipped with a basepoint $s\colon \ast\to S^n$.
\end{definition}

Applying the internal hom functor $X^{(-)}\colon  \C\to\C$ to the basepoint $s\colon \ast\to S^n$, we obtain the \textit{evaluation map} $ev\colon X^{S^n}\to X$. Applying it to the terminal map $S^n\to \ast$, we get the \textit{constant} or \textit{diagonal map} $\cst \colon X\to X^{S^n}$. The diagonal $\Delta\colon X\to X\times X=X^{S^0}$ is a canonical basepoint of $X^{S^0}$ in the slice over $X$. Taking the iterated loop space in $\slice{\C}{X}$ yields:
 \[\begin{tikzcd}
    {X^{S^{n}}} \arrow[r] \arrow[d] \arrow[dr, phantom, "\lrcorner"{anchor=center, pos=0.125}] & {X} \arrow[d] \\
    {X} \arrow[r] & {X^{S^{n-1}}}
\end{tikzcd}.\]

\begin{definition}
        Let $n\geq -1$ be an integer. An object $X$ in $\C$ is \begin{enumerate}
        \item \textit{$(-2)$-truncated} if it is terminal;
        \item \textit{$n$-truncated} if $\cst\colon X\to X^{S^{n+1}}$ is an equivalence for $n\geq-1$.
    \end{enumerate}
    A map $f\colon X\to Y$ is \textit{$n$-truncated} if it is $n$-truncated as an object of $\slice{\C}{Y}$. A map $f\colon X\to Y$ which is $(-1)$-truncated is called a \textit{monomorphism}. We denote by $\C^{\sleq n}\subseteq\C$ the full subcategory on the $n$-truncated objects.
\end{definition}
\begin{remark}
    A map $f\colon X\to Y$ is a monomorphism if and only if one of those two squares is Cartesian:
\[\begin{tikzcd}[ampersand replacement=\&]
	X \& X \& X \& Y \\
	X \& Y \& {X\times X} \& {Y\times Y}
	\arrow[equals, from=1-1, to=1-2]
	\arrow[equals, from=1-1, to=2-1]
	\arrow["f"', from=1-2, to=2-2]
	\arrow["f", from=1-3, to=1-4]
	\arrow["{\Delta_X}", from=1-3, to=2-3]
	\arrow["{\Delta_Y}", from=1-4, to=2-4]
	\arrow["f"', from=2-1, to=2-2]
	\arrow["{f\times f}", from=2-3, to=2-4]
\end{tikzcd}.\]
This precisely means that for every generalized element $x_0,x_1\colon T\to X$, the path objects (\cref{definition: path objects})\[X(x_0,x_1)\xrightarrow{\simeq} Y(f(x_0), f(x_1))\]  are equivalent. In $\Sp$ this means that $f$ is fully faithful, i.e. an inclusion of connected components. 
\end{remark}
\begin{example}
    Let $\{ x_i\colon T\to X\vert\, i\in I\}$ be a jointly epimorphic family of generalized points in an $\infty$-topos $\E$, and $f\colon Y\to X$ a map. By \cref{prop: coproduct preserves truncations} and universality of coproducts, $f$ is $n$-truncated if and only if its generalized fibers $F_{x_i}$ are $n$-truncated in $\slice{\E}{T}$. In $\Sp$, a non-empty $\infty$-groupoid $X$ is $n$-truncated if and only if its homotopy groups $\pi_k(X,x)$ vanish for all $x\in X$ and for all $k>n$. It is $(-1)$-truncated if it is empty or contractible.
\end{example}

\begin{prop}[{{\cite[§5.5.6]{LurieHTT}}}]\label{prop: equivalent characterization of truncated maps}
Let $f\colon X\to Y$ and $g\colon Y\to Z$ be morphisms in $\C$, and $n\geq -2$ an integer.
\begin{enumerate}[label=\arabic*., ref=\thelemma.\arabic*]
\item \label{prop: equivalent characterization of truncated maps 1} The following are equivalent:
\begin{enumerate}
    \item $f$ is $n$-truncated;
    \item the diagonal $\Delta_f\colon X\to X\times_Y X$ is $(n-1)$-truncated;
    \item for all objects $A$, the map of $\infty$-groupoids $\C(A,X)\to\C(A,Y)$ is $n$-truncated.
\end{enumerate}
If $Y$ is $n$-truncated, these are equivalent to $X$ being $n$-truncated.
    \item\label{prop: equivalent characterization of truncated maps 2} If $g$ and $gf$ are $n$-truncated, $f$ is $n$-truncated.
    \item If $f$ and $g$ are $n$-truncated, $gf$ is $n$-truncated.
\end{enumerate}
\end{prop}
\begin{remark}\label{remark: left exact conservative functor preserves and reflects n-truncated maps}
    Let $\C$ and $\C'$ be $\infty$-categories with terminal objects and pullbacks. A functor $F \colon \C \to \C'$ that preserves these finite limits necessarily preserves $n$-truncated maps. If, in addition, $F$ is conservative, it also reflects $n$-truncated maps.
\end{remark}

In a slice $\infty$-category, an object is $n$-truncated if its structure map has $n$-truncated fibers. By contrast, in a coslice, an object is $n$-truncated when its target is so:
\begin{example}\label{example: U preserves and reflects n-truncated maps}
    Let $U\colon \C^{A/}\to\C$ be a coslice forgetful functor, which preserves limits and is conservative. Then $U$ preserves and reflects $n$-truncated maps. In particular, a pointed map $f\colon(X,x)\to(Y,y)$ is $n$-truncated in $\C_\ast$ if and only if $Uf$ is so in $\C$.
\end{example}
\begin{prop}[{\cite[6.2.3.17]{LurieHTT}}]\label{prop: pullback of truncated map is truncated}
    Let $\E$ be a semitopos, and consider the following Cartesian square:
\[\begin{tikzcd}
    X \arrow[r] \arrow[d, "f'"'] \arrow[dr, phantom, "\lrcorner"{anchor=center, pos=0.125}] & Y \arrow[d, "f"] \\
    Z \arrow[r, "p"'] & T
\end{tikzcd}.\]
If $f$ is $n$-truncated, so is $f'$. The converse holds if $p$ is an effective epimorphism.
\end{prop}
\begin{prop}[{\cite[5.5.6.18~\&~6.5.1.2]{LurieHTT}}]
   Let $n\geq-2$. The inclusion $\C^{\sleq n}\subseteq\C$ of $n$-truncated objects is a reflective subcategory:
\[\begin{tikzcd}
	{\C^{\sleq n}} & {\C.}
	\arrow[""{name=0, anchor=center, inner sep=0}, shift right=2, hook, from=1-1, to=1-2]
	\arrow[""{name=1, anchor=center, inner sep=0}, "{\tau_n^{\C}}"', shift right=2, from=1-2, to=1-1]
	\arrow["\dashv"{anchor=center, rotate=-90}, draw=none, from=1, to=0]
\end{tikzcd}\]
\end{prop}
The unit of this adjunction yields a map $\eta_X\colon X\to \tau^{\C}_n X$ for all objects $X$.
\begin{notation}
    When the ambient $\infty$-category $\C$ is clear from context, we simply write $\tau_n$, removing the subscript $\C$ from the notation. The truncation functor in a slice $\infty$-category is called a \textit{fiberwise truncation}, it is denoted by $\tau^X_n\colon\slice{\C}{X}\to\left(\slice{\C}{X}\right)^{\sleq n}$. In particular for $f\colon Y\to X$ a map, $\tau^X_nf\colon \tau_n^X Y\to X$ is an object in $\slice{\C}{X}$, while $\tau_nf\colon \tau_n Y\to\tau_n X$ denotes the image of $f$ by the functor $\tau_n\colon \C\to\C$.
\end{notation}
Let $F\colon \C\to\D$ be a functor between $\infty$-categories. If $F$ preserves finite limits and is conservative, then $F$ preserves and reflects $n$-truncated objects. This does \emph{not} imply that $F$ commutes with $n$-truncation. For instance, a sheaf $\infty$-category inclusion $\Sh(\C)\hookrightarrow\Psh(\C)$ preserves and reflects $n$-truncated objects, but truncating a sheaf is not simply objectwise truncation, one must then sheafify. If $F$ additionally preserves small colimits, then it does commute with truncation by \cite[5.5.6.28]{LurieHTT}. The forgetful functor $U\colon \C_\ast\to\C$ from the pointed $\infty$-category does \emph{not} preserve colimits, yet it still commutes with $n$-truncation. In order to obtain a fibered version of this result (\cref{corollary: fibered truncation commutes with pointed fibered truncation}), we first prove it in a general coslice $\infty$-category.
\begin{prop}\label{prop: n-truncation commute with coslice forgetful functors}
    Let $U\colon\C^{A/}\to\C$ be the forgetful functor from a coslice. Then there is an equivalence of functors $U\circ \tau_n^{\C^{A/}} \simeq \tau_n^{\C}\circ U$.
\end{prop}
\begin{proof}
The functor $U$ is left exact and hence preserves $n$-truncated objects by \cref{remark: left exact conservative functor preserves and reflects n-truncated maps}. Thus it restricts to a functor $U\colon (\C^{A/})^{\sleq n}\to \C^{\sleq n}$. We must show that the square
\[
\begin{tikzcd}
    \C^{A/} \ar[r,"\tau_n^{\C^{A/}}"] \ar[d,"U"'] & (\C^{A/})^{\sleq n} \ar[d,"U"] \\
    \C \ar[r,"\tau_n^{\C}"'] & \C^{\sleq n}
\end{tikzcd}
\]
commutes. Unwinding the definitions, this means that for every object $A\to X$ in $\C^{A/}$, the composite $A\to X\to \tau_n^{\C}X$ is the $n$-truncation of $A\to X$ in $\C^{A/}$. To verify its universal property, fix an $n$-truncated object $Y\in(\C^{A/})^{\sleq n}$. The mapping spaces of the coslice sit in fiber sequences \[\C^{A/}(X,Y) \longrightarrow \C(X,Y) \longrightarrow \C(A,Y),\] where the second map evaluates at the structure morphism $A\to X$. The same holds for $\tau_n^{\C}X$ in place of $X$. These fiber sequences assemble into a commutative diagram
\[\begin{tikzcd}
    \C^{A/}(\tau_n^{\C}X,Y) \ar[r] \ar[d] & \C(\tau_n^{\C}X,Y) \ar[r] \ar[d,"\simeq"] & \C(A,Y) \ar[d,"="] \\
    \C^{A/}(X,Y) \ar[r] & \C(X,Y) \ar[r] & \C(A,Y)
\end{tikzcd}.\]
The middle vertical map is an equivalence because $\tau_n^{\C}X$ is the $n$-truncation of $X$ in $\C$, and $Y$ (via $U$) is $n$-truncated. The rightmost vertical map is the identity. Since both rows are fiber sequences, the leftmost vertical map is also an equivalence, verifying the desired universal property.
\end{proof}
As a corollary, we obtain a fibered version of the correspondence between pointed and unpointed truncations. We will use it to connect the pointed and unpointed notions of connectivity (\cref{prop: basepoint forgetful functor preserves and reflects connected maps}).
\begin{corollary}\label{corollary: fibered truncation commutes with pointed fibered truncation}
    Let $(Y,y)\in\C_\ast$ be a pointed object, and $U\colon \slice{\C_\ast}{(Y,y)}\to\slice{\C}{Y}$ the  forgetful functor. Then there is an equivalence \[U\circ\tau_n^{\C_{\ast/(Y,y)}} \simeq \tau_n^{Y}\circ U\] between functors $\slice{\C_\ast}{(Y,y)}\to \left(\slice{\C}{Y}\right)^{\sleq n}$.
\end{corollary}
\begin{proof}
    There is an equivalence of $\infty$-categories $\slice{(\C_\ast)}{(Y,y)}\simeq(\slice{\C}{Y})^{y/}$, under which $U$ is the coslice forgetful functor $(\slice{\C}{Y})^{y/}\to\slice{\C}{Y}$. The claim now follows directly from \cref{prop: n-truncation commute with coslice forgetful functors} applied to the $\infty$-category $\slice{\C}{Y}$ cosliced under the base object $y\colon \ast\to Y$.
\end{proof}
   We provide a proof of the following basic property for completeness, as we were unable to find it in the literature.
    \begin{prop}\label{prop: coproduct preserves truncations}
        Suppose that $\C$ has universal and disjoint coproducts.
        \begin{enumerate}
            \item If $f$ and $g$ are $m$-truncated maps for $m\geq -2$, then $f\amalg g$ is $m$-truncated.
            \item If $X$, $Y$ are $n$-truncated for $n\geq 0$, then $X\amalg Y$ is $n$-truncated.
        \end{enumerate} 
    \end{prop}
    \begin{proof}
        (1)\quad We proceed by induction on $m\geq -2$. Write $f\colon A\to X$ and $g\colon B\to Y$. The base case $m=-2$ is trivial. If $\Delta_f$ and $\Delta_g$ are $(m-1)$-truncated, we have to show that \[\Delta_{f\amalg g}\colon A\amalg B\to (A\amalg B)\times_{(X\amalg Y)}(A\amalg B)\] is $(m-1)$-truncated. By the universality and disjointness of coproducts in $\C$, the target decomposes as $(A\times_XA)\amalg (B\times_YB)$ and $\Delta_{f\amalg g}=\Delta_f\amalg\Delta_g$, which is $(m-1)$-truncated by induction.
        
        (2)\quad The map $X\amalg Y\to \ast\amalg \ast$ is $n$-truncated by the first point. We get the result by composing with the $0$-truncated map $\ast\amalg\ast\to \ast$.
    \end{proof}
    The following property, which can be seen as a partial converse to Proposition~\ref{prop: equivalent characterization of truncated maps 1}, appears in \cite[Lemma~2.12]{BachmannKlausAnton_MonadicResolutionsGeneralizedSpaces}, where it is proved for an $\infty$-topos $\E$ using a specific presentation $\E\simeq \mathrm{Sh}(\C,\Sp)$.
    \begin{prop}\label{prop: property for an n-truncated map from an n+1-truncated object}
        Let $\E$ be a semitopos, $n\geq -2$ an integer, and $f:A\twoheadrightarrow X$ an $n$-truncated effective epimorphism in $\E$. Then $A$ is $(n+1)$-truncated if and only if $X$ is $(n+1)$-truncated. 
    \end{prop}
    \begin{proof}
        If $X$ is $(n+1)$-truncated, this is Proposition~\ref{prop: equivalent characterization of truncated maps 1}. For the converse, consider the following Cartesian square 
\[\begin{tikzcd}
    {A \times_X A} \arrow[r, two heads] \arrow[d, "\alpha"'] \arrow[dr, phantom, "\lrcorner"{anchor=center, pos=0.125}] & X \arrow[d, "\Delta"] \\
    {A \times A} \arrow[r, "{f \times f}"', two heads] & {X \times X}
\end{tikzcd}.\]
Since $f$ is $n$-truncated, \cref{prop: pullback of truncated map is truncated} ensures that the projections $p_i:A\times _XA\to A$ are $n$-truncated. Because $A$ is an $(n+1)$-truncated object, it then follows from the first part of the proof that $A\times _XA$ is also $(n+1)$-truncated. Consequently, $\alpha$ factors as a composition of two $n$-truncated maps: \[\alpha:A\times_XA\xrightarrow{\Delta}(A\times_XA)\times(A\times_XA)\xrightarrow{p_1\times p_2}A\times A\]
and is therefore $n$-truncated. Finally since $f\times f$ is an effective epimorphism, $\Delta:X\to X\times X$ is also $n$-truncated.
    \end{proof}
    \begin{corollary}\label{cor: realization of n truncated groupoid is n+1 truncated}
        Let  $n\geq -2$ be an integer and $\E$ an $\infty$-topos. Suppose that $U_\sbullet\colon \Delta^\op\to\E$ is a groupoid object such that $U_0$ and $U_1$ are $n$-truncated. Then its realization $\vert U_\sbullet\vert$ is $(n+1)$-truncated.
    \end{corollary}
    \begin{proof}
        Because groupoid objects are effective, the following square is Cartesian (see \cref{remark: definition groupoid object}): 
\[\begin{tikzcd}
    {U_1} \arrow[r, "{d_0}"] \arrow[d, "{d_1}"'] \arrow[dr, phantom, "\lrcorner"{anchor=center, pos=0.125}] & {U_0} \arrow[d, two heads] \\
    {U_0} \arrow[r, two heads] & {\vert U_\sbullet \vert}
\end{tikzcd}.\]
The faces $d_i$ are $n$-truncated since both their domain and target are. It follows that $U_0\to\vert U\vert$ is an $n$-truncated effective epimorphism, whose domain is $(n+1)$-truncated. We conclude by \cref{prop: property for an n-truncated map from an n+1-truncated object}.
    \end{proof}
      \begin{remark}\label{rmk: on truncations of realizations}
         Under the assumptions of \cref{cor: realization of n truncated groupoid is n+1 truncated}, $U_k$ is $n$-truncated as well, for any $k\geq 2$. Also, if one of the two faces $d_i\colon U_1\to U_0$ is a monomorphism, then $U_0\simeq\vert U_\sbullet\vert$.
    \end{remark}
\begin{prop}\label{truncation level of the object of equivalences}
    Let $n\geq-2$ and $A,B\in\C$ be objects such that $B$ is $n$-truncated. Then $\Eq(A,B)\in\C$ is $n$-truncated.
\end{prop}
\begin{proof}
    There is a monomorphism $\Eq(A,B)\hookrightarrow\Map(A,B)$, and the latter object is $n$-truncated.
\end{proof}
As an application of \cref{truncation level of the object of equivalences}, we can determine the truncation level of the universe of truncated maps.
\begin{corollary}\label{universe of n-truncated maps is (n+1)-truncated}
    A univalent universe $\U^{\sleq{n}}\in\C$ classifying $n$-truncated maps is itself $(n+1)$-truncated.
\end{corollary}
\begin{proof}
    It suffices to show that $\Delta_{\U}\colon \U^{\sleq{n}}\to\U^{\sleq{n}}\times \U^{\sleq{n}}$ is $n$-truncated. But its generalized fiber at a generalized point $(\corner{A},\corner{B})\colon T\to\U^{\sleq{n}}\times\U^{\sleq{n}}$ is $\slice{\Eq}{T}(A,B)$, which is $n$-truncated in $\slice{\E}{T}$ since $B\to T$ is so. We conclude by choosing $T=\U^{\sleq{n}}\times \U^{\sleq{n}}$ and $(\corner{A},\corner{B})\coloneq1_{\U^{\sleq{n}}\times \U^{\sleq{n}}}$.
\end{proof}
We conclude this subsection by identifying the sections of certain maps $E\to X$ with those of their $n$-truncation $\tau_nE\to\tau_nX$.
\begin{prop}\label{prop: cartesian square of n-truncation implies equivalence on the objects of sections}
    Consider a Cartesian square in an $\infty$-topos
\[\begin{tikzcd}
    E \arrow[r, "\eta_E"] \arrow[d, "p"'] \arrow[dr, phantom, "\lrcorner"{anchor=center, pos=0.125}] & \tau_n E \arrow[d, "\tau_n p"] \\
    X \arrow[r, "\eta_X"'] & \tau_n X
\end{tikzcd},\]
where the horizontal arrows are $n$-truncation maps. Then there is an equivalence \[\prod_{X}E\simeq\prod_{\tau_nX}\tau_nE.\]
\end{prop}
\begin{proof}
    By \cite[Rmk.~5.1.32]{Schreiber_DifferentialCohomologyCohesiveInfinityTopos} there is an identification of the object of sections \[\prod_BA\simeq \Map(B,A)\times_{\Map(A,A)}\{1_A\}\]
    for maps $A\to B$. Using this we obtain a pasting of Cartesian squares
\[\begin{tikzcd}
    {\prod_X E} \arrow[r] \arrow[d] \arrow[dr, phantom, "\lrcorner"{anchor=center, pos=0.125}] & 
    {\Map(X,E)} \arrow[r] \arrow[d, "{p_!}"] \arrow[dr, phantom, "\lrcorner"{anchor=center, pos=0.125}] & 
    {\Map(X,\tau_n E)} \arrow[r, "\simeq"] \arrow[d] \arrow[dr, phantom, "\lrcorner"{anchor=center, pos=0.125}] & 
    {\Map(\tau_n X, \tau_n E)} \arrow[d] \\
    \ast \arrow[r, "{1_X}"'] & 
    {\Map(X,X)} \arrow[r] & 
    {\Map(X,\tau_n X)} \arrow[r, "\simeq"'] & 
    {\Map(\tau_n X, \tau_n X)}
\end{tikzcd}.\]
 The middle square is Cartesian since $\Map(X,-)$ preserves limits, and the right square is Cartesian since horizontal arrows are equivalences. It follows that the limit object is $\prod_{\tau_nX}\tau_nE$, as desired.
\end{proof}
\subsection{Connectivity}\label{section: preliminaries subsection: connectivity and homotopy groups}
The notion of an $n$-connected map may be defined in various contexts: for instance in the setting of a locally Cartesian closed $\infty$-category with finite limits and colimits \cite[§1]{Rasekh_ElementaryApproachTruncations}, or alternatively within any presentable $\infty$-category \cite[§4.2]{SchlankYanovski_InfiniteCategoricalEckmanHilton}. However, we require these maps to exhibit well-behaved properties, such as forming a factorization system with truncated maps, or identifying $(-1)$-connected maps with effective epimorphisms. Consequently, we will mostly adopt the non-minimal setting of an $\infty$-topos $\E$, which simultaneously provides all of these desired features.
\begin{definition}
    Let $X$ be an object of a presentable $\infty$-category $\C$, and $n\geq -2$. We say that $X$ is \textit{$n$-connected} if its $n$-truncation $\tau_nX\simeq \ast$ is terminal. A map $f\colon X\to Y$ is $n$-connected if it is $n$-connected as an object of $\slice{\E}{Y}$. 
\end{definition}
\begin{remark}
    \begin{enumerate}
        \item In the $\infty$-category $\Sp$ of homotopy types, a map $f\colon X\to Y$ is $n$-connected if and only if for all $x\in X$ the induced homomorphisms $\pi_k(f)\colon \pi_k(X,x)\to\pi_k(Y,f(x))$ are isomorphisms for all $k\leq n$ and is surjective for $k=n+1$.
        \item Our convention differs from the one of Lurie \cite{LurieHTT}, but coincides with \cite{GepnerKock_UnivalenceInLCCC, Rasekh_ElementaryApproachTruncations, SchlankYanovski_InfiniteCategoricalEckmanHilton}, for instance.
        \item We say $X$ is \textit{inhabited} if it is $(-1)$-connected, \textit{connected} if it is $0$-connected, and \textit{simply connected} if it is $1$-connected.
        \item In an $\infty$-topos, the $(-1)$-connected maps are precisely the effective epimorphisms \cite[6.2.3.4]{LurieHTT}.
        \item We denote by $\C^{\sg n}$ the full subcategory of $\C$ on the $n$-connected objects.
    \end{enumerate}
\end{remark}
\begin{prop}\label{prop: characterization connected maps}
    Let $f\colon X\to Y$ be a map in an $\infty$-topos $\E$ with a section $s\colon Y\to X$, and let $n\geq -1$. The following are equivalent:
    \begin{enumerate}
    \item $f$ is $n$-connected;
    \item $s$ is $(n-1)$-connected;
    \item $\Delta_f\colon X\to X\times_Y X$ is $(n-1)$-connected and $f$ is an effective epimorphism.
    \end{enumerate}
    In that case $\tau_nf\colon\tau_nX\to\tau_nY$ is an equivalence.
\end{prop}
\begin{proof}
    By working in the slice $\slice{\E}{Y}$ we can suppose that $Y\simeq\ast$ is terminal. The implication $1)\Rightarrow 2)$ is \cite[Prop.~8.8]{Rezk_HomotopyTopos}.
    Notice that $s\colon \ast\to X$ is a pullback of $\Delta_X$ along $\ast\times X\to X\times X$. Since the latter is an effective epimorphism by \cite[Lem.~3.20]{Rasekh_ElementaryApproachTruncations} this gives $2)\Leftrightarrow 3)$. Lastly \cite[Prop.~3.54]{Rasekh_ElementaryApproachTruncations} gives $3)\Rightarrow 1)$. The last statement is \cite[Lemma~3.27]{Rasekh_ElementaryApproachTruncations}.
\end{proof}
    Let $\C$ and $\C'$ be presentable $\infty$-categories. Any left adjoint functor $F \colon\C\to\C'$ preserves $n$-connected maps \cite[Lem.~4.2.5]{SchlankYanovski_InfiniteCategoricalEckmanHilton}. Despite not being a left adjoint, the forgetful functor $U\colon\C_\ast\to\C$ enjoys the same property, and in fact reflects $n$-connected maps as well.
\begin{prop}\label{prop: basepoint forgetful functor preserves and reflects connected maps}
    The forgetful functor $U\colon\C_\ast\to\C$ preserves and reflects $n$-connected maps. That is, a map $f\colon(X,x)\to(Y,y)$ in $\C_\ast$ is $n$-connected if and only if its underlying map $f\colon X\to Y$ is $n$-connected in $\C$.
\end{prop}
\begin{proof}
    By definition, $f$ is $n$-connected in $\C_\ast$ if and only if its $n$-truncation in the slice $\slice{\C_\ast}{(Y,y)}$ is the identity on $(Y,y)$:
    \[\tau_n^{\slice{\C_\ast}{(Y,y)}}f\simeq1_{(Y,y)}.\] 
    Since $U$ is conservative, this holds if and only if    \[U\left(\tau_n^{\slice{\C_\ast}{(Y,y)}} f\right)\simeq U(1_{(Y,y)})=1_Y.\] 
    By \cref{corollary: fibered truncation commutes with pointed fibered truncation}, the left-hand side is equivalent to $\tau_n^{\slice{\C}{Y}}(U f)$. Therefore the condition becomes $\tau_n^{\slice{\C}{Y}}(f) \simeq 1_Y$, which is precisely the definition of $f$ being $n$-connected in $\C$.
\end{proof}
\begin{prop}[{\cite[Cor. 3.36 \& Prop. 3.50 \& Prop. 1.33]{Rasekh_ElementaryApproachTruncations}}]\label{prop: connected maps are preserved by pullback and pushouts}
    Consider a commutative square in an $\infty$-topos $\E$:
\[\begin{tikzcd}
	{X'} & X \\
	{Y'} & Y
	\arrow[from=1-1, to=1-2]
	\arrow["{f'}", from=1-1, to=2-1]
	\arrow["f", from=1-2, to=2-2]
	\arrow["p",from=2-1, to=2-2]
\end{tikzcd}\]
Suppose that the square is
    \begin{enumerate}
        \item Cartesian. Then, if $f$ is $n$-connected, so is $f'$. The converse holds if $p$ is an effective epimorphism.
        \item coCartesian. Then, if $f'$ is $n$-connected, so is $f$. 
         \end{enumerate}
\end{prop}
\begin{prop}[{\cite[Prop. 5.18]{Rasekh_ElementaryApproachTruncations}\cite[Prop.~3.3.6]{ABFJ_GeneralizedBlakersMassyTheorem}}]\label{prop: connected truncated factorization system}
    Let $\E$ be an $\infty$-topos. The classes $(C_n,T_n)$ of $n$-connected and $n$-truncated maps form a factorization system whose left class is stable under base change. In particular the class $C_n$ is left orthogonal to the class $T_n$, $C_n$ is stable under colimits, and $T_n$ is stable under limits.
\end{prop}
\begin{remark}\label{remark: fiberwise truncation is the n-connected/truncated factorization}
    By \cite[Prop.~8.5]{Rezk_HomotopyTopos}, the $n$-connected/$n$-truncated factorization is given by the reflector of the fibered $n$-truncation functor: \[X\to\tau_n^Y X\xrightarrow{\tau_n^Y f}Y.\]  In particular, the unit map $\eta_X\colon X\to\tau_nX$ is $n$-connected.
\end{remark}
\begin{remark}\label{rmk: construction of the epi mono factorization}
    Let $\C$ be an $\infty$-category with finite limits and realizations, effective groupoid objects, and universal $\Delta^\op$-shaped colimits. The factorization of a morphism $f \colon X \to Y$ in $\C$ into an effective epimorphism followed by a monomorphism can be explicitly constructed using the \v{C}ech nerve $\check{C}(f) \colon \Delta_+^\op \to \C$. 
    Specifically, the realization of the \v{C}ech nerve yields an effective epimorphism $p \colon X \twoheadrightarrow |\check{C}(f)|$. By the universal property of the realization, there exists a unique comparison map $m \colon |\check{C}(f)| \to Y$ such that the following diagram commutes:
    \[\begin{tikzcd}
        X \arrow[rr, "f"] \arrow[dr, two heads, "p"'] & & Y \\
        & {|\check{C}(f)|} \arrow[ur, hook, "m"']
    \end{tikzcd}.\]
    One can show that $m$ is $(-1)$-truncated (a monomorphism) following the argument in \cite[6.2.3.4]{LurieHTT}. Although Lurie's result is formulated for semitopoi, the proof requires only that $\Delta^{\op}$-shaped colimits are universal. We call this an effective epi/mono factorization. In an $\infty$-topos, this is the $(-1)$-connected/$(-1)$-truncated factorization.
\end{remark}
\begin{definition}\label{definition: image of a map}
    Let $\E$ be an $\infty$-topos.
    \begin{itemize}
        \item For any map $f\colon X\to Y$, its \textit{$n$-image} is the intermediate object $\im_n(f)$ arising from the $n$-connected/$n$-truncated factorization $X\to \im_n(f)\to Y$. We simply call the \textit{image} the $(-1)$-image, which corresponds to the effective epi/mono factorization $X\twoheadrightarrow \im(f)\hookrightarrow Y$.
        \item If $x\colon\ast\to X$ is a global point, we call its $(n-1)$-image $\im_{n-1}(x)$ the \textit{$n$-connected cover} of $X$ at $x\in X$, and denote the factorization by $\ast\to X\langle n\rangle\to X$. When $n=0$, this object is called the \textit{connected component} of $x\in X$.
    \end{itemize}
\end{definition}
\begin{remark}\label{remark: contruction of the n-connected cover}
    The $n$-connected cover of a pointed object $(X,x)$ can be constructed as the following pullback:
\[\begin{tikzcd}
    {X\langle n\rangle} \arrow[r] \arrow[d] \arrow[dr, phantom, "\lrcorner"{anchor=center, pos=0.125}] & \ast \arrow[d] \\
    X \arrow[r, "{\eta_X}"] & {\tau_n X}
\end{tikzcd}.\]
\end{remark}

\begin{prop}[{\cite[Prop. III.4.21]{Samuel_Thesis}}]\label{prop: n-connected pointed spaces is coreflective in pointed spaces}
    Let $n\geq -1$. The full inclusion $\E^{\sg n}_\ast\to \E_\ast$ admits a colocalization
\[\begin{tikzcd}
	{\E_\ast^{\sg n}} & {\E_\ast}
	\arrow[""{name=0, anchor=center, inner sep=0}, shift left=2, hook, from=1-1, to=1-2]
	\arrow[""{name=1, anchor=center, inner sep=0}, "{-\langle n\rangle}", shift left=2, from=1-2, to=1-1]
	\arrow["\dashv"{anchor=center, rotate=-90}, draw=none, from=0, to=1]
\end{tikzcd},\]
where the right adjoint sends a pointed object $(X,x)$ to its $n$-connected cover $(X\langle n\rangle, x)$. The counit of the adjunction $\varepsilon\colon X\langle n\rangle \to X$ is the $(n-1)$-image of the basepoint $\ast\to X$.
\end{prop}

One can now characterize $n$-truncated objects using $n$-connected covers.
\begin{corollary}\label{prop: cover whose n-connected objects are contractible is n-truncated}
    Let $\{x_i:\ast\to X\vert\,i\in I\}$ be a jointly epimorphic family of global points in $\E$. Then $X$ is $n$-truncated if and only if its $n$-connected covers $X_i\langle n\rangle\simeq\ast$ at $x_i\in X$ are contractible for all $i\in I$.
\end{corollary}
\begin{proof}
    Suppose that $X$ is $n$-truncated. Its diagonal $\Delta_X$ is $(n-1)$-truncated. Consequently every base change of $\Delta_X$, including each global point $x_i\colon \ast\to X$, is $(n-1)$-truncated. This immediately implies that each $n$-connected cover is contractible.

    Conversely, suppose $X_i\langle n \rangle \simeq \ast$ for all $i \in I$, which means each map $x_i \colon \ast \to X$ is $(n-1)$-truncated. The induced map $p \colon \coprod_{i\in I} \ast \twoheadrightarrow X$ is an effective epimorphism. Let $U$ be its \v{C}ech nerve. By the universality of coproducts, $U_1 \simeq \coprod_{i,j\in I} X(x_i,x_j)$. Since each $X(x_i, x_j)$ is a pullback of the $(n-1)$-truncated maps $x_i$ and $x_j$, it is also $(n-1)$-truncated. We proceed in two cases:
\begin{itemize}
    \item \text{Case 1: $n\geq 1$.} Because $n-1 \geq 0$, taking coproducts preserves the truncation level by \cref{prop: coproduct preserves truncations}. Therefore, the entire object $U_1$ is $(n-1)$-truncated. Since $p$ is an effective epimorphism, $X$ is the realization of the groupoid object $U$. It then follows from \cref{cor: realization of n truncated groupoid is n+1 truncated} that $X$ is $n$-truncated.
    \item \text{Case 2: $n= 0$.} Here, $n-1 = -1$. The coproduct of $(-1)$-truncated objects is generally not $(-1)$-truncated, so the previous argument does not apply. Instead, each map $X(x_i,x_j)\hookrightarrow \ast$ is a monomorphism. Consider the following Cartesian square:
\[\begin{tikzcd}
	{U_1\simeq\coprod_{i,j\in I}X(x_i,x_j)} \arrow[d, "{(d_0,d_1)}"'] \arrow[r] \arrow[dr, phantom, "\lrcorner"{anchor=center, pos=0.125}] & X \arrow[d, "{\Delta_X}"] \\
	{U_0\times U_0\simeq(\coprod_{i\in I}\ast)\times(\coprod_{j\in I}\ast)} \arrow[r, "{p\times p}"', two heads] & {X\times X}
\end{tikzcd}\]
Because coproducts are universal, the bottom left object is $\coprod_{i,j\in I}\ast$, and $(d_0,d_1)$ is the coproduct indexed by $I\times I$ of monomorphisms $X(x_i,x_j)\hookrightarrow\ast$. Since monomorphisms are stable under coproducts by \cref{prop: coproduct preserves truncations}, $(d_0,d_1)$ is a monomorphism as well. Because $p\times p$ is an effective epimorphism, $\Delta_X$ is a monomorphism. It follows that $X$ is $0$-truncated, as desired.
\end{itemize}
\end{proof}

We now consider the behavior of connected and truncated maps under the loop construction.
\begin{lemma}\label{lemma: truncation and connectivity of loop space}
     Let $f\colon (X,x_0)\to (Y,y_0)$ be a pointed map in $\E_\ast$, and $n\geq-1$. 
     \begin{enumerate}[label=\arabic*., ref=\thelemma.\arabic*]
         \item\label{lemma: truncation and connectivity of loop space 1} If $f\colon X\to Y$ is $n$-truncated (resp. $n$-connected map), then $\Omega f\colon \Omega X\to\Omega Y$ is $(n-1)$-truncated (resp. $(n-1)$-connected).
         \item\label{lemma: truncation and connectivity of loop space 2} If $X$ and $Y$ are connected, then the converse holds for connectivity: if $\Omega f$ is $(n-1)$-connected, then $f$ is $n$-connected.
         \item\label{lemma: truncation and connectivity of loop space 3} If $X$, $Y$ and $f$ are connected, then the converse holds for truncated maps: if $\Omega f$ is $(n-1)$-truncated, then $f$ is $n$-truncated.
     \end{enumerate}
\end{lemma}
\begin{proof}
    Consider the following pasting of Cartesian squares:
\begin{equation}\label{lemma: truncation and connectivity of loop space, eq1}
    \begin{tikzcd}
    {\Omega X} \arrow[r] \arrow[d, "{\Omega f}"'] \arrow[dr, phantom, "\lrcorner"{anchor=center, pos=0.125}] & \ast \arrow[d, "{x_0}"] \\
    {\Omega Y} \arrow[r] \arrow[d] \arrow[dr, phantom, "\lrcorner"{anchor=center, pos=0.125}] & F \arrow[r] \arrow[d] \arrow[dr, phantom, "\lrcorner"{anchor=center, pos=0.125}] & \ast \arrow[d, "{y_0}"] \\
    \ast \arrow[r, "{x_0}"'] & X \arrow[r, "f"'] & Y
\end{tikzcd}.
\end{equation}
\begin{enumerate}
    \item If $f$ is $n$-truncated (resp. $n$-connected), so is $F$. This implies that $\Delta_F:F\to F\times F$ is $(n-1)$-truncated (resp. $(n-1)$-connected), and so is $x_0:\ast\to F$ since it is a base change of $\Delta_F$. Hence its own base change $\Omega f$ is as well.
    \item Pulling back along effective epimorphisms reflects connectivity levels. Since $x_0:\ast\to X$ is an effective epimorphism, so is $\Omega Y\to F$, and thus $x_0:\ast\to F$ is $(n-1)$-connected. As it is a section of $F\to \ast$, we get that $F$ is $n$-connected. Since $y_0$ is an effective epimorphism, $f$ is $n$-connected.
    \item  Pulling back along effective epimorphisms reflects truncation levels. Moreover every arrow in \eqref{lemma: truncation and connectivity of loop space, eq1} is an effective epimorphism: $f$ being connected implies that $x_0:\ast\twoheadrightarrow F$ is one. Since $x_0:\ast\twoheadrightarrow F$ is an $(n-1)$-truncated effective epimorphism whose domain is $n$-truncated, $F$ is $n$-truncated by \cref{prop: property for an n-truncated map from an n+1-truncated object}. Since $y_0\colon \ast\twoheadrightarrow Y$ is an effective epimorphism, $f$ is $n$-truncated as well. 
\end{enumerate}
\end{proof}
As a corollary, we find that the truncation functor behaves well with loop objects:
\begin{corollary}\label{cor: loop space commutes with truncation}
    Let $f\colon (X,x)\to (Y,y)$ be a pointed map in $\E_\ast$, and $\eta_{n+1}^Y\colon X\to \tau_{n+1}^YX$ its $(n+1)$-truncation in $\slice{\E}{Y}$. Then $\Omega (\eta^Y_{n+1})\simeq\eta_n^{\Omega Y}$ defines an $n$-truncation of $\Omega X$ in $\slice{\E}{\Omega Y}$:
    \[\Omega\left(\tau_{n+1}^Y X\right)\simeq\tau_n^{\Omega Y}(\Omega X).\]
\end{corollary}
\begin{proof}
    By \cref{remark: fiberwise truncation is the n-connected/truncated factorization}, the $(n+1)$-truncation of $f$ in $\slice{\E}{Y}$ is given by its $(n+1)$-connected/$(n+1)$-truncated factorization:
    \[X\to \tau_{n+1}^Y X\to Y.\]
    Applying $\Omega$ we obtain an $n$-connected/truncated factorization of $\Omega f$ by Lemma~\ref{lemma: truncation and connectivity of loop space 1}, yielding an $n$-truncation in $\slice{\E}{\Omega Y}$:
    \[\Omega X\to \Omega(\tau_{n+1}^Y X)\simeq \tau_n^{\Omega Y}(\Omega X)\to \Omega Y.\]
\end{proof}

In \cite[§6.2.3]{LurieHTT} Lurie characterizes $(-1)$-connected maps (i.e. effective epimorphisms), as exactly those maps $f\colon X\to Y$ for which the induced pullback functor on the ordinary categories of subobjects, $f^\ast\colon \mathrm{Sub}(Y)\to \mathrm{Sub} (X)$, is injective. We now extend this result to characterize $n$-connected maps. While the forward implication is known \cite[Lem.~3.61]{Rasekh_ElementaryApproachTruncations}, \cite[7.2.1.13]{LurieHTT}, we are not aware of a similar result for the converse in the literature. Note that because $f^\ast\colon \slice{\E}{Y}\to\slice{\E}{X}$ is left exact, it restricts to the full sub-$\infty$-categories of $n$-truncated objects.
\begin{prop}\label{prop: characterization n-connected maps through pullback functor}
    Let $n\geq -1$ be an integer. A map $f\colon X\to Y$ in $\E$ is $n$-connected if and only if the pullback functor \[f^\ast\colon \left(\slice{\E}{Y}\right)^{\sleq{n}}\to \left(\slice{\E}{X}\right)^{\sleq{n}}\]
    is fully faithful. Furthermore, when $f$ is $n$-connected, this functor restricts to an equivalence on truncated objects in the slice $\infty$-categories $f^\ast\colon \left(\slice{\E}{Y}\right)^{\sleq{m}}\xrightarrow{\simeq}\left(\slice{\E}{X}\right)^{\sleq{ m}}$ for all $m<n$.
\end{prop}
\begin{proof}
        Because of the equivalence $\slice{\left(\slice{\E}{Y}\right)}{f}\simeq\slice{\E}{X}$, we may assume that $Y\simeq \ast$ is terminal by working in the slice $\slice{\E}{Y}$. Suppose first that $X$ is $n$-connected, and let $A,B\in\E$ be $n$-truncated objects. We need to show that the induced map on mapping spaces \begin{equation}\label{prop: characterization n-connected maps through pullback functor, eq1}
        \E(A,B)\to \slice{\E}{X}(A\times X, B\times X)
        \end{equation}
    is an equivalence of $n$-groupoids.
    The internal mapping object $\Map(A,B)\in\E$ precisely represents these mapping spaces, allowing us to rewrite \eqref{prop: characterization n-connected maps through pullback functor, eq1} as: \begin{equation}\label{prop: characterization n-connected maps through pullback functor, eq2}
    \E(1,\Map(A,B))\to \E(X,\Map(A,B)).
    \end{equation}
    Because $B$ is $n$-truncated, the internal hom $\Map(A,B)$ is also $n$-truncated, so that the mapping space $\E(X,\Map(A,B))$ is equivalent to $\E(\tau_nX, \Map(A,B))$. We can conclude since $X$ is $n$-connected.

    Conversely, suppose that the pullback functor $f^\ast$ is fully faithful. We wish to show that $\tau_nX\to\ast$ is an equivalence. Since both are $n$-truncated objects, we may work entirely in the full subcategory $\E^{\sleq{n}}\subseteq\E$ and show that the induced map on mapping spaces \[\E^{\sleq{n}}(1,C)\to \E^{\sleq{n}} (\tau_nX,C)\] is an equivalence for all $n$-truncated objects $C$. This is achieved by the following sequence of equivalences:
    \begin{align*}
        \E(1,C)&\simeq \E(1,\Map(1,C)) &\\
         &\simeq \E(X,\Map(1,C)) &&\proofstep{By hyp. \eqref{prop: characterization n-connected maps through pullback functor, eq2} is an equivalence}\\
         &\simeq \E(\tau_nX,\Map(1,C)) &&\proofstep{$\Map(1,C)$ is $n$-truncated}\\
        &\simeq \E(\tau_nX,C).&
    \end{align*}
    This implies that $\tau_nX\to 1$ is an equivalence, so $X$ is $n$-connected.
    
    Finally, assume that $f$ is $n$-connected and $m<n$. To make the construction explicit for future use, we drop the assumption that $Y\simeq\ast$. The functor $f^\ast$ clearly restricts to a fully faithful functor on $m$-truncated objects, so we only need to show essential surjectivity. Suppose that $p\colon E\to X$ is an $m$-truncated map. It arises as the pullback of a classifying map $\corner{E}\colon X\to\U^{\sleq{m}}$. Because the universe $\U^{\sleq{m}}$ is $(m+1)$-truncated (\cref{universe of n-truncated maps is (n+1)-truncated}), the classifying map $\corner{E}$ factors through the truncation map $\eta_X\colon X\to\tau_{m+1} X$. Since $f$ is $n$-connected and $(m+1)\leq n$, $\eta_X$ itself factors through $f$ to obtain a factorization  $\corner{E}\colon X\to Y\to \tau_{m+1}X\to\U^{\sleq{m}}$. Hence we have a pasting of pullbacks 
\begin{equation}\label{prop: characterization n-connected maps through pullback functor, eq3}
\begin{tikzcd}
    E \arrow[r] \arrow[d, "{p}"'] \arrow[dr, phantom, "\lrcorner"{anchor=center, pos=0.125}] & 
    {E'} \arrow[r] \arrow[d, "{p'}"'] \arrow[dr, phantom, "\lrcorner"{anchor=center, pos=0.125}] & 
    {E_{m+1}} \arrow[r] \arrow[d, "p_{m+1}"'] \arrow[dr, phantom, "\lrcorner"{anchor=center, pos=0.125}] & 
    {\U^{\sleq{m}}_\ast} \arrow[d] \\
    X \arrow[r, "f"', two heads] & 
    Y \arrow[r, two heads] & 
    {\tau_{m+1}X} \arrow[r] & 
    {\U^{\sleq{m}}}
\end{tikzcd},
\end{equation}
where $p'$ is $m$-truncated since $p$ is (\cref{prop: pullback of truncated map is truncated}).
    \end{proof}

   \begin{corollary}[{\cite[Prop. 7.2.1.14]{LurieHTT}}]\label{effective epi are seen in 0-truncation}
        A morphism $f\colon X\to Y$ in $\E$ is an effective epimorphism if and only if $\tau_0f\colon \tau_0X\to\tau_0Y$ is an effective epimorphism in the $1$-topos $\E^{\sleq 0}$.
    \end{corollary}
\begin{corollary}\label{cor: characterization n-connected maps through pullback functor in the pointed category}
    Let $n>m\geq -2$ be integers, and $f\colon (X,x)\to (Y,y)$ a pointed and $n$-connected map in $\E_\ast$. Then there is an equivalence \[f^\ast\colon \left(\slice{\E_\ast}{(Y,y)}\right)^{\sleq{m}}\xrightarrow{\simeq}\left(\slice{\E_\ast}{(X,x)}\right)^{\sleq{ m}}.\]
\end{corollary}
\begin{proof}
    Recall that if $A,B$ are objects of a coslice $\infty$-category $\C^{Z/}$, the corresponding mapping space sits in a fiber sequence \[\C^{Z/}(A,B)\to\C(A,B)\to\C(Z,B),\]
    for objects $A,B\in\C^{Z/}$. Moreover, there is a canonical equivalence \[\slice{\E_\ast}{(X,x)}\simeq \left(\slice{\E}{X}\right)^{x/}.\]
    We can suppose without loss of generality that $(Y,y)$ is terminal. We start by showing that $f^\ast=-\times X$ is fully faithful. Let $(A,a),(B,b)\in\E_\ast^{\sleq m}$, which are $m$-truncated in $\E$ by \cref{example: U preserves and reflects n-truncated maps}. We have to show that the induced map
    \begin{equation}\label{cor: of the characterization n-connected maps through pullback functor in the pointed category eq}
        \E_\ast(A,B)\xrightarrow{-\times (X,x)}\slice{\E_\ast}{(X,x)}\left(A\times X,B\times X\right)
    \end{equation}
    is an equivalence. Consider the following commutative diagram of mapping spaces:
\[\begin{tikzcd}
	{\E(A,B)} & {\slice{\E}{X}\left(A\times X,B\times X\right)} \\
	& {\slice{\E}{X}(\ast\times X,B\times X)} \\
	{\E(\ast,B)} & {\slice{\E}{X}(\ast, B\times X)}
	\arrow["{-\times X}", from=1-1, to=1-2]
	\arrow["{a^\ast}"', from=1-1, to=3-1]
	\arrow["{(a\times X)^\ast}"', from=1-2, to=2-2]
	\arrow["{x^\ast}", from=2-2, to=3-2]
	\arrow["{-\times X}", from=3-1, to=2-2]
	\arrow[from=3-1, to=3-2]
\end{tikzcd},\]
where the bottom horizontal arrow is the composition $x^\ast\circ (-\times X)$.
The induced map on the fibers is precisely \eqref{cor: of the characterization n-connected maps through pullback functor in the pointed category eq}. It suffices to show that both horizontal arrows are equivalences. Because $X$ is $n$-connected in $\E$ by \cref{prop: basepoint forgetful functor preserves and reflects connected maps}, both $-\times X$ maps are equivalences by \cref{prop: characterization n-connected maps through pullback functor}. The bottom dashed composition is an equivalence by adjunction. 

To prove that $f^\ast$ is essentially surjective, let $p\colon (E,e)\to (X,x)$. By \cref{prop: characterization n-connected maps through pullback functor}, there exists $p'\colon E'\to Y$ such that $E'\times_YX\simeq E$. The basepoint of $E$ yields a canonical basepoint of $E'$, which renders $p'$ pointed, as desired. 
\end{proof}

 The following corollary is crucial for the proof of \cref{theorem: Deck p is loop space of the base B}.

\begin{corollary}\label{cor: n-connected map induce mono on universe of n-truncated objects}
     Let $n\geq -1$ be an integer, $f\colon X\to Y$ an $n$-connected map in $\E$, and $\U^{\sleq{n}}$ a universe for $n$-truncated maps. Then the induced map \[(\U^{\sleq{n}})^Y\to (\U^{\sleq{n}})^X\]
     is a monomorphism.
 \end{corollary}
\begin{proof}
To show that $f^*$ is a monomorphism, it suffices to show that for every object $C \in \E$, the induced map on mapping spaces
\[\E(C, (\U^{\sleq{n}})^Y) \to \E(C, (\U^{\sleq{n}})^X)\]
is fully faithful. By adjunction and definition of the universe $\U^{\sleq{n}}$, this is equivalent to showing that the pullback functor
\[ (f \times C)^* \colon \left( \slice{\E}{Y \times C}^\simeq \right)^{\sleq{n}} \to \left( \slice{\E}{X \times C}^\simeq \right)^{\sleq{n}}\]
is fully faithful. Moreover $f\times C$ is $n$-connected by \cite[Lem.~1.41]{Rasekh_ElementaryApproachTruncations}. By \cref{prop: characterization n-connected maps through pullback functor}, the pullback functor on the full slice $\infty$-categories of $n$-truncated objects is fully faithful. Since any fully faithful functor restricts to a fully faithful functor on the maximal subgroupoids, the result follows.
\end{proof}
 The following corollary is well known \cite[Cor.~3.62]{Rasekh_ElementaryApproachTruncations}, but we mention it for further use.
 \begin{corollary}\label{corollary: Cartesian square of truncation maps}
     Let $f\colon E\to X$ be an $m$-truncated map in $\E$, and let $n>m$. Then the following square is Cartesian:
\[\begin{tikzcd}
    E \arrow[r, "{\eta_E}"] \arrow[d, "f"'] \arrow[dr, phantom, "\lrcorner"{anchor=center, pos=0.125}] & {\tau_n E} \arrow[d, "{\tau_nf}"] \\
    X \arrow[r, "{\eta_X}"'] & {\tau_n X}
\end{tikzcd}.\]
 \end{corollary}
 \begin{proof}
     We need to prove that in the diagram \eqref{prop: characterization n-connected maps through pullback functor, eq3}, the map $E\to E_{m+1}$ is an $(m+1)$-truncation. As it arises as a pullback of $\eta_X$, it is $(m+1)$-connected. Moreover $E_{m+1}$ is the domain of the $(m+1)$-truncated map $p_{m+1}\colon E_{m+1}\to\tau_{m+1}X$, hence is $(m+1)$-truncated by Proposition~\ref{prop: equivalent characterization of truncated maps 1}.
 \end{proof}
 The following corollary appears in Lurie \cite[6.5.1.2]{LurieHTT}, but its proof relies on a presentation of $\E$ as sheaves on a site $\E\simeq Sh(\C,\Sp)$. Our proof, by contrast, invokes Rasekh's work that is established in the framework of elementary $\infty$-topoi, hence argues from first principles.
\begin{prop}\label{prop: truncation functor preserve finite products}
    The truncation functor $\tau_n:\E\to\E$ preserves finite products.
\end{prop}
\begin{proof}
    By induction it is sufficient to show that binary products are preserved for any objects $X,Y$ in $\E$. Since $n$-truncated objects are closed under limits, the canonical map $\eta_X\times\eta_Y\colon X\times Y\to \tau_nX\times\tau_nY$ factors through \[\alpha\colon\tau_n(X\times Y)\to\tau_nX\times\tau_nY\]
    Since both domain and target are $n$-truncated, so is $\alpha$. Moreover the product of two $n$-connected maps is $n$-connected by \cite[Lem.~1.41]{Rasekh_ElementaryApproachTruncations}, so both $\eta_X\times\eta_Y$ and $\eta_{X\times Y}\colon X\times Y\to\tau_n(X\times Y)$ are $n$-connected. Using \cite[Prop.~1.35]{Rasekh_ElementaryApproachTruncations} it follows that $\alpha$ is $n$-connected as well. We conclude that $\alpha$ is an equivalence.
\end{proof}
We finish this section by a construction that generalizes classical $k$-invariants, which have been studied in a variety of contexts, for instance in stable homotopy theory \cite{Dugger&Shipley_PostnikovExtensionsOfRingSpectra}, for $(\infty,n)$-categories \cite{Harpaz&Nuiten&Prasma_kInvariantForNCategories}, for spherical presheaves \cite{Pstragowski_ModuliSpaceWithPrescribedHomotopyGroups}, or copersistent rational homotopy theory \cite{HessMaggsLavenir_CellDecompositionOfPersistentMinimalModels}.
\begin{prop}\label{prop: construction of k-invariants}
    Let $-2\leq k\leq n+1$ and $p\colon X\to Y$ an $n$-truncated and $(n-k)$-connected map in $\E$, where $X$ is $k$-connected. Then there exists an $(n+1)$-truncated and $(n-k+1)$-connected object $K\in\E_\ast^{\sg{n-k+1},\sleq n+1}$ together with a map $k\colon Y\to K$ whose fiber is $p$:
\[\begin{tikzcd}
    X \arrow[r] \arrow[d, "p"'] \arrow[dr, phantom, "\lrcorner"{anchor=center, pos=0.125}] & \ast \arrow[d] \\
    Y \arrow[r, "k"'] & K
\end{tikzcd}.\]
\end{prop}
\begin{proof}
    Let $P$ be the cofiber of $p$. We will prove that $K=\tau_{n+1}P$ and that $k$ is the composition $Y\to P\xrightarrow{\eta_P}\tau_{n+1}P$. Let $Q$ be the fiber of $Y\to P$, and consider the following pasting of Cartesian squares, together with the gap maps from $X$: 
\[\begin{tikzcd}
    X \arrow[r, "{(p,X)}", bend left=20] \arrow[rd, "{(p,X)'}"{description}, dashed] \arrow[rdd, "p"', bend right=40] & Q \arrow[r] \arrow[d] \arrow[dr, phantom, "\lrcorner"{anchor=center, pos=0.125}] & \ast \arrow[d] & \\
    & {Q'} \arrow[r] \arrow[d] \arrow[dr, phantom, "\lrcorner"{anchor=center, pos=0.125}] & {P\langle n+1 \rangle} \arrow[r] \arrow[d] \arrow[dr, phantom, "\lrcorner"{anchor=center, pos=0.125}] & \ast \arrow[d] \\
    & Y \arrow[r] & P \arrow[r, "{\eta_P}"'] & {\tau_{n+1} P}
\end{tikzcd}.\]
    We want to show that the composition $(p,X)'\colon X\to Q'$ is an equivalence. The section $\ast\to P\langle n+1\rangle$ is $n$-connected, so that $Q\to Q'$ is as well. By the Blakers-Massey theorem \cite[Cor.~4.3.1]{ABFJ_GeneralizedBlakersMassyTheorem} the gap map $(p,X)\colon X\to Q$ is $(n-k)+k=n$-connected. It follows that $(p,X)'$ is $n$-connected, being the composition of $n$-connected maps. Moreover $\ast\to \tau_{n+1}P$ is $n$-truncated, so that $Q'\to Y$ is as well. Since $p$ is $n$-truncated by hypothesis, $(p,X)'$ is as well by Proposition~\ref{prop: equivalent characterization of truncated maps 2}. Because the map $(p,X)'$ is both $n$-connected and $n$-truncated, it is an equivalence. 
    
    To finish the proof, we must show that $K=\tau_{n+1}P$ has the correct connectivity level. Since $p$ is $(n-k)\geq (-1)$-connected, its pushout $\ast\twoheadrightarrow P$ is as well. Consequently $P$ is connected and $\Omega P$ is $(n-k)$-connected. By Lemma~\ref{lemma: truncation and connectivity of loop space 2} $P$ is $(n-k+1)$-connected. In particular, its truncation $\tau_{n+1}P$ is also $(n-k+1)$-connected.
\end{proof}
\begin{remark}
    For $k=1$ we recover that a map $p\colon X\to Y$ between simply connected objects whose fiber is a $\B^nA$ is classified by a $k$-invariant $Y\to \B^{n+1}A$ \cite[Lemma~3.4.2]{May_MoreConciseAT}, although our proof has a very different nature.
\end{remark}

\section{$\infty$-Group objects}\label{section : group objects}
This section further develops the theory of $n$-group objects within an $\infty$-topos $\E$. Our primary goal is to characterize the interaction between their truncation levels and their associated classifying objects. We focus on the study of $k$-symmetric $n$-groups, which provides the appropriate framework for higher abelianness in the homotopical context. After establishing the general properties of sub-$n$-groups and normality, we prove that sufficiently symmetric sub-objects are necessarily normal. We conclude by the development of our primary examples of interest: the fundamental $n$-group and the $n$th-homotopy $1$-group of a pointed object. 
\subsection{$k$-Symmetric $n$-groups and $k$-fold classifying objects}
We begin by specializing the theory of groupoid objects to the case of objects with a single vertex. In this subsection, $\C$ is an $\infty$-category with finite products and realizations.
\begin{definition}\label{def: n-groups}
Let $0\leq n\leq\infty$. A simplicial object $G_\sbullet\colon\Delta^\op\to\C$ is:
    \begin{itemize}
        \item an \textit{$\infty$-group object} if $G_0\simeq\ast$ is terminal and $G_\sbullet$ is a groupoid object;
        \item an \textit{$n$-group object} if, in addition, $G_1$ is $(n-1)$-truncated. 
        \item The \textit{category of $n$-group objects} $\Grp_n(\C)$ is the full subcategory of $\Gpd(\C)$ spanned by $n$-group objects.
        \item The realization $\vert G_\sbullet\vert$ of an effective group object $G_\sbullet$ is its \textit{classifying object} and is denoted by $\B G\in\C$.
    \end{itemize}
\end{definition}
\begin{remark}
    By \cite[Lem.~2.20\, \&\, Cor.~2.24]{Beardsley&Peroux_KozulDualityInHigherTopoi}, there is an equivalence $\Grp_n(\C) \simeq \Grp(\tau_{n-1}\C)$ identifying $n$-group objects in $\C$ with $\infty$-group objects in $\tau_{n-1}\C$. While this indexing shift may seem counterintuitive, it perfectly recovers ordinary groups as $1$-groups. An $n$-group is simply a group object within an $n$-category.
\end{remark}
\begin{remark}
    The $\infty$-category $\Grp_n(\C)$ has all limits that $\C$ has. Indeed $\Grp(\C)$ has this property by \cite[Lem.~3.8]{SeverinBunk_ooBundlesAndSmthStringModels}, and $n$-truncated objects are closed under limits.
\end{remark}

   The classifying object of an effective group object is canonically pointed with an effective epimorphism $\ast\twoheadrightarrow\B G$. If $\C$ has effective groupoid objects, we obtain an equivalence of $\infty$-categories between $\Grp(\C)$ and the full subcategory of $\Fun([1],\C)$ spanned by the effective epimorphisms of the form $\ast\twoheadrightarrow B$, whose domain is a terminal object. In an $\infty$-topos, such effective epimorphisms correspond to pointed connected objects by \cref{prop: characterization connected maps}.
\begin{theorem}\label{theorem: group objects are pointed connected objects}
    Let $\E$ be an $\infty$-topos and $0\leq n\leq \infty$. There is an equivalence of $(n+1)$-categories
\[\begin{tikzcd}
	{\E_\ast^{\sg{0},\sleq n}} & {\Grp_n(\E)}
	\arrow[""{name=0, anchor=center, inner sep=0}, "\Omega", shift left=2, from=1-1, to=1-2]
	\arrow[""{name=1, anchor=center, inner sep=0}, "\B ", shift left=2, from=1-2, to=1-1]
	\arrow["\simeq"{description}, draw=none, from=0, to=1]
\end{tikzcd}\]
where $\Omega (\ast\twoheadrightarrow X)=\check{C}(\ast\xrightarrow{}X)\vert_{\Delta^\op}$, and $\B G_\sbullet \coloneq \vert G_\sbullet\vert$.
\end{theorem}
\begin{proof}
    For $n=\infty$ this is \cite[7.2.2.11]{LurieHTT}. One sees directly that this equivalence restricts to the case $n<\infty$ using \cref{cor: realization of n truncated groupoid is n+1 truncated} and Lemma~\ref{lemma: truncation and connectivity of loop space}.
\end{proof}
\begin{example}\label{example: classifying object of the loop space is the connected component of X}
     Let $X$ be a pointed object in $\C$. The \v{C}ech nerve of $\ast\to X$ defines an $\infty$-group structure on $\Omega X$. If $\C$ has effective groupoid objects and universal $\Delta^{op}$-shaped colimits, its classifying object is constructed with the effective epi/mono factorization (see \cref{rmk: construction of the epi mono factorization}):
     \[\ast\twoheadrightarrow\B\Omega X\hookrightarrow X.\]
     When $\C$ is an $\infty$-topos, this is the connected component of the base point in $X$. In particular if $X$ is connected, then $X\simeq \B(\Omega X)$.
\end{example}
\begin{example}\label{ex: BAut(F) as the image in the universe}
    Let $X$ be an object of an $\infty$-topos $\E$ and $\U$ be a universe with $\corner{X}\colon\ast\to \U$ a classifying map for $X$. The object $\Aut(X)\coloneq\Eq(X,X)$ of autoequivalences admits an $\infty$-group structure. This can be seen with the Univalence Axiom \ref{Univalence Axiom in a topos}, which yields in particular an equivalence \[\Omega_{\corner{X}}\U\simeq\Aut(X).\] By the previous \cref{example: classifying object of the loop space is the connected component of X}, its classifying object is the connected component of $\corner{X}\in \U$:\[\ast\twoheadrightarrow \B \Aut(X)\hookrightarrow \U.\]
    This allows us to compare the $\infty$-group of autoequivalences of a map $p\colon E\to X$ with the one of its fiber at a point $x\colon \ast\to X$.
\begin{corollary}\label{cor: properties of autoequivalence group of p}
    Let $x\colon\ast\to X$ be a pointed object of an $\infty$-topos, $p\colon E\to X$ an object of the slice $\slice{\E}{X}$ with fiber $F$, and $\Aut_{/X}(E)\in\Grp(\slice{\E}{X})$ the $\infty$-group of autoequivalences of $p$. 
    \begin{enumerate}
        \item The fiber of $\B \Aut_{/X}(E)\to X$ is $ \B \Aut(F)$ in $\E$;
        \item there is an equivalence $\B \Aut_{/X}(F\times X)\simeq \big(\B \Aut(F)\big)\times X$ in $\slice{\E}{X}$.
    \end{enumerate}
\end{corollary}
\begin{proof}
    Let $\U$ be a universe classifying $p$ with classifying map $\corner{E}\colon X\to\U$. \cref{lemma: universe in slice categories} tells us that $\corner{p}\coloneq(\corner{E}, 1_X)\colon X\to \U\times X$ classifies $p$ in the slice $\slice{\E}{X}$. We use that both effective epimorphisms and monomorphisms are stable by pullback. For $(1)$ pull back the epi/mono factorization $\corner{p}\colon (X\twoheadrightarrow \B \Aut_{/X}(E)\hookrightarrow\U\times X)$ in $\slice{\E}{X}$ along $x:\ast\to X$ to obtain an epi/mono factorization $\corner{F}\colon (\ast \twoheadrightarrow \B \Aut(F)\hookrightarrow\U)$ in $\E$. For $(2)$, pull back the factorization of $\corner{F}$ in $\E$ along $X\to \ast$ to obtain the desired factorization of $\corner{F\times X}$ in $\slice{\E}{X}$.
    \end{proof}
    \end{example}

    The hypotheses we impose in the following theorem ensure that $\infty$-group objects have classifying objects, and that effective epi/mono factorizations exist (\cref{rmk: construction of the epi mono factorization}).
    \begin{prop}\label{prop: product preserving functors preserves group objects}
        Let $\C$, $\C'$ be presentable $\infty$-categories with effective groupoids and universal $\Delta^\op$-shaped colimits. Let $L:\C\to\C'$ be a functor and $G_\sbullet$ an $\infty$-group in $\C$ with classifying object $\B G$.
        \begin{enumerate}
            \item If $L$ preserves finite products, $L\circ G_\sbullet$ is an $\infty$-group object in $\C'$, denoted $LG$.
            \item If $L$ preserves finite products and realizations, the classifying object of the $\infty$-group $LG$ is  \[\B (LG) \simeq L(\B G).\]
            \item If $L$ is left exact, the classifying object of the $\infty$-group $LG$ is the image of $L(\ast\to \B G)$: \[\ast\twoheadrightarrow \B (LG)\hookrightarrow L(\B G).\] 
 \end{enumerate}
    \end{prop}
\begin{proof}
    The first two points are precisely \cite[Prop.~3.5]{SeverinBunk_ooBundlesAndSmthStringModels}. Since $L$ preserves finite limits, $L\circ G_\sbullet$ is a groupoid object, and the following square is Cartesian:
\[\begin{tikzcd}
    {L(G)} \arrow[r] \arrow[d] \arrow[dr, phantom, "\lrcorner"{anchor=center, pos=0.125}] & \ast \arrow[d] \\
    \ast \arrow[r] & {L(\B G)}
\end{tikzcd}.\]
It follows that $L\circ G_\sbullet^+$ is the \v{C}ech nerve of $\ast\to L(\B G)$ by \cite[6.1.2.11]{LurieHTT}. By \cref{example: classifying object of the loop space is the connected component of X}, its classifying object is given by the effective epi/mono factorization.
\end{proof}
In an $\infty$-topos $\E$, the $n$-truncation functor $\tau_n:\E\to\E$ preserves finite products by \cref{prop: truncation functor preserve finite products}, which implies that $\tau_nG$ is an $\infty$-group. However $\tau_n$ is not left exact, nor does it preserve realizations. Nonetheless it behaves well with classifying objects:
\begin{corollary}\label{cor: classifying object of the truncation of the group}
    Let $\E$ be an $\infty$-topos, $G$ be an $\infty$-group object in $\E$, and $n\geq -2$. Then $\tau_nG$ is an $(n+1)$-group object in $\E$ with classifying object
    \[\tau_{n+1}(\B G)\simeq \B (\tau_nG).\]
\end{corollary}
\begin{proof}
    By \cref{prop: product preserving functors preserves group objects} $\tau_nG$ is an $\infty$-group. Moreover the unit $\eta\colon Id_{\E}\to\tau_n$ induces a natural transformation $G_\sbullet\to\tau_n\circ G_\sbullet$, and taking realizations it induces a map $\B \eta\colon \B G\to \B (\tau_nG)$. By \cref{cor: realization of n truncated groupoid is n+1 truncated} the target is $(n+1)$-truncated, hence $\B \eta$ factors through \[\alpha\colon\tau_{n+1}(\B G)\to \B (\tau_{n}G).\] We show that it is an equivalence by showing that it is both a monomorphism and an effective epimorphism. The latter is clear since the target is connected. To show that it is $(-1)$-truncated, we show that it is a connected map between connected objects whose looping is an equivalence, and conclude by Lemma~\ref{lemma: truncation and connectivity of loop space 3}. Its looping $\Omega\alpha$ is an equivalence:  \[\Omega(\tau_{n+1}(\B G))\simeq\tau_n(\Omega \B G)\simeq \tau_n(G)\xrightarrow[\simeq]{\Omega\alpha}\Omega \B (\tau_nG),\]
    where we used \cref{cor: loop space commutes with truncation}. In particular $\Omega \alpha$ is $(-1)$-connected. Lemma~\ref{lemma: truncation and connectivity of loop space 2} implies that $\alpha$ is a connected map, as desired.
\end{proof}

The category of \textit{abelian groups} can be identified with the category $\Grp(\Grp(\Set))$ of group objects in groups. We call them \textit{$2$-symmetric $1$-groups}, where the terminology is borrowed from \cite[Def.~2.3]{BuchholtzRijke_LESofHomotopynGroups}. This is sometimes called a discrete $\mathbb{E}_2$-group. We define $k$-symmetric $n$-groups following \cite[§1.6.3]{Lavenir_HiltonMilnorsTheoremInftyTopoi}. Lurie studies those, under the name of grouplike $\mathbb{E}_k$-objects, using the theory of $\infty$-operads \cite{LurieHigherAlgebra}.
\begin{definition}\label{def: k-symmetric n-groups}
    Let $\C$ be an $\infty$-category with finite limits, and let $1\leq k<\infty$, $0\leq n\leq \infty$. We define $k$-symmetric $n$-group objects inductively.
    \begin{itemize}
        \item A $1$-symmetric $n$-group object is an $n$-group object in $\C$. We denote the corresponding $\infty$-category by \[\Grp_{(n,1)}(\C)=\Grp_n(\C).\]
        \item A $k$-symmetric $n$-group object is an $n$-group object in $\Grp_{(n,k-1)}(\C)$. We denote the corresponding $\infty$-category by \[\Grp_{(n,k)}(\C)=\Grp(\Grp_{(n,k-1)}(\C)).\]
    \end{itemize}
\end{definition}
Unfolding the definition, a $k$-symmetric $n$-group is a $k$-fold simplicial object $G_{\sbullet,\cdots,\sbullet}\colon (\Delta^{op})^{\times k}\to\C$ whose $(n_1,\ldots ,n_k)$-th coordinate is a Cartesian product $G^{\times (n_1\cdots n_k)}$ for $G$ an $(n-1)$-truncated object.

In an $\infty$-topos, these correspond to pointed $(k-1)$-connected and $(n+k-1)$-truncated objects. 
\begin{theorem}\label{theorem: symmetric group objects are pointed highly connected objects}
    Let $\E$ be an $\infty$-topos, $0\leq n\leq \infty$ and $1\leq k<\infty$. There is an equivalence of $\infty$-categories
\[\begin{tikzcd}
	{\E_\ast^{\sg{ k-1},\sleq n+k-1}} & {\Grp_{(n,k)}(\E)}
	\arrow[""{name=0, anchor=center, inner sep=0}, "\Omega^k", shift left=2, from=1-1, to=1-2]
	\arrow[""{name=1, anchor=center, inner sep=0}, "\B ^k", shift left=2, from=1-2, to=1-1]
	\arrow["\simeq"{description}, draw=none, from=0, to=1]
\end{tikzcd},\]
where $\Omega^k$ is the iterated \v{C}ech nerve construction, and $\B^kG$ is the colimit of $G_{\sbullet,\cdots,\sbullet}\colon (\Delta^\op)^{\times k}\to \E$.
\end{theorem}
\begin{proof}
    For $n=\infty$ this is {\cite[Thm.~1.6.27]{Samuel_Thesis} or \cite[Thm. 5.2.6.15]{LurieHigherAlgebra}}. One sees directly that this equivalence restricts to the case $n<\infty$ using \cref{cor: realization of n truncated groupoid is n+1 truncated} and Lemma~\ref{lemma: truncation and connectivity of loop space}. This has also been established in \cite[Thm.~2.26]{Beardsley&Peroux_KozulDualityInHigherTopoi}.
\end{proof}
\begin{remark}
    Different cases of this equivalence have been established in the literature, as explained in \cite[Ex.~2.28-2.30]{Beardsley&Peroux_KozulDualityInHigherTopoi}. For example $2$-symmetric $2$-groups are known as braided categorical groups.
\end{remark}
 We write $G\coloneq G_{1,\ldots,1}$ for both the underlying object and the $k$-symmetric $n$-group $G_{\sbullet,\cdots,\sbullet}$. The object $\B^k G$ is called the \textit{$k$-fold classifying object} of $G$. For any $1\leq l\leq k$, a $k$-symmetric $n$-group possesses an \textit{underlying $l$-symmetric $n$-group}, forming a functor 
 \begin{equation}\label{def: underlying k symmetric n group}
    U_{k,l}\colon\Grp_{(n,k)}(\C)\to \Grp_{(n,l)}(\C).
    \end{equation}
    At the level of classifying object, it is given by the $(k-l)$-fold looping functor $\Omega^{k-l}\colon \B^kG\mapsto \B^l G$. In the case $l=1$ we recover the \textit{underlying $n$-group}. Moreover, any $k$-symmetric $n$-group can be seen as a higher-order group object with less symmetry. By considering the $(k-1)$-fold colimit of the $k$-fold simplicial object, we obtain a full subcategory:
 \begin{equation}\label{remark: (n k) groups} 
            \Grp_{(n,k)}(\C)\subseteq \Grp_{n+k-1}(\C),
    \end{equation}
defined by $G\mapsto (\B^{k-1} G\colon\Delta^\op\to\C)$. Indeed if $\B^k G$ is a $k$-fold classifying object, its loop object $\B^{k-1}G$ admits a group structure. In an $\infty$-topos $\E$, this corresponds to the inclusion $\E_\ast^{\sg {k-1},\sleq n+k-1}\subseteq \E_\ast^{\sg{0},\sleq n+k-1}$. 
\begin{definition}
    Let $l<k$ be integers. An $l$-symmetric structure on an $n$-group object $G$ in $\C$ \textit{lifts} to a $k$-symmetric structure if there exists a $k$-symmetric $n$-group whose underlying $l$-symmetric $n$-group is $G$.
\end{definition}
In an $\infty$-topos this amounts to the existence of a pointed $(k-1)$-connected $(n+k-1)$-truncated object $\B^kG$ with a pointed equivalence $\Omega^{k-l}\B^kG\simeq\B^l G$.

\medskip
An abelian group is in particular an infinite loop object. More generally the following proposition shows that the hierarchy of symmetry stabilizes once the degree of symmetry exceeds the truncation level of the group. Specifically, when $k > n$, the $\infty$-category $\Grp_{(n,k)}(\E)$ stabilizes, and $k$-symmetric $n$-groups become equivalent to $(n-1)$-truncated infinite loop objects (i.e., connective spectra or $\mathbb{E}_\infty$-groups).
\begin{prop}\label{prop: k-symmetric groups are k+1-symmetric for large k}
    Let $0 \leq n < k$ be integers, and let $G$ be an $n$-group object in $\E$. Any $k$-symmetric structure on $G$ lifts to a unique $(k+1)$-symmetric structure on $G$.
\end{prop}
\begin{proof}
    By \cref{prop: construction of k-invariants} $\tau_{n+k}(\Sigma \B ^kG)$ is a pointed $k$-connected $(n+k)$-truncated object such that $\Omega (\tau_{n+k}(\Sigma \B ^kG))\simeq \B ^kG$. This defines a $(k+1)$-fold classifying object $\B ^{k+1}G$. Such an object is unique as it corresponds to the classifying object of the $(n+k)$-group $\B^k G$.
\end{proof}
The following definition of sub-$n$-groups is generalized from Roberts \cite[Def.~4.82]{Roberts_2CoveringSpaces}, who considered the case $n=2$.
\begin{definition}\label{def: sub n groups}
    Let $n \geq 1$ be an integer and let $G \in \Grp_n(\E)$ be an $n$-group object.
    \begin{itemize}
        \item A \textit{sub-$n$-group of $G$} is an object $K \in \Grp_n(\E)$ together with a morphism $K \to G$ such that the induced map on classifying objects $\B K \to \B G$ is $(n-1)$-truncated. We denote the $\infty$-category of sub-$n$-groups of $G$ by $\Sub_n(G)$, which is a full subcategory of the slice $\slice{\Grp_n(\E)}{G}$.
        
        \item Let $k \geq 1$. A sub-$n$-group $K \to G$ is called a \textit{$k$-symmetric sub-$(n-k+1)$-group} if $K \simeq \B^{k-1}H$ for some $H \in \Grp_{(n-k+1,k)}(\E)\subseteq \Grp_n(\E)$. The full subcategory of such objects is denoted by $\Sub_{(n-k+1,k)}(G) \subseteq \Sub_n(G)$.
    \end{itemize}
\end{definition}
\begin{remark}\label{remark: on sub n groups}
        \begin{enumerate}
            \item We will primarily be interested in $n$-symmetric sub-$1$-groups of an $n$-group $G$. In particular, when $G = \Pi_n(X,x)$ is the fundamental $n$-group of a pointed object $(X,x)$, these correspond precisely to the sub-$1$-groups $A \leq \pi_n(X,x)$ of the $n$-th homotopy group (see \cref{theorem: n-symmetric subgroups of the fundamental n-group are 1-subgroups of n-th homotopy group}). For example, when $G=\B^{n-1} H$ for $H$ a $1$-group, then $n$-symmetric sub-$1$-groups are of the form $\B^{n-1}A$ for $A\leq H$ a sub-$1$-group.
            \item Let $f \colon A \to G$ be a morphism in $\Grp_{(n,k)}(\E)$. The underlying morphism of $n$-groups $f \colon A \to G$ defines a sub-$n$-group if and only if $f$, when viewed in $\Grp_{n+k-1}(\E)$ via \eqref{remark: (n k) groups}, defines a sub-$(n+k-1)$-group. Indeed, $\B f$ is $(n-1)$-truncated if and only if $\B^k f$ is $(n+k-2)$-truncated, by Lemma~\ref{lemma: truncation and connectivity of loop space 3}. 
        \end{enumerate}
\end{remark}
\begin{definition}
    Let $f\colon A\to G$ be a map of $\infty$-group objects in $\C$. The \textit{kernel} $\kernel(f)$ of $f$ is the fiber of $f$ in $\Grp(\C)$.
\end{definition}
\begin{prop}\label{prop: kernel is a sub n group}
    Let $1\leq n<\infty$ and $f\colon A\to G$ be a morphism between $n$-groups. Then its kernel $\kernel(f)$ is an $n$-group and the map $\kernel(f)\to A$ is a sub-$n$-group.
\end{prop}
\begin{proof}
    The classifying object $\B\kernel(f)$ is given by the fiber of $\B f$ in $\E_\ast^{\sg0}$. By \cref{prop: n-connected pointed spaces is coreflective in pointed spaces}, it is given by the connected component $K\langle0\rangle$ of the fiber $K$ of $\B f$ in $\E_\ast$:
\[\begin{tikzcd}
    K \arrow[r] \arrow[d] \arrow[dr, phantom, "\lrcorner", very near start] & \ast \arrow[d] \\
    {\B A} \arrow[r, "{\B f}"'] & {\B G}
\end{tikzcd}.\]
Since $\B f$ is a map between $n$-truncated objects, it is $n$-truncated. Hence its base change $K$, and thus $\B\kernel(f)$, are $n$-truncated as well. This proves that $\kernel(f)$ is an $n$-group. On the other hand $\ast\to\B G$ is $(n-1)$-truncated since $\B G$ is $n$-truncated. It follows that its base change $K\to \B A$ is $(n-1)$-truncated as well. Since $n\geq 1$ the connected component $\B\kernel(f)\hookrightarrow K$ is $(n-1)$-truncated. It follows that the composition $\B \kernel(f)\hookrightarrow K\to \B A$ is $(n-1)$-truncated as desired.
\end{proof}
A classical $1$-group homomorphism $A\to G$ is injective if and only if every sub-$1$-group $K\leq A$ lying in its kernel is trivial (i.e. is a $0$-group). We now prove an analogous characterization for sub-$n$-groups.
    \begin{prop}\label{prop: sub n group characterization}
        Let $1\leq n<\infty$ and $f\colon A\to G$ be a morphism of $n$-groups. Then it defines a sub-$n$-group if and only if for every sub-$n$-group $K\to A$ such that the composition $K\to A\to G$ is trivial, $K$ is an $(n-1)$-group.
    \end{prop}
    \begin{proof}
        Assume first that $f$ defines a sub-$n$-group, meaning that $\B f$ is $(n-1)$-truncated. Let $ K\to A$ be a sub-$n$-group mapping trivially to $G$. This yields a commutative square:
\[\begin{tikzcd}
    {\B K} \arrow[r] \arrow[d] & \ast \arrow[d] \\
    {\B A} \arrow[r, "{\B f}"'] & {\B G}
\end{tikzcd}.\]
The bottom map $\B f$ and the left map are $(n-1)$-truncated by hypothesis, so their composition $\B K\to\ast\to \B G$ is $(n-1)$-truncated as well. Moreover, $\ast\to \B G$ is $(n-1)$-truncated since $\B G$ is $n$-truncated. By Proposition~\ref{prop: equivalent characterization of truncated maps} this implies that $\B K\to \ast$ is $(n-1)$-truncated. Thus, $K$ is an $(n-1)$-group.

Conversely, we prove that $f$ defines a sub-$n$-group.  Consider the $(n-1)$-image of $\B f$, which by connectivity and truncation levels is the classifying object of some $n$-group $H$: \[\B f\colon \B A\xrightarrow{\B g}\B H\xrightarrow{\B h}\B G.\]
We will show that $\B g$ is $(n-1)$-truncated. Since $\B g$ is also $(n-1)$-connected by definition of the image factorization, it will follow that $\B g$ is an equivalence, implying that $\B f\simeq \B h$ is $(n-1)$-truncated, as desired. To this end, consider the kernels $\kernel(f)\to A$ and $\ker(h)\to H$. Both define sub-$n$-groups by \cref{prop: kernel is a sub n group}, and their respective compositions to $G$ are trivial. Our hypotheses imply that $\kernel(f)$ is an $(n-1)$-group, while the first part of this proof implies that $\ker(h)$ is an $(n-1)$-group (since $h$ defines a sub-$n$-group). Recall that $\B\kernel(h)$ is the connected component of the fiber $P$ of $\B h$ in $\E_\ast$, and $\B\kernel(f)$ is the connected component of the fiber $Q$ of $\B f$ in $\E_\ast$. We depict this in the following pasting of Cartesian squares in $\E_\ast$:
\[\begin{tikzcd}
    {\B\kernel(f)} \arrow[r] \arrow[d, hook]\arrow[dr, phantom, "\lrcorner"{anchor=center, pos=0.125}, near start] & {\B\kernel(h)} \arrow[d, hook] & \\
    Q \arrow[r] \arrow[d] \arrow[dr, phantom, "\lrcorner"{anchor=center, pos=0.125}, near start] & P \arrow[r] \arrow[d] \arrow[dr, phantom, "\lrcorner"{anchor=center, pos=0.125}, near start] & \ast \arrow[d] \\
    {\B A} \arrow[r, "\B g"'] & {\B H} \arrow[r, "\B h"'] & {\B G}
\end{tikzcd},\]
where the top square is Cartesian since $\B g$ is connected: indeed the total pullback object is a connected component of $Q$ because $\B\ker(h)$ is connected. As argued above, the domain and target of the top map $\B\kernel(f)\to\B\kernel(h)$ are $(n-1)$-groups, hence it is $(n-1)$-truncated. Because $\B H$ is connected, the middle vertical composition $\B\kernel(h)\to\B H$ is an effective epimorphism. We conclude that $\B g$ is also $(n-1)$-truncated by \cref{prop: pullback of truncated map is truncated}.
    \end{proof}
\subsection{Normal maps}\label{section: normal maps}
  In classical group theory, a homomorphism is normal if its quotient inherits a group structure. We study this notion for $\infty$-groups following the work of Farjoun and Segev for $1$-groups \cite{FarjounSegev_CrossedModulesasHomotopyNormalMaps}, Prasma for general loop spaces \cite{Prasma_HomotopyNormalMaps}, and Beardsley and Fox for $\infty$-groups in an $\infty$-topos \cite{BeardsleyFoxHigherGroupsHigherNormality}.
\begin{definition}\label{def: normal maps}
    Let $f \colon A \to G$ be a morphism of $\infty$-groups objects in $\E$. 
    \begin{enumerate}[label=\arabic*., ref=\thedefinition.\arabic*]
        \item\label{def: normal maps: quotient} The \textit{quotient} $G\sslash A\in\E$ of $f$ is the fiber of $\B f$.
        \item The map $f$ is \textit{normal} if there exists a pointed connected object $\B Q$ and a fiber sequence in $\E$:
        \[ \B A \xrightarrow{\B f} \B G \xrightarrow{\B \pi} \B Q.\]
        In this case, the quotient $G\sslash A$ is equivalent to $\Omega \B Q$ and inherits an $\infty$-group structure with classifying object $\B Q$. 
        \item Suppose that $G$ is an $n$-group. We write $\Sub_n^{\nor}(G)$ for the full subcategory of $\Sub_n(G)$ spanned by the normal sub-$n$-groups, and $\Sub^{\nor}_{(n-k+1,k)}(G)$ for the corresponding category of normal $k$-symmetric sub-$(n-k+1)$-groups.
    \end{enumerate}
\end{definition}
\begin{remark}\label{remark: on normal maps}
\begin{enumerate}[label=\arabic*., ref=\theremark.\arabic*]
    \item The compatible $\infty$-group structure on the quotient $G \sslash A$ (the \textit{normal structure} in \cite{FarjounSegev_CrossedModulesasHomotopyNormalMaps}, or \textit{normality datum} in \cite{BeardsleyFoxHigherGroupsHigherNormality}) is generally not unique, as the delooping $\B Q$ may not be uniquely determined by $\B f$. However, we show in \cref{cor: quotient of n-symmetric sub-1-groups are unique} that for $n$-symmetric sub-$1$-groups $\B^{n-1}A\to G$, this structure is unique.
    \item\label{remark: on normal maps 3}
    Farjoun and Segev \cite{FarjounSegev_CrossedModulesasHomotopyNormalMaps} prove that for $1$-groups, normality is equivalent to a crossed module structure on $f$. While we do not pursue it here, this suggests a potential theory of $\infty$-crossed modules for higher normal maps.
\end{enumerate}
\end{remark}
\begin{example}
Consider the following variation of \cite[Example~4.5]{BeardsleyFoxHigherGroupsHigherNormality}. Let $f$ be the composition $\B^2\Z \times \Z \xrightarrow{\pi_1} \B^2\Z \xrightarrow{\B^2(\cdot 2)} \B^2\Z$, which is a sub-$3$-group whose delooping is the $2$-truncated map $\B f \colon \B^3\Z \times \B\Z \to \B^3\Z$. Its quotient $Q$, arising as the fiber of $\B f$, is $\B^2(\Z/2) \times \B\Z$. There exist distinct spaces $F, F'$ arising as the fibers of two distinct maps $\B^2\Z \to \B^4(\Z/2)$ in $H^4(\B^2\Z; \Z/2) \cong \Z/2$, such that $\Omega F \simeq Q \simeq \Omega F'$. This demonstrates the non-uniqueness of normality data.
\end{example}
Subgroups of classical abelian groups are always normal. The following proposition generalizes this to the stable range.
    \begin{prop}\label{prop: sufficiently symmetric sub n groups are normal}
        Let $f\colon A\to G$ be a morphism in $\Grp_{(n,k)}(\E)$ with $k>n$. If the underlying $n$-group map defines a sub-$n$-group, then $f$ is normal. Moreover the quotient $G\sslash A$ admits a $k$-symmetric $n$-group structure whose $k$-fold classifying object is the $(n+k-1)$-truncation of the cofiber of $\B^k f$. 
     \end{prop}
     \begin{proof}
          By hypothesis $\B ^kA$ is $l\coloneq n$-connected and $\B f\colon \B A\to \B G$ is $(n-1)$-truncated. It follows from \cref{remark: on sub n groups} that $\B ^kf$ is $m\coloneq(n+k-2)$-truncated. Applying \cref{prop: construction of k-invariants}, $\B^k f$ is the fiber of a map $\B^k G \to K$, where $K$ is $(m+1)=(n+k-1)$-truncated and $(m-l+1)=(k-1)$-connected. This $K$ provides the $k$-fold delooping $\B^k(G \sslash A)$ for the quotient, exhibiting it as a $k$-symmetric $n$-group: 
        \[\begin{tikzcd}
            {\B^k A} \arrow[r] \arrow[d, "{\B^k f}"'] \arrow[dr, phantom, "\lrcorner", very near start] 
            & \ast \arrow[d] \\
            {\B^k G} \arrow[r] 
            & {\B^k(G \sslash A)}
        \end{tikzcd}.\]
    \end{proof}
\subsection{Homotopy groups and fundamental $n$-groups}\label{subsection: fundamental n group and homotopy groups}
This section defines the basic $\infty$-group arising from a pointed object $X$. They provide the core objects in the classification of $n$-covering maps in \cref{section : Coverings}.
\begin{definition}
    Let $\E$ be an $\infty$-topos, $(X,x)\in\E_\ast$ a pointed object, and $n\geq1$.
    \begin{enumerate}
        \item  The \textit{$n$-th homotopy group object} of $(X,x)$ is the $n$-symmetric $1$-group defined by \[\pi_n(X,x)\coloneq\Omega^n(\tau_nX)\in\Grp_{(1,n)}(\E).\]
        \item The \textit{fundamental $n$-group object} of $(X,x)$ is the $n$-group defined by \[\Pi_{n}(X,x)\coloneq\Omega(\tau_{n} X)\in\Grp_{(n,1)}(\E).\]
    \end{enumerate}
    When $n=0$, we define the $0$-th homotopy set as $\pi_0(X)\coloneq \tau_0(X)\in\E_\ast^{\sleq 0}$.
    \end{definition}

\begin{remark}\label{remark: classifying objects of homotopy groups fundamental n-groups}
\begin{enumerate}
    \item\label{remark: classifying objects of homotopy groups fundamental n-groups 1} The fundamental $n$-group $\Pi_n(X,x)$ has the connected component of the $n$-truncation as its classifying object (\cref{example: classifying object of the loop space is the connected component of X}): $\B\Pi_n(X,x) \simeq \tau_n X \langle 0 \rangle$.
    \item\label{remark: classifying objects of homotopy groups fundamental n-groups 2} For $1\leq k\leq n$, the $k$-fold classifying object of the underlying $k$-symmetric $1$-group of $\pi_n(X,x)$ (\eqref{def: underlying k symmetric n group}) is \[\B ^k\pi_n(X,x)=\Omega^{n-k}\left(\tau_nX\langle n-1\rangle\right)=\tau_k(\Omega^{n-k}X)\langle k-1\rangle.\]
\end{enumerate}
\end{remark}

\begin{remark}\label{remark: unpointed generalization}
    The definition of homotopy groups naturally generalizes to unpointed objects $X \in \E$ by working internally within the slice $\infty$-topos $\slice{\E}{X}$ \cite[6.5.1.1]{LurieHTT}. In this context, the unpointed homotopy group $\pi_n(X)$ is defined as the homotopy group of the diagonal $\Delta_X \colon X \to X \times X$ viewed as a pointed object over $X$. While we primarily rely on pointed homotopy groups for our classification of coverings, this unpointed generalization is essential for global connectivity results. Specifically, it allows one to test the connectivity of a space without assuming the existence of a basepoint $x \colon \ast \to X$, as in the following proposition.
\end{remark}
\begin{prop}[{\cite[Prop.~9.8]{Rezk_HomotopyTopos}}]
    Let $(X,x)$ be a pointed object of an $\infty$-topos $\E$ and $-1\leq n\leq \infty$. Then $X$ is $n$-connected if and only if $\pi_k(X,x)\simeq\ast\in\Grp(\E)$ for all $0\leq k\leq n$.
\end{prop}
When $G$ is an $\infty$-group, there are two canonical $1$-group structure on $\pi_n(G)$: the one coming from the $\infty$-group structure of $G$ (since $\pi_n$ preserves finite products), and the one coming from the $n$-th homotopy group construction.
\begin{corollary}\label{cor: two group structure on pi_n(G) coincide}
    Let $G$ be an $\infty$-group and $n\geq1$. The two $1$-group structures on $\pi_n(G)$ coincide.
   \end{corollary}
\begin{proof}
    By \cref{theorem: group objects are pointed connected objects} it suffices to show that the classifying objects coincide. Notice that $\pi_n$ is given by the composition \[\E_\ast\xrightarrow{\Omega^n}\E_\ast\xrightarrow{U}\E\xrightarrow{\tau_0}\E.\]
    The first functor $\Omega^n$ is left exact, $U$ preserves finite products and realizations, and $\tau_0$ preserves finite products. Using \cref{prop: product preserving functors preserves group objects} and \cref{cor: classifying object of the truncation of the group} we obtain that $\pi_n(G)$ is an $\infty$-group whose classifying object is: \[\tau_1 \left((\Omega^n \B G)\langle 0\rangle\right) \simeq \tau_1(\Omega^{n-1} G)\langle0\rangle,\]
    where we removed $U$ from the notation.
    By \cref{remark: classifying objects of homotopy groups fundamental n-groups} the classifying object of $\pi_n(G)$ coming from the group structure of $\pi_n$ is \[\tau_1(\Omega^{n-1}G)\langle0\rangle,\]
    which concludes the proof.
\end{proof}
A distinguished $n$-symmetric sub-$1$-group of $\Pi_n(X,x)$ is given by $\B^{n-1}\pi_n(X,x)$. Indeed, the counit of the $(n-1)$-connected cover functor of \cref{prop: n-connected pointed spaces is coreflective in pointed spaces} at $\B \Pi_n(X,x)$ is a pointed $(n-2)$-truncated map $\varepsilon\colon \B ^n\pi_n(X,x)\to \B \Pi_n(X,x)$. As we establish in \cref{theorem: n-symmetric subgroups of the fundamental n-group are 1-subgroups of n-th homotopy group}, it is the terminal $n$-symmetric sub-$1$-group: any other $\B^{n-1}A\to\Pi_n(X,x)$ factors through it. This terminality yields a correspondence between sub-$1$-groups of $\pi_n(X,x)$ and $n$-symmetric sub-$1$-groups of $\Pi_n(X,x)$. Moreover, by \cref{corollary: Cartesian square of truncation maps}, the map $\varepsilon$ arises from the following Cartesian square, where both horizontal arrows are $(n-1)$-truncations:
\begin{equation}\label{example: n symmetric sub 1 group of Pi_n(X)}
    \begin{tikzcd}
    {\B^n\pi_n(X,x)} \arrow[r] \arrow[d, "\varepsilon"'] \arrow[dr, phantom, "\lrcorner"{anchor=center, pos=0.125}] & \ast \arrow[d] \\
    {\B\Pi_n(X,x)} \arrow[r, "{\eta_X}"] & {\B\Pi_{n-1}(X,x)}
\end{tikzcd}.
\end{equation}
    This shows that $\B^{n-1}\pi_n(X,x)\to\Pi_n(X,x)$ is normal. The following theorem argues that the aforementioned correspondence restricts to normal maps.
\begin{theorem}\label{theorem: n-symmetric subgroups of the fundamental n-group are 1-subgroups of n-th homotopy group}
    Let $(X,x)\in\E_\ast$ be a pointed object and $n\geq 1$. 
    \begin{enumerate}[label=\arabic*., ref=\thelemma.\arabic*]
        \item\label{theorem: n-symmetric subgroups of the fundamental n-group are 1-subgroups of n-th homotopy group 1} There is an equivalence
        \begin{equation*}
        \Sub_1(\pi_n(X,x))\simeq  \Sub_{(1,n)}(\Pi_n(X,x))
    \end{equation*}
    between sub-$1$-groups of $\pi_n(X,x)$ and $n$-symmetric sub-$1$-groups of $\Pi_n(X,x)$.
    \item\label{theorem: n-symmetric subgroups of the fundamental n-group are 1-subgroups of n-th homotopy group 2} This equivalence restricts to an equivalence 
    \begin{equation*}
        \Sub_1^{\nor}\left(\pi_n(X,x)\right)\simeq \Sub^{\nor}_{(1,n)}\left(\Pi_n(X,x)\right)
    \end{equation*}
    between normal sub-$1$-groups of $\pi_n(X,x)$ and normal $n$-symmetric sub-$1$-groups of $\Pi_n(X,x)$.
    \end{enumerate}
\end{theorem}
\begin{proof}
    The proof of (1) follows from an adjunction argument, identifying the slice of $n$-symmetric $1$-groups over $\pi_n(X,x)$ with the slice over $\Pi_n(X,x)$. Part (2) then lifts the identification to normal maps. 
 
    (1)\,The adjunction 
\begin{tikzcd}
	{\E_\ast^{\sg{n-1}}} & {\E_\ast:-\langle n-1\rangle}
	\arrow[shift left, hook, from=1-1, to=1-2]
	\arrow[shift left, from=1-2, to=1-1]
\end{tikzcd} of \cref{prop: n-connected pointed spaces is coreflective in pointed spaces} induces an adjunction on slice $\infty$-categories (\cite[\href{https://kerodon.net/tag/02KD}{Tag 02KD}]{kerodon}, dual to \cite[5.2.5.1]{LurieHTT})
\[\begin{tikzcd}
	{\sum\limits_\varepsilon\colon\slice{\E_\ast^{\sg{n-1}}}{\B ^n\pi_n(X,x)}} & {\slice{\E_\ast}{\B \Pi_n(X,x)}:-\langle n-1\rangle}
	\arrow[""{name=0, anchor=center, inner sep=0}, shift left=2, hook, from=1-1, to=1-2]
	\arrow[""{name=1, anchor=center, inner sep=0}, shift left=2, from=1-2, to=1-1]
	\arrow["\dashv"{anchor=center, rotate=-90}, draw=none, from=0, to=1]
\end{tikzcd}\]
where the left adjoint $\sum_\varepsilon$ acts by postcomposition with the counit $\varepsilon\colon\B^n\pi_n(X,x)\xrightarrow{}\B\Pi_n(X,x)$. By restricting to the objects for which the counit is an equivalence, we obtain an equivalence 
    \begin{equation*}
        \sum_\varepsilon\colon\slice{\E^{\sg{n-1}}_\ast}{\B ^n\pi_n(X,x)}\xrightarrow[]{\simeq}\slice{\E_\ast^{\sg{n-1}}}{\B \Pi_n(X,x)}\,\,.
    \end{equation*}
    This restricts further to an equivalence on the full sub-$\infty$-categories of objects whose domains are, in addition, $n$-truncated:
     \begin{equation*}
          \sum_\varepsilon\colon\slice{\E_\ast^{\sg{n-1},\sleq n}}{\B ^n\pi_n(X,x)}\xrightarrow{\simeq}\slice{\E_\ast^{\sg{n-1},\sleq n}}{\B \Pi_n(X,x)}\,\,.
    \end{equation*}
     By definition this is an equivalence \begin{equation}\label{theorem: n-symmetric subgroups of the fundamental n-group are 1-subgroups of n-th homotopy group eq1}
     \slice{\Grp_{(1,n)}(\E)}{\pi_n(X,x)}\xrightarrow{\simeq} \slice{\Grp_{(1,n)}(\E)}{\Pi_n(X,x)}\,\, ,
     \end{equation}
     where the right hand side is the slice over $\Pi_n(X,x)\in\Grp_n(\E)$ restricted to the full subcategory $\Grp_{(1,n)}(\E)\subseteq\Grp_n(\E)$. To see that this equivalence further restricts to the subcategories of sub-$n$-groups, consider an $n$-symmetric $1$-group $A$. A pointed map $f\colon \B^nA\to\B^n\pi_n(X,x)$ is mapped to $g\colon \B^nA\to \B\Pi_n(X,x)$ via the following commutative triangle: 
\[\begin{tikzcd}
    {\B ^nA} && {\B \Pi_n(X,x)} \\
    & {\B ^n\pi_n(X,x)}
    \arrow["g", from=1-1, to=1-3]
    \arrow["f"', from=1-1, to=2-2]
    \arrow["\varepsilon"', from=2-2, to=1-3]
\end{tikzcd}.\]
Because $\varepsilon$ is $(n-2)$-truncated, $f$ is $(n-1)$-truncated if and only if $g$ is so. This ensures that the equivalence \eqref{theorem: n-symmetric subgroups of the fundamental n-group are 1-subgroups of n-th homotopy group eq1} restricts to:
     \begin{equation*}
        \Sub_{(1,n)}(\pi_n(X,x))\simeq \Sub_{(1,n)}(\Pi_n(X,x)).
    \end{equation*} 
   Lastly, \cite[7.2.2.12]{LurieHTT} identifies such maps $f$ with maps of $1$-groups $A\to\pi_n(X,x)$. By Lemma~\ref{lemma: truncation and connectivity of loop space}, $n$-symmetric sub-$1$-groups correspond exactly to classical sub-$1$-groups. Consequently, we can identify \[\Sub_{(1,n)}(\Pi_n(X,x))\simeq\Sub_1(\pi_n(X,x)),\] completing the proof.

   \medskip
   
   (2)\, We need to show that a sub-$1$-group $A\to\pi_n(X,x)$ is normal if and only if the corresponding $n$-symmetric sub-$1$-group $\B^{n-1}A\to\Pi_n(X,x)$ is normal. Suppose first that $\B^{n-1}A\to\Pi_n(X,x)$ is normal. By definition there exists a pointed connected $n$-truncated object $\B Q$ together with a quotient map $\B q\colon \B\Pi_n(X,x)\to\B Q$ with fiber $\B^nA$. We can form the $(n-1)$-connected cover of the basepoint, given by the $(n-2)$-connected/$(n-2)$-truncated factorization $\ast\to \B^n\pi_n(\B Q)\to\B Q$. Since those classes are stable under base change, pulling this factorization back along the quotient map $\B q$ recovers the $(n-1)$-connected cover of $\B\Pi_n(X,x)$: 
\[\begin{tikzcd}
    {\B^nA} \arrow[r] \arrow[d] \arrow[dr, phantom, "\lrcorner", very near start] 
    & {\B^n\pi_n(X,x)} \arrow[r] \arrow[d] \arrow[dr, phantom, "\lrcorner", very near start] 
    & {\B\Pi_n(X,x)} \arrow[d, "\B q"] \\
    \ast \arrow[r] 
    & {\B^n\pi_n(\B Q)} \arrow[r] 
    & {\B Q}
\end{tikzcd}.\]
In particular $\B^nA\to\B^n\pi_n(X,x)$ is normal with quotient $\B^n\pi_n(\B Q)$.

   Conversely, suppose that $A\to \pi_n(X,x)$ is normal with quotient $q\colon \pi_n(X,x)\to\pi_n(X,x)\sslash A$. Since these are $n$-symmetric $1$-groups, we can form the pushout $P$ of $\B^nq$ along $\varepsilon$:
\[\begin{tikzcd}
    {\B^nA} \arrow[r] \arrow[d] \arrow[dr, phantom, "\lrcorner", very near start] 
    & {\B^n\pi_n(X,x)} \arrow[d, "{\B^nq}"] \arrow[r, "\varepsilon"] \arrow[dr, phantom, "\ulcorner", very near end] 
    & {\B\Pi_n(X,x)} \arrow[d, "f"] \\
    \ast \arrow[r, two heads] 
    & {\B^n(\pi_n(X,x)\sslash A)} \arrow[r] 
    & P
\end{tikzcd},\]
where the left square is Cartesian and the right square is coCartesian.
Because $\B^nq$ is $(n-1)$-connected (as $\B^n A$ is so), its cobase change $f$ must be $(n-1)$-connected as well. This implies that $f$ induces an equivalence on the fundamental $(n-1)$-groups: \[\tau_{n-1}f\colon \B\Pi_{n-1}(X,x)\xrightarrow{\simeq}\B\Pi_{n-1}\left(P,f(x)\right).\] It follows that the fiber of the composition $\eta_P\circ \tau_nf\simeq \tau_{n-1}f\circ\eta_X$ is equivalent to the fiber of $\eta_X$, which is precisely $\varepsilon$ by \eqref{example: n symmetric sub 1 group of Pi_n(X)}. We can thus assemble the following pasting of Cartesian squares:
\[\begin{tikzcd}
    {\B^n\pi_n(X,x)} \arrow[r] \arrow[d, "{\varepsilon}"] \arrow[dr, phantom, "\lrcorner", very near start] 
    & {\B^n\pi_n(P, f(x))} \arrow[r] \arrow[d] \arrow[dr, phantom, "\lrcorner", very near start] 
    & \ast \arrow[d] \\
    {\B\Pi_n(X,x)} \arrow[r, "{\tau_n f}"] 
    & {\B\Pi_n\left(P,f(x)\right)} \arrow[r, "{\eta_P}"] 
    & {\B\Pi_{n-1}(P,f(x))}
\end{tikzcd}.\]
We are left to show that $\pi_n(P,f(x)) \simeq \pi_n(X,x)\sslash A$, which would exhibit $\Pi_n(P,f(x))$ as the quotient of $\B^nA\to\B\Pi_n(X,x)$. To this end, we construct an $n$-connected map $\varphi \colon P\langle n-1\rangle \to \B^n(\pi_n(X,x)\sslash A)$, which immediately induces the desired equivalence on $n$-truncations: 
\[\tau_n\varphi\colon \B ^n\pi_n(P, f(x))\xrightarrow{\simeq} \B ^n(\pi_n(X,x)\sslash A).\] 
To construct $\varphi$, recall from \cref{remark: contruction of the n-connected cover} that the $(n-1)$-connected cover of $P$ arises as the fiber of the $(n-1)$-truncation $\eta_P$, forming the fiber sequence: \[P\langle n-1\rangle\to P \xrightarrow{\eta_P} \B\Pi_{n-1}(P,f(x)).\] Since $P$ is a pushout and pushouts are universal in $\E$, the fiber $P\langle n-1\rangle$ is the pushout of the respective fibers, which we compute now. Because $(n-1)$-connected objects are left-orthogonal to $(n-1)$-truncated ones, any map from an object of the form $\B^n G$ to $\B\Pi_{n-1}(X,x)$ is nullhomotopic. This implies that its fiber splits as a product $\Pi_{n-1}(X,x)\times \B^n G$ (\cref{rmk: fiber of a nullhomtopic map is the product of the loop of the base}), yielding the following pushout square of fibers:
    \[\begin{tikzcd}
        {\Pi_{n-1}(X,x)\times \B^n\pi_n(X,x)} \arrow[r] \arrow[d] 
        & {\B^n\pi_n(X,x)} \arrow[d] \\
        {\Pi_{n-1}(X,x) \times \B^n(\pi_n(X,x)\sslash A)} \arrow[r] 
        & {P\langle n-1\rangle} \arrow[ul, phantom, "\ulcorner", very near start]
    \end{tikzcd}.\]
    Its universal property naturally yields the comparison map $\varphi \colon P\langle n-1\rangle \to \B^n(\pi_n(X,x)\sslash A)$. To conclude, we must show that the fiber of $\varphi$ is $n$-connected. Again, by universality of pushouts, the fiber of $\varphi$ is the pushout of the fibers of the constituent maps into $\B^n(\pi_n(X,x)\sslash A)$, yielding the following coCartesian square:
    \[\begin{tikzcd}
        {\Pi_{n-1}(X,x)\times \B^n A} \arrow[r] \arrow[d] 
        & {\B^n A} \arrow[d] \\
        {\Pi_{n-1}(X,x)} \arrow[r] 
        & {\operatorname{fib}(\varphi)} \arrow[ul, phantom, "\ulcorner", very near start]
    \end{tikzcd}.\]
    This exactly defines the join:
    \[ \operatorname{fib}(\varphi) \simeq \B^nA \ast \Pi_{n-1}(X,x). \]
    Since $\B^nA$ is $(n-1)$-connected and $\Pi_{n-1}(X,x)$ is $(-1)$-connected, the dual Blakers-Massey theorem \cite[Cor.~3.5.2]{ABFJ_GeneralizedBlakersMassyTheorem} bounds the connectivity of the join:
    \[\operatorname{conn}(\operatorname{fib}(\varphi))\geq (n-1)-1+2=n.\] Therefore, the fiber is $n$-connected, which implies that $\varphi$ is as well, and the proof is complete.
\end{proof}
\begin{remark}\label{remark: n symmetric sub 1 group of an n group G}
    The previous theorem is stated specifically for the fundamental $n$-group $G=\Pi_n(X,x)$. Since every $n$-group $G$ is of this form, namely $G\simeq \Pi_n(\B G,e)$, we could equivalently express the result for an arbitrary $n$-group $G$ as the equivalence: \[\Sub_1(\pi_n(\B G, e))\simeq \Sub_{(1,n)}(G),\]
    along with the corresponding restriction to normal maps.
\end{remark}
\begin{remark}\label{remark: construction of the n group quotient by a sub 1 group}
 Let $X\in\Sp$ and let $A\to\pi_n(X,x)$ be a sub-$1$-group with $n\geq 2$. Classically, if $X$ is simple, one can construct a delooping of the quotient $n$-group $\Pi_n(X)\sslash\B^{n-1}A$ as follows: the $n$-truncation $\tau_n X$ is the fiber of the $k$-invariant $\tau_{n-1}X\to\B^{n+1}\pi_n(X)$. Post‑composing this $k$-invariant with the quotient map $\B^{n+1}\pi_n(X)\to\B^{n+1}(\pi_n(X)/A)$ and taking the fiber yields the desired object $\B(\Pi_n(X)\sslash\B^{n-1}A)$, as can be verified by a straightforward pasting of Cartesian squares. For general $X$ this construction fails because the $k$-invariant may not exist. \cref{theorem: n-symmetric subgroups of the fundamental n-group are 1-subgroups of n-th homotopy group} shows that in our framework the $n$-symmetric sub-$1$-group $\B^{n-1}A\to\Pi_n(X)$ is normal, and therefore the quotient $n$-group $\Pi_n(X)\sslash\B^{n-1}A$ exists in the $\infty$-topos, without any simplicity hypothesis.
\end{remark}
    We place a strong emphasis on $n$-symmetric sub-$1$-groups of $n$-groups for the following reason. While an arbitrary pointed $n$-covering map $E\to X$ corresponds to a sub-$n$-group $\Pi_n(E,e)\to\Pi_n(X,x)$, the most interesting results about deck transformations (\cref{theorem: action of Deck is n-free}, \ref{cor: group deck is the quotient of fundamental n-group}, \ref{theorem: characterization normal n-covering maps}) require the total space $E$ to be $(n-1)$-connected. Under this connectivity assumption, the fundamental $n$-group of $E$ simplifies to $\Pi_n(E,e)\simeq\B^{n-1}\pi_n(E,e)$, so that the sub-$n$-group becomes an $n$-symmetric sub-$1$-group $\B^{n-1}\pi_n(E,e)\to\Pi_n(X,x)$. Via \cref{theorem: n-symmetric subgroups of the fundamental n-group are 1-subgroups of n-th homotopy group}, these are identified with ordinary sub-$1$-groups of $\pi_n(X,x)$. An enjoyable consequence is that for $n\geq 2$, every $(n-1)$-connected $n$-covering map is automatically normal (\cref{cor: classification normal n-covering maps}).
 
\section{$\infty$-Group actions}\label{section: infinty group actions}
As explained in the introduction, the classification theorem of $n$-coverings establishes a correspondence between $n$-covering maps over $(X,x)\in\E_\ast$ and $\infty$-actions of the fundamental $n$-group $\Pi_n(X,x)$ on $(n-1)$-truncated objects. In this section, we study the general theory of $\infty$-actions, with a particular focus on this truncated setting. 

We begin in \cref{subsection: group actions : background} by recalling definitions and establishing basic properties and examples. We continue in \cref{subsection: relation to truncation and homotopy groups} by analyzing the behavior of $\infty$-actions under truncation and homotopy group functors. In \cref{subsection: n free n transitive and n faithful actions}, we introduce higher analogs of classical group action properties: $n$-transitive, $n$-free, and $n$-faithful $\infty$-actions. We then relate the transitivity and freeness of an $\infty$-action of a product $\infty$-group in \cref{subsection: action from a product}, and finally investigate fixed points in \cref{subsection: fixed points}.
\subsection{Background}\label{subsection: group actions : background}
For the foundational theory of $\infty$-actions, we follow \cite{NSS_PrincipalBundlesGeneralTheory, SeverinBunk_ooBundlesAndSmthStringModels}.
 \begin{definition}
   Let $\C$ be an $\infty$-category with pullbacks and $G_\sbullet\in\Grp(\C)$ be an $\infty$-group.
   \begin{enumerate}
       \item A \textit{$G$-action in $\C$} is a groupoid object $(X\sslash G)_\sbullet\in\Gpd(\C)$ equipped with a Cartesian natural transformation $(X\sslash G)_\sbullet\to G_\sbullet$. Explicitly, for every morphism $[i]\to[j]$ in $\Delta^\op$, the following square is Cartesian:
\[\begin{tikzcd}
    {(X\sslash G)_i} \arrow[r] \arrow[d] \arrow[dr, phantom, "\lrcorner"{anchor=center, pos=0.125}, near start] & {G_i} \arrow[d, "{G_\sbullet\alpha}"'] \\
    {(X\sslash G)_j} \arrow[r] & {G_j}
\end{tikzcd}.\]
We refer to this as a \textit{$G$-action on the object $X\coloneq(X\sslash G)_0$}, denoted $G\curvearrowright X$.
A \textit{morphism of $G$-actions} (or a \textit{$G$-equivariant map from $X$ to $Y$}) is a map of groupoid objects $(X\sslash G)_\sbullet\to (Y\sslash G)_\sbullet$ over $G_\sbullet$. The \textit{$\infty$-category of $G$-actions} is defined as \[\Act_G(\C)\coloneq\slice{\text{Cart}(\Gpd(\C))}{G_\sbullet}.\]
We define $\Act_G(X)\coloneq \Act_G(\C)\times_\C\{X\}$ the $\infty$-category of \textit{$G$-actions on a fixed object $X$}.
\item The \textit{quotient} $X\sslash G$ of a $G$-action $G\curvearrowright X$ is the realization $\colim_{\Delta^\op} (X\sslash G)_\sbullet$ in $\C$. The quotient augments the simplicial object to a functor $(X\sslash G)_\sbullet\colon \Delta^\op_+\to \C$ where $(X\sslash G)_{-1}=X\sslash G$. 
\end{enumerate}
    \end{definition}
    \begin{remark}
        \begin{enumerate}
            \item There are canonical identifications $(X\sslash G)_n\simeq G_n\times X\simeq G^{\times n}\times X$. The components of the Cartesian map become the Cartesian projections $G^{\times n}\times X\to G^{\times n}$, and the face map $d_0\colon G\times X\to X$ is the projection. The other face map $d_1\eqqcolon\rho\colon G\times X\to X$ is the \textit{action map}.
            \item By the pasting law for pullbacks, any morphism of $G$-actions $(X\sslash G)_\sbullet\to (Y\sslash G)_\sbullet$ is necessarily Cartesian.
        \end{enumerate}
    \end{remark}
     When $\C=\E$ is an $\infty$-topos, the $\infty$-group $G_\sbullet$ is effective and corresponds to its classifying object $\ast\twoheadrightarrow\B G$. Similarly, the $G$-action groupoid $(X\sslash G)_\sbullet$ is effective and can be recovered as the \v{C}ech nerve of its quotient map $X\twoheadrightarrow X\sslash G$. It follows that the Cartesian transformation $(X\sslash G)_\sbullet\to G_\sbullet$ can be recovered from the following commutative square:
\[\begin{tikzcd}
	X & \ast \\
	{X\sslash G} & {\B G}
	\arrow[from=1-1, to=1-2]
	\arrow[two heads, from=1-1, to=2-1, "q"]
	\arrow[two heads, from=1-2, to=2-2, "\ast_{\B G}"]
	\arrow[from=2-1, to=2-2]
\end{tikzcd}.\]
By descent, this square is Cartesian. Since a Cartesian morphism $q\to\ast_{\B G}$ is completely determined by its codomain map $X\sslash G\to\B G$, this motivates the following result. Although it has been established multiple times in the context of higher topoi, our exposition differs from the existing literature \cite[Prop.~I.6.19]{Samuel_Thesis}, \cite[Prop. 5.1.267]{Schreiber_DifferentialCohomologyCohesiveInfinityTopos}.
\begin{theorem}\label{theorem: characterization of G-actions}
    Let $\C$ be a presentable $\infty$-category with effective groupoid objects, and satisfying descent along universal $\Delta^\op$-shaped colimits. Let $G$ be an $\infty$-group object in $\C$. There is an equivalence of $\infty$-categories: 
    \[\Act_G(\C)\simeq \slice{\C}{\B G}.\]
\end{theorem}
\begin{proof}
    By hypothesis, $\C$ has effective groupoids, yielding an equivalence
    \begin{equation}\label{prop: characterization of pointed action eq1}
        \Gpd(\C)\xrightarrow[\colim]{\simeq}\Eff(\C).
    \end{equation}
   We claim that this restricts to an equivalence between the wide sub-$\infty$-categories of Cartesian morphisms:
    \[\Cart(\Gpd(\C))\xrightarrow[\colim]{\simeq}\Cart(\Eff(\C)).\]
    To verify this, one only needs to check that the colimit functor \eqref{prop: characterization of pointed action eq1} preserves and reflects Cartesian morphisms. This translates to verifying that a natural transformation $A_\sbullet\to B_\sbullet$ in $\Gpd(\C)$ is Cartesian if and only if the corresponding commutative square in $\Eff(\C)$ is Cartesian:
\[\begin{tikzcd}
    {A_0} \arrow[r, two heads] \arrow[d] & {\colim_{\Delta^\op}A} \arrow[d] \\
    {B_0} \arrow[r, two heads] & {\colim_{\Delta^\op}B}
\end{tikzcd}.\]
     The \textit{only if} direction is precisely the definition of descent, while the \textit{if} direction follows from universality (see \cite[Prop.~3.8]{NSS_PrincipalBundlesGeneralTheory} for the complete argument in a related setting). We can now establish a sequence of equivalences to conclude:
    \begin{align*}
        \Act_G(\C)&\coloneq\slice{\Cart\left(\Gpd(\C)\right)}{G_\sbullet}\\
        &\simeq\slice{ \Cart\left(\Eff(\C)\right)}{(\ast\twoheadrightarrow \B G)}\\
        &\simeq \slice{\C}{\B G}.
    \end{align*}
    The final equivalence simply translates the fact that the data of a Cartesian morphism into $\ast\twoheadrightarrow \B G$ is uniquely determined by the target map $X\sslash G\to\B G$.
\end{proof}
The hypotheses of the previous theorem were chosen to be as minimal as possible, precisely so that the result applies to the pointed $\infty$-category $\E_\ast$ of an $\infty$-topos. This is essential for studying loop space and homotopy group constructions on $\infty$-actions, since those functors require pointed objects. Moreover, whereas the classification of $n$-covering is in terms of $\infty$-actions, pointed $n$-coverings correspond to pointed $\infty$-actions (\cref{Theorem: representation theorem for n-covers}) and, in particular, to subgroups (\cref{corollary: pointed coverings correspond to subgroups}), which are inherently pointed.
\begin{corollary}\label{cor: characterization of pointed action}
Let $\E$ be an $\infty$-topos, and let $\C$ denote either $\E$ or $\E_\ast$. Let $G$ be an $\infty$-group object in $\E$. Then there is an equivalence of $\infty$-categories:
     \[\Act_G(\C)\simeq \slice{\C}{\B G}.\]
\end{corollary} 
\begin{proof}
     Note first that $G$ is also a group object in $\E_\ast$ by \cite[7.2.2.10]{LurieHTT}. By \cref{cor: pointed topos has nice properties}, both $\E$ and $\E_\ast$ satisfy the required properties of \cref{theorem: characterization of G-actions}.
\end{proof}
In \cite{SeverinBunk_ooBundlesAndSmthStringModels}, Bunk defines $\infty$-actions as specific simplicial objects of the form $(X\times G^{\times\sbullet})$ and demonstrates that they are, in particular, groupoid objects \cite[Thm.~3.19]{SeverinBunk_ooBundlesAndSmthStringModels}. This observation implies that a functor need only preserve finite products, rather than all pullbacks, to preserve $\infty$-actions. The hypotheses we impose on $\C$ and $\C'$ in the following proposition ensure that the correspondence of \cref{theorem: characterization of G-actions} holds, and that effective epi/mono factorizations exist. This result serves as the $\infty$-action analog to \cref{prop: product preserving functors preserves group objects}.
\begin{prop}\label{prop: product preserving functors preserves actions}
        Let $\C$, $\C'$ be presentable $\infty$-categories with effective groupoid objects, and satisfying descent along universal $\Delta^\op$-shaped colimits. Let $L:\C\to\C'$ be a functor, $G$ an $\infty$-group in $\C$ with classifying object $\B G$, and $G\curvearrowright X$ an $\infty$-action in $\C$ with groupoid object $(X\sslash G)_\sbullet$.
        \begin{enumerate}
            \item If $L$ preserves finite products, $L\circ (X\sslash G)_\sbullet$ is an $\infty$-action, denoted $LG\curvearrowright LX$.
            \item If $L$ preserves finite products and realizations, the quotient of the $LG$-action on $LX$ is \[LX\sslash LG \simeq L(X\sslash G).\]
            \item If $L$ is left exact, the quotient of the $LG$-action on $LX$ is given by the image of $L(X\to X\sslash G)$: \[LX\twoheadrightarrow LX\sslash LG\hookrightarrow L(X\sslash G).\] In particular this yields the following pasting of Cartesian squares:
\[\begin{tikzcd}
    LX \arrow[r, two heads] \arrow[d] \arrow[dr, phantom, "\lrcorner"{anchor=center, pos=0.125}, near start] & {LX\sslash LG} \arrow[r, hook] \arrow[d] \arrow[dr, phantom, "\lrcorner"{anchor=center, pos=0.125}, near start] & {L(X\sslash G)} \arrow[d] \\
    \ast \arrow[r, two heads] & {\B(LG)} \arrow[r, hook] & {L(\B G)}
\end{tikzcd}.\]
 \end{enumerate}  
    \end{prop}
    \begin{proof}
       The first two points are established in \cite[Thm.~3.32]{SeverinBunk_ooBundlesAndSmthStringModels}. The proof of the third point proceeds identically to that of \cref{prop: product preserving functors preserves group objects}.
    \end{proof}
      \begin{example}\label{example: group actions}
      We work in an $\infty$-topos $\E$.
    \begin{enumerate}[label=\arabic*., ref=\theexample.\arabic*]
        \item\label{example: group actions: loop space action} Suppose that $X$ is pointed and connected. Then a fiber sequence $F\to E\to X$ corresponds to an $\infty$-action $\Omega X\curvearrowright F$ with quotient $E\simeq F\sslash \Omega X$.
        \item Recall that $\B\Aut(X)$ is the image of the classifying map $\corner{X}\colon \ast\to \U$. Pulling this back along the univalent map $p\colon\U_\ast\to\U$ yields: 
 \[\begin{tikzcd}
    X \arrow[r] \arrow[d] \arrow[dr, phantom, "\lrcorner"{anchor=center, pos=0.125}, near start] & {X\sslash \Aut(X)} \arrow[r, hook] \arrow[d] \arrow[dr, phantom, "\lrcorner"{anchor=center, pos=0.125}, near start] & {\U_\ast} \arrow[d, "p"] \\
    \ast \arrow[r, two heads] & {\B\Aut(X)} \arrow[r, hook] & \U
\end{tikzcd}.\]
        The left square defines the \textit{evaluation} $\infty$-action $\Aut(X)\curvearrowright X$. The middle vertical quotient map $X\sslash \Aut(X)\to\B\Aut(X)$ is univalent and classifies $X$-bundles.
        \item\label{example: group actions: as a map in Aut(X)} Let $q\colon X\sslash G\to \B G$ be the classifying map of a $G$-action, which is itself classified by a map $\corner{q}\colon \B G\to \U$. Because the fiber $q$ is $X$ and $\ast\twoheadrightarrow\B G$ is an effective epi, it factors through the subuniverse $\B\Aut(X)$. It follows that $q$ is classified by a pointed map $\B G\to\B\Aut(X)$. This motivates the following equivalence:
        \[\Act_G(X)^\simeq \simeq \E_\ast(\B G,\B\Aut(X)).\]
        See for example \cite[Prop.~III.2.30]{Samuel_Thesis} for more details.
        \item Every $\infty$-group $G$ acts trivially on the point $\ast$, with quotient $\ast\sslash G\simeq\B G$.
        \item More generally there is a \textit{trivial} $G$-action on $X$ for all $\infty$-group $G$ and object $X$, uniquely determined by the condition that the action map is the projection $\rho=\pi_2\colon G\times X\to X$. Consequently, its $\infty$-action-groupoid is $(X\sslash G)_\sbullet\simeq G_\sbullet\times X_\sbullet$, where $X_\sbullet$ is the constant simplicial object. Thus, an $\infty$-action is trivial if and only if its classifying map is the projection $X\sslash G\simeq \B G\times X\xrightarrow{\pi_1}\B G$. This happens if and only if $\B G\to\B \Aut(X)$ is nullhomotopic, that is, it factors through the point $\ast$. 
        \item There is a \textit{left multiplication} $\infty$-action $G\curvearrowright G$, whose classifying map is $\ast\to \B G$.
        \item\label{example: group actions: restricted action}
        Let $f\colon A\to G$ be a morphism of $\infty$-groups, and $G\curvearrowright X$ an $\infty$-action. There is a \textit{restricted} $\infty$-action $A\curvearrowright X$ defined by the following Cartesian squares:
     \[\begin{tikzcd}
        X \arrow[r] \arrow[d] \arrow[dr, phantom, "\lrcorner"{anchor=center, pos=0.225}, near start] & {X\sslash A} \arrow[r] \arrow[d] \arrow[dr, phantom, "\lrcorner"{anchor=center, pos=0.125}, near start] & {X\sslash G} \arrow[d] \\
        \ast \arrow[r] & {\B A} \arrow[r, "\B f"'] & {\B G}
    \end{tikzcd}.\] 
     \item\label{example: group actions: action of a map A-> G} Let $f\colon A\to G$ be a morphism of $\infty$-groups. The restriction of the left multiplication $G$-action on itself along $f$ yields an $\infty$-action $A\curvearrowright G$. Its quotient $G\sslash A$ is the fiber of $\B f$. This definition coincides with the quotient of $f$ of Definition~\ref{def: normal maps: quotient}. Moreover the fiber sequence $G\sslash A\to\B A\to \B G$ defines a $G$-action on $G\sslash A$ with quotient $\B A$.
    \end{enumerate}
\end{example}
Using Example \ref{example: group actions: as a map in Aut(X)}, one defines the following notion:
\begin{definition}\label{def: group map respects actions}
    Let $G\curvearrowright X$ and $H\curvearrowright X$ be $\infty$-actions on a fixed object $X$. An $\infty$-group map $f\colon G\to H$ \textit{respects the $\infty$-actions on $X$} if the following triangle 
\[\begin{tikzcd}
	G && H \\
	& {\Aut(X)}
	\arrow["f", from=1-1, to=1-3]
	\arrow[from=1-1, to=2-2]
	\arrow[from=1-3, to=2-2]
\end{tikzcd}\]
commutes in $\Grp(\E)$.
\end{definition}

\subsection{Relation with truncations and homotopy groups}\label{subsection: relation to truncation and homotopy groups}
    We now analyze how $\infty$-actions interact with truncation functors. If $X$ is an $(n-1)$-truncated object, any $G$-action on $X$ naturally factors through an $\infty$-action of the $n$-group $\tau_{n-1}G$, allowing us to restrict our attention to the latter.
    \begin{prop}\label{prop: actions on $n$-truncated objects corresponds to actions of an n-group}
         Let $X$ be an $(n-1)$-truncated object in an $\infty$-topos $\E$. An $\infty$-action $G\curvearrowright X$ corresponds to an $\infty$-action $\tau_{n-1} G\curvearrowright X$ of the underlying $n$-group $\tau_{n-1} G\in\Grp_n(\E)$. Moreover there is an equivalence $X\sslash \tau_{n-1}G\simeq \tau_n(X\sslash G)$.
    \end{prop}
    \begin{proof}
      Since $\B\Aut(X)$ is $n$-truncated, there are equivalences
     \[\E_\ast(\B G,\B\Aut(X))\simeq\E_\ast(\tau_n(\B G),\B\Aut(X))\simeq \E_\ast(\B(\tau_{n-1}G),\B\Aut(X)),\]
     where we used \cref{cor: classifying object of the truncation of the group} for the second one. Note that the inverse map is the restricted $\infty$-action construction from Example~\ref{example: group actions: restricted action}. Now consider the following commutative square:
\[\begin{tikzcd}
    {X\sslash G} \arrow[r, two heads] \arrow[d] \arrow[dr, phantom, "\lrcorner"{anchor=center, pos=0.125}, near start] & {\tau_n(X\sslash G)} \arrow[d] \\
    \B G \arrow[r, two heads] & {\tau_n(\B G)}
\end{tikzcd},\]
which is Cartesian by \cref{corollary: Cartesian square of truncation maps}.
This shows that $X\sslash \tau_{n-1}G\simeq\tau_n(X\sslash G)$.
    \end{proof}
    \begin{example}\label{example: action of Pi_n(X) on the fiber F}
        Let $F\to E\to X$ be a fiber sequence where $X$ is pointed and connected, and $F$ is $(n-1)$-truncated. The $\Omega X$-action on $F$ from Example~\ref{example: group actions: loop space action} induces a $\Pi_n(X,x)$-action on $F$:
\[\begin{tikzcd}
    F \arrow[r] \arrow[d] \arrow[dr, phantom, "\lrcorner"{anchor=center, pos=0.225}, near start] & {F\sslash \Omega X} \arrow[r, "{\eta_E}"] \arrow[d] \arrow[dr, phantom, "\lrcorner"{anchor=center, pos=0.125}, near start] & {F\sslash\Pi_n(X,x)} \arrow[d] \\
    \ast \arrow[r] & {\B\Omega X} \arrow[r, "{\eta_X}"] & {\B\Pi_n(X,x)}
\end{tikzcd},\]
where $\B\Pi_n(X,x)=\tau_n X$, $F\sslash \Omega X\simeq E$, and $F\sslash \Pi_n(X,x)\simeq \tau_nE\simeq\tau_n(F\sslash\Omega X)$. Moreover, the $\infty$-group map $\Omega X\to \Pi_n(X,x)$ respects the $\infty$-actions on $F$ by construction.
    \end{example}
    The $n$-truncation functor $\tau_n:\E\to\E$ preserves finite products, which by \cref{prop: product preserving functors preserves actions} implies that it preserves $\infty$-actions. While it is not left exact and does not preserve realizations, it nonetheless behaves well with quotients. Recall that it also behaves well with classifying objects (\cref{rmk: on truncations of realizations}).
    \begin{prop}\label{prop: truncation functor preserve actions}
        Let $G\curvearrowright X$ be an $\infty$-action in an $\infty$-topos $\E$, and $n\geq -2$. Then there is an $\infty$-action $\tau_nG\curvearrowright\tau_n X$ classified by the Cartesian square:
\[\begin{tikzcd}
    {\tau_n X} \arrow[r] \arrow[d] \arrow[dr, phantom, "\lrcorner"{anchor=center, pos=0.125}, near start] & \ast \arrow[d] \\
    {\tau_{n+1}(\tau_n^{\B G}(X\sslash G))} \arrow[r] & {\B (\tau_nG)}
\end{tikzcd},\]
where the bottom classifying map is $\tau_{n+1}(\tau_n^{\B G}(X\sslash G)\to \B G)$.
    \end{prop}
\begin{proof}
We proceed in two steps: first we truncate $X$, and then $G$. The functor $\mathrm{fib}\colon \slice{\E}{\B G}\to\E$ has both adjoints, which implies by \cite[5.5.6.28]{LurieHTT} that it commutes with the respective $n$-truncation functors:
\[\mathrm{fib}\circ \tau_n^{\B G}\simeq \tau_n\circ\mathrm{fib}.\]
It follows that there is a pasting of Cartesian squares 
\[\begin{tikzcd}
    X \arrow[r] \arrow[d] \arrow[dr, phantom, "\lrcorner"{anchor=center, pos=0.125}, near start] & {\tau_nX} \arrow[r] \arrow[d] \arrow[dr, phantom, "\lrcorner"{anchor=center, pos=0.125}, near start] & \ast \arrow[d] \\
    {X\sslash G} \arrow[r] & {\tau_n^{\B G}(X\sslash G)} \arrow[r] & \B G
\end{tikzcd},\]
which witnesses a $G$-action on $\tau_nX$ with quotient $\tau_nX\sslash G\simeq\tau_n^{\B G}(X\sslash G)$. Since $\tau_nX$ is $n$-truncated, we apply \cref{prop: actions on $n$-truncated objects corresponds to actions of an n-group} to obtain the result.
\end{proof}
 
 \begin{prop}\label{prop: induced action on pi_k}
     Let $\E$ be an $\infty$-topos, $G\curvearrowright X$ an $\infty$-action in $\E$ on a pointed object $(X,x)$, and $k\geq0$. There exists an action $\pi_k(G)\curvearrowright\pi_k(X,x)$ classified by the Cartesian square:
\[\begin{tikzcd}
    {\pi_k(X,x)} \arrow[r] \arrow[d] \arrow[dr, phantom, "\lrcorner"{anchor=center, pos=0.125}, near start] & \ast \arrow[d] \\
    {\tau_1\left(\tau_0^{\B(\Omega^k G)}(\Omega^kX\sslash \Omega ^k G)\right)} \arrow[r] & {\B\pi_k(G)}
\end{tikzcd},\]
with its quotient given by the following image factorization in terms of $X\sslash G$ and $\B G$: 
\begin{equation}\label{prop: induced action on pi_k eq1}
\pi_k(X,x)\twoheadrightarrow \pi_k(X,x)\sslash\pi_k(G)\hookrightarrow \tau_1\left(\tau_0^{\Omega^k(\B G)}\left(\Omega^k(X\sslash G)\right)\right).    
\end{equation}
 \end{prop}
\begin{remark}\label{remark: any action with a basepoint is pointed}
    There is a potential ambiguity in the hypothesis of \cref{prop: induced action on pi_k} regarding whether one requires an $\infty$-action on a pointed object or a pointed $\infty$-action. These perspectives are compatible: any unpointed $\infty$-action $G\curvearrowright X$ on an object equipped with a basepoint $x\colon\ast\to X$ canonically lifts to a pointed $\infty$-action in $\Act_G(\E_\ast)$. Indeed, the composition $\ast \xrightarrow{x} X \to X\sslash G$ provides a basepoint for the quotient, which renders the classifying map $X\sslash G\to\B G$ pointed. This corresponds to a pointed $\infty$-action by \cref{cor: characterization of pointed action}.
\end{remark}
\begin{proof}
    The functor $\pi_k$ is given by the composition: \[\E_\ast\xrightarrow{\Omega^k}\E_\ast\xrightarrow{U}\E\xrightarrow{\tau_0}\E.\]
    Since $\Omega^k$ is left exact, and $U$ preserves finite products and realizations, we can apply \cref{prop: product preserving functors preserves actions} twice to the pointed $(G,e)$-action on $(X,x)$ in $\Act(\E_\ast)$ to obtain a $(U\Omega^kG)$-action on $U\Omega^kX$ in $\Act(\E)$ given by (omitting $U$ from the notation):
\[\begin{tikzcd}
    {\Omega^kX} \arrow[r, two heads] \arrow[d] \arrow[dr, phantom, "\lrcorner"{anchor=center, pos=0.225}, near start] & {\Omega^kX\sslash \Omega^kG} \arrow[r, hook] \arrow[d] \arrow[dr, phantom, "\lrcorner"{anchor=center, pos=0.125}, near start] & {\Omega^k(X\sslash G)} \arrow[d] \\
    \ast \arrow[r, two heads] & {\B (\Omega^kG)} \arrow[r, hook] & {\Omega^k(\B G)}
\end{tikzcd},\]
where we used that $U$ preserves monomorphisms (as it is left exact) and effective epimorphisms (as it is left exact and preserves realizations).
Applying \cref{prop: truncation functor preserve actions} for $n=0$ yields the action $\pi_k(G)\curvearrowright\pi_k(X,x)$ with the desired defining square.

\medskip

To obtain the quotient in terms of $X\sslash G$ and $\B G$, one notices that the pullback functors 
\[\slice{\E}{\Omega^k(\B G)}\to\slice{\E}{\B(\Omega^k G)}\to\E\]
along the effective epi/mono factorization $\ast\twoheadrightarrow \B(\Omega^k G)\hookrightarrow \Omega^k(\B G)$ have both adjoints and thus commute with the respective $0$-truncation functors. Since effective epi/mono factorizations are stable under pullback, we obtain the factorization \[\pi_k(X)\twoheadrightarrow \tau_0^{\B(\Omega^k G)}\left(\Omega^k X\sslash\Omega^k G\right)\hookrightarrow \tau_0^{\Omega^k(\B G)}\left(\Omega^k(X\sslash G)\right).\] Applying $\tau_1$, which preserves monomorphisms and effective epimorphisms, yields the desired factorization. 
\end{proof}
We are now interested in the stabilizer $\Stab_{\pi_{k}(G)}(1_x)\to \pi_{k}(G)$ of this action at the identity element $1_x\in\pi_{k}(X)$. 
\begin{definition}[{\cite[Def. 5.1.289]{Schreiber_DifferentialCohomologyCohesiveInfinityTopos}}]\label{def: stabilizer}
    Let $G\curvearrowright X$ be an $\infty$-action and $x:\ast\to X$ a global element. Its \textit{stabilizer $\infty$-group} is \[\Stab_G(x)\coloneq\Omega_x(X\sslash G).\]
    Its classifying object is the connected component of the basepoint $x$ in $X\sslash G$:
\[\begin{tikzcd}
	\ast & X & {X\sslash G} \\
	& {\B\Stab_G(x)} & {\B G}
	\arrow["x", from=1-1, to=1-2]
	\arrow[two heads, from=1-1, to=2-2]
	\arrow[from=1-2, to=1-3]
	\arrow[from=1-3, to=2-3]
	\arrow[hook, from=2-2, to=1-3]
	\arrow[dashed, from=2-2, to=2-3]
\end{tikzcd}.\]
This factorization yields a canonical pointed map $\B\Stab_G(x)\to\B G$.
\end{definition}
\begin{prop}\label{prop: k+1 fold stabilizers of pi_k(G)}
    Let $\E$ be an $\infty$-topos, $G\curvearrowright X$ an $\infty$-action in $\E$ on a pointed object $(X,x)$, and $k\geq0$. Consider the stabilizer $\Stab_{\pi_k(G)}(1_x)$ of the action $\pi_k(G)\curvearrowright \pi_k(X,x)$. Its $(k+1)$-fold classifying object is the $(k-1)$-image of the following composition:
\begin{equation}\label{prop: k+1 fold stabilizers of pi_k(G) eq1}
\begin{tikzcd}
	\ast & X & {X\sslash G} & {\tau_{k+1}\left(\tau_k^{\B G}(X\sslash G)\right)} \\
	& {\B^{k+1}\Stab_{\pi_k(G)}(1_x)} & {}
	\arrow[from=1-1, to=1-2]
	\arrow[from=1-1, to=2-2]
	\arrow[from=1-2, to=1-3]
	\arrow[from=1-3, to=1-4]
	\arrow[from=2-2, to=1-4]
\end{tikzcd}
\end{equation}
    \end{prop}
\begin{proof}
    Denote the $(k-1)$-image by $A$, which is pointed by construction. Because the first map of the factorization is $(k-1)$-connected, $A$ is $k$-connected. Because the second map is $(k-1)$-truncated into a $(k+1)$-truncated object, $A$ is itself $(k+1)$-truncated. It follows that $A$ is the $(k+1)$-fold classifying object of a $(k+1)$-symmetric $1$-group. Applying $\Omega^k$ to \eqref{prop: k+1 fold stabilizers of pi_k(G) eq1}, we obtain the effective epi/mono factorization of: \begin{equation*}
        1\to\Omega^k\left(\tau_{k+1}\left( \tau_k^{\B G}\left( X\sslash G\right)\right)\right)\simeq \tau_1\left(\Omega^k\left(\tau_k^{\B G}\left(X\sslash G\right)\right)\right)\simeq\tau_1\left(\tau_0^{\B G}\left(\Omega^k\left( X\sslash G\right)\right)\right),
    \end{equation*}
    where we used \cref{cor: loop space commutes with truncation} to obtain the equivalences. By \eqref{prop: induced action on pi_k eq1} its image $\Omega^kA$ is the image of $\ast\to \pi_k(X,x)\sslash \pi_k(G)$. By definition, this yields a pointed equivalence $\Omega^k A\simeq \B \Stab_{\pi_k(G)}(1_x)$. It follows that $A\simeq \B ^{k+1}\Stab_{\pi_k(G)}(1_x)$, as desired.
\end{proof}
\begin{corollary}\label{cor: induced action of pi_n-1 in the truncated setting and classifing object of the stabilizer}
    Let $G\curvearrowright X$ be an $\infty$-action of an $n$-group $G$ on an $(n-1)$-truncated and pointed object $(X,x)$. Then the Cartesian square classifying $\pi_{n-1}(G)\curvearrowright\pi_{n-1}(X,x)$ is: 
\[\begin{tikzcd}
    {\pi_{n-1}(X,x)} \arrow[r, two heads] \arrow[d] \arrow[dr, phantom, "\lrcorner"{anchor=center, pos=0.125}, near start] & {\pi_{n-1}(X,x)\sslash\pi_{n-1}(G)} \arrow[r, hook] \arrow[d] \arrow[dr, phantom, "\lrcorner"{anchor=center, pos=0.125}, near start] & {\Omega^{n-1} (X\sslash G)} \arrow[d] \\
    \ast \arrow[r, two heads] & {\B \pi_{n-1}(G)} \arrow[r, hook] & {\Omega^{n-1}(\B G)}
\end{tikzcd}.\]
Moreover, $\B^{n}\Stab_{\pi_{n-1}(G)}(1_x)$ is given by the $(n-2)$-image of the following composition:
\begin{equation*}
    \begin{tikzcd}
	1 & X & {X\sslash G} \\
	& {\B ^n\Stab_{\pi_{n-1}(G)}(1_x)} & \B G
	\arrow[from=1-1, to=1-2]
	\arrow[from=1-1, to=2-2]
	\arrow[from=1-2, to=1-3]
	\arrow[from=1-3, to=2-3]
	\arrow[from=2-2, to=1-3]
	\arrow[dashed, from=2-2, to=2-3]
\end{tikzcd}.
\end{equation*}
It comes equipped with a pointed map $\B^n\Stab_{\pi_{n-1}(G)}(1_x)\to\B G$ which defines an $n$-symmetric sub-$1$-group.
\end{corollary}
 \begin{proof}
     For the first statement, by \cref{prop: induced action on pi_k}, it suffices to observe that \[\Omega^{n-1}(X\sslash G)\to \tau_1\left(\tau_0^{\Omega^{n-1}\B G}\left(\Omega^{n-1}(X\sslash G)\right)\right)\] is an equivalence. Similarly, for the second statement, by \cref{prop: k+1 fold stabilizers of pi_k(G)}, it suffices to observe that \[X\sslash G\xrightarrow{}\tau_{n}\left(\tau_{n-1}^{\B G}\left( X\sslash G\right)\right)\] is an equivalence. Both statements follow immediately from the truncation assumptions on $X$ and $G$. Finally, the map $\B^n\Stab_{\pi_{n-1}(G)}(1_x)\to\B G$ is a composition of $(n-1)$-truncated maps, and hence is $(n-1)$-truncated as well. By definition, this exhibits an $n$-symmetric sub-$1$-group. 
 \end{proof}
 \begin{remark}\label{remark: classifying object of stabilizer is the connected cover of the quotient}
    By \cref{definition: image of a map}, there is an equivalence $\B^n\Stab_{\pi_{n-1}(G)}(1_x)\simeq (X\sslash G)\langle n-1\rangle$ with the $(n-1)$-connected cover of the quotient.
\end{remark}

\subsection{$n$-Free, $n$-transitive and $n$-faithful $\infty$-actions}\label{subsection: n free n transitive and n faithful actions}
The standard definition of ($1$-)transitivity and ($1$-)freeness for $\infty$-actions, as defined in \cite[Def.~3.2.91]{Sati&Schreiber_EquivariantPrincipalInfinityBundles}, are quite restrictive outside of discrete contexts. For instance, if $X$ is connected, then any $\infty$-action on it is transitive, meaning the only free $\infty$-action on $X$ is the multiplication $\infty$-action on itself ($X\simeq G$). To capture meaningful properties in broader contexts, we generalize these notions to higher $n$, alongside the concept of $n$-faithfulness. Because an $\infty$-action is fully characterized by its classifying map $X\sslash G\to\B G$, we will express these properties in terms of the quotient. These higher generalizations admit a more familiar characterization via the induced actions on homotopy groups. Specifically, $n$-freeness and $n$-transitivity are controlled by the classical $1$-freeness and $1$-transitivity of $\pi_{n-1}$-actions. Specializing to the case $n=1$ thus recovers the expected classical behavior. We also give a characterization of $n$-faithfulness in terms of sub-$n$-groups of $G$.
\begin{definition}\label{def: n transitive free and faithful actions}
    Let $\E$ be an $\infty$-topos, $n\geq1$, $G\in\Grp(\E)$ an $\infty$-group, and $X\in\E$ an object. An $\infty$-action $G\curvearrowright X$ is:
    \begin{enumerate}
        \item \textit{$n$-transitive} if $(\rho,\pi_2)\colon  G\times X\twoheadrightarrow X\times X$ is $(n-2)$-connected;
        \item \textit{$n$-free} if $(\rho,\pi_2)\colon  G\times X\to X\times X$ is $(n-2)$-truncated;
        \item \textit{regular} if $(\rho,\pi_2)\colon  G\times X\xrightarrow{\simeq} X\times X$ is an equivalence;
        \item \textit{$n$-faithful} if $G\to \Aut(X)$ is $(n-2)$-truncated.
    \end{enumerate}
    We denote by $\Act_G^{n\text{-trans}}(\C)$ the full subcategory on the $n$-transitive $\infty$-actions.
\end{definition}

\begin{theorem}\label{theorem: free transitive regular action characterization}
    Let $\E$ be an $\infty$-topos, $n\geq 1$, and $G\curvearrowright X$ be an $\infty$-action in $\E$. Then the $\infty$-action is: 
    \begin{enumerate}
        \item $n$-free if and only if $X\sslash G$ is $(n-1)$-truncated;
        \item $n$-transitive if and only if $X\sslash G$ is $(n-1)$-connected.
        \item regular if and only if $X\sslash G \simeq \ast$.
    \end{enumerate}
\end{theorem}
\begin{proof}
    By effectiveness of the $\infty$-action groupoid, there is a Cartesian square:
\[\begin{tikzcd}
    {G\times X} \arrow[r, "{\pi_2}"] \arrow[d, "\rho"'] \arrow[dr, phantom, "\lrcorner"{anchor=center, pos=0.125}, near start] & X \arrow[d, two heads] \\
    X \arrow[r, two heads] & {X\sslash G}
\end{tikzcd}.\]
Using Fubini for pullbacks, we infer the existence of the following Cartesian square:

\[\begin{tikzcd}
    {G\times X} \arrow[r] \arrow[d, "{(\rho,\pi_2)}"'] \arrow[dr, phantom, "\lrcorner"{anchor=center, pos=0.125}, near start] & {X\sslash G} \arrow[d, "\Delta"] \\
    {X\times X} \arrow[r, two heads] & {X\sslash G \times X\sslash G}
\end{tikzcd}.\]
Because the bottom horizontal map is an effective epimorphism, $(\rho,\pi_2)$ is $(n-2)$-truncated (resp. $(n-2)$-connected) if and only if $\Delta$ is $(n-2)$-truncated (resp. $(n-2)$-connected). This holds if and only if $X\sslash G$ is $(n-1)$-truncated (resp. $(n-1)$-connected).
\end{proof}
\begin{remark}
    An $\infty$-action is $n$-free if it is $n$-efficient in the sense of \cite[Def.~6.4.3.1]{LurieHTT}. This generalizes the fact that, in an ordinary $1$-topos, a groupoid object is effective if and only if it is an equivalence relation $R \hookrightarrow X \times X$.
\end{remark}
\begin{remark}
   We recover the classical fact that for a $1$-group $G$ and $X$ a set, an action $G\curvearrowright X$ is $1$-free if and only if the strict quotient in $\tau_0\E$ coincides with the homotopy quotient. 
\end{remark}
    An ordinary group action $G\curvearrowright X$ is free if and only if the stabilizer subgroups $\Stab_G(x_i)$ are trivial for a fundamental domain $\{x_i\in X\}_{i\in I}$. Indeed, the classifying spaces of the stabilizers groups are the connected components of the homotopy quotient $X\sslash G$, which is discrete by freeness of the action. The following corollary extends this property to $n$-group actions.    
\begin{corollary}\label{cor: action is n-free iff stabilizers are trivial}
    Let $G\curvearrowright X$ be an $\infty$-action of an $n$-group $G$ on an $(n-1)$-truncated object $X$, and $\{x_i\colon\ast\to X\vert\,i\in I\}$ be a family of global points such that the family $\{\ast\xrightarrow{x_i} X\to X\sslash G\vert\, i\in I\}$ is jointly epimorphic. The following are equivalent:
    \begin{enumerate}
        \item $G\curvearrowright X$ is $n$-free;
        \item the $1$-action $\pi_{n-1}(G)\curvearrowright\pi_{n-1}(X,x_i)$ is $1$-free for all $i\in I$;
        \item the stabilizers of the actions $\pi_{n-1}(G)\curvearrowright\pi_{n-1}(X,x_i)$ at $1_{x_i}\in\pi_{n-1}(X,x_i)$ are trivial for all $i\in I$:\[\Stab_{\pi_{n-1}(G)}(1_{x_i})\simeq 1;\]
        \item the stabilizer of the action $\pi_{n-1}(G\times X)\curvearrowright \pi_{n-1}(X\times X,\Delta_X)$ in $\slice{\E}{X}$ at $1_{\Delta_X}\in\pi_{n-1}(X\times X,\Delta_X)$ is trivial:
        \[\Stab_{\pi_{n-1}(G\times X)}(1_{\Delta_X})\simeq 1_X.\]
    \end{enumerate}
\end{corollary}
\begin{proof}
    $(1\Leftrightarrow 3)$ 
        By \cref{theorem: free transitive regular action characterization}, the $G$-action is $n$-free if and only if $X\sslash G$ is $(n-1)$-truncated. By \cref{prop: cover whose n-connected objects are contractible is n-truncated}, this holds if and only if the $(n-1)$-connected covers $(X\sslash G)_i\langle n-1\rangle\simeq\ast$ at $\ast\xrightarrow{x_i}X\to X\sslash G$ are trivial for all $i\in I$. However, by \cref{remark: classifying object of stabilizer is the connected cover of the quotient}, these covers are precisely the $n$-fold classifying objects of the stabilizers of the mentioned actions: \[(X\sslash G)_i\langle n-1\rangle\simeq \B ^n\Stab_{\pi_{n-1}(G)}(1_{x_i}).\]
    
\medskip

    $(1\Rightarrow 2)$
       By \cref{cor: induced action of pi_n-1 in the truncated setting and classifing object of the stabilizer}, there is a monomorphism $\pi_{n-1}(X,x_i)\sslash \pi_{n-1}(G)\hookrightarrow\Omega^{n-1}(X\sslash G)$. Because $X\sslash G$ is $(n-1)$-truncated by $(1)$, the target loop object is discrete. Consequently, the domain is discrete as well, which means that the action is $1$-free.
       
\medskip
         
    $(2\Rightarrow 3)$ 
        The composition $\ast\xrightarrow{1_{x_i}}\pi_{n-1}(X,x_i)\to \pi_{n-1}(X,x_i)\sslash\pi_{n-1}(G)$ is a monomorphism, which implies that its image $\B \Stab_{\pi_{n-1}(G)}(1_{x_i})\simeq\ast$ is trivial.
        
\medskip

    $(1\Leftrightarrow 4)$
        The $\infty$-action $G\curvearrowright X$ is $n$-free if and only if $X\sslash G$ is $(n-1)$-truncated, which is equivalent to the diagonal $\Delta_{X\sslash G}$ being $(n-2)$-truncated. This holds if and only if $X\to X\sslash G\times X$ is $(n-2)$-truncated, as witnessed by the following Cartesian square whose vertical arrows are effective epimorphisms:
     \[\begin{tikzcd}
        X \arrow[r] \arrow[d, two heads] \arrow[dr, phantom, "\lrcorner"{anchor=center, pos=0.125}, near start] & {X\sslash G\times X} \arrow[d, two heads] \\
        {X\sslash G} \arrow[r, "\Delta_{X\sslash G}"'] & {X\sslash G\times X\sslash G}
    \end{tikzcd}.\]
    Since $(X\times X)\sslash (G\times X)\simeq (X\sslash G\times X)$  in $\slice{\E}{X}$ by \cref{prop: product preserving functors preserves actions}, this condition holds if and only if $\B^n\Stab_{\pi_{n-1}(G\times X)}(1_{\Delta_X})\simeq 1_X\in\slice{\E}{X}$, which in turn happens if and only if $\Stab_{\pi_{n-1}(G\times X)}(1_{\Delta_X})$ is trivial.
\end{proof}
\begin{remark}
    If we do not impose truncation restrictions on $G$ and $X$, one can show that $G\curvearrowright X$ is $n$-free if and only if for all $i \in I$, the action $\pi_{k}(G)\curvearrowright\pi_k(X,x_i)$ is $1$-free for $k=n-1$ and regular for $k\geq n$.
    Similarly, it is $n$-transitive if and only if the action on $\pi_k(X,x)$ is $1$-transitive for $k=n-1$ and regular for $0\leq k<n-1$. We do not pursue this here.
\end{remark}

    A classical group action $G\curvearrowright X$ on a set is faithful if and only if the only $H\leq G$ acting trivially is the trivial group. Equivalently, the classifying object $\B H$ is $0$-truncated, meaning $H$ is a $0$-group. For $n$-group actions, we obtain the following result.
    \begin{corollary}\label{cor: faithful action characterization}
     Let $G\curvearrowright X$ be an $\infty$-action of an $n$-group $G$ on a $(n-1)$-truncated object $X$. Then the $\infty$-action is $n$-faithful if and only if for every sub-$n$-group $H\to G$ such that the restricted $\infty$-action $H\curvearrowright X$ is trivial, $H$ is an $(n-1)$-group.
\end{corollary}
\begin{proof}
     Because $X$ is $(n-1)$-truncated, both $G$ and $\Aut(X)$ are $n$-groups. Therefore, the $\infty$-action is $n$-faithful if and only if the induced map $G\to \Aut(X)$ is a sub-$n$-group. By \cref{prop: sub n group characterization}, this holds if and only if every sub-$n$-group $H\to G$ with trivial composition $H\to G\to\Aut(X)$ is an $(n-1)$-group.
\end{proof}

\subsection{$\infty$-Actions of a product}\label{subsection: action from a product}
Classically, an action of a product group $G \times H$ on a set $X$ corresponds to two commuting actions $G \curvearrowright X$ and $H \curvearrowright X$, satisfying $g\cdot (h\cdot x) = h\cdot(g\cdot x)$ for all $g \in G$, $h \in H$, and $x \in X$. In this situation one has the following elementary result.
\begin{prop*}
    If the $G$-action is transitive and the $H$-action is faithful, then the $H$-action is free.
\end{prop*}
\begin{proof}
    Let $h\in H$ such that $h \cdot x = x$ for some $x \in X$, and let $y \in X$ be arbitrary. By transitivity of $G$, there exists $g \in G$ such that $g \cdot x = y$. Because the actions commute,
    \[h\cdot y=h\cdot(g\cdot x)=g\cdot(h\cdot x)=g\cdot x=y.\]
    Thus $h$ acts trivially on $X$, and faithfulness of the $H$-action forces $h=1$.
\end{proof}
The goal of this section is to generalize this result to the case of two commuting $n$-group actions on an $(n-1)$-truncated object. The main theorem is:
\begin{theorem}\label{Theorem: action of a product restricted to H is free}
    Let $G,H$ be $n$-group objects equipped with an $\infty$-action $(G\times H)\curvearrowright X$ on an $(n-1)$-truncated object $X$. Suppose that the restricted $\infty$-action $G\curvearrowright X$ is $n$-transitive and that the restricted $\infty$-action $H\curvearrowright X$ is $n$-faithful. Then the $\infty$-action $H\curvearrowright X$ is $n$-free.
\end{theorem}
The elementary algebraic manipulation above does not lift directly to the abstract framework. To bypass this, we isolate a more conceptual proof of the classical statement which does generalize. That proof runs as follows.
\begin{proof}
    By \cref{cor: action is n-free iff stabilizers are trivial}, it suffices to show that the $H$-stabilizer $\Stab_H(x)$ is trivial for an arbitrary $x\in X$. Because the $G$-action is transitive, one can show that the restricted action of $\Stab_H(x)$ on $X$ is trivial (this is \cref{cor: the action of the stabilizer of H is trivial if the action of G is transitive}). Indeed, for any $h\in\Stab_H(x)$ and $y\in X$, pick $g\in G$ with  $y=g\cdot x$; then $$h\cdot y=h\cdot (g\cdot x)=g\cdot (h\cdot x)=g\cdot x=y.$$ Since the $H$-action is faithful, the subgroup $\Stab_H(x)\leq H$ that acts trivially must be trivial by \cref{cor: faithful action characterization}.
\end{proof}
Of the three corollaries used in this argument, two have already been established for $n$-groups. The third, namely that $G$-transitivity forces the $H$-stabilizer to act trivially, is the only one that we need to lift from the classical case. It is quite technical and we establish it in \cref{subsubsection: triviality of the H-stabilizer action}, after introducing the necessary properties of commuting $\infty$-actions in \cref{subsubsection: commuting actions}. With this final piece in place, we replicate the proof for $n$-groups in \cref{subsubsection: proof of action of a product restricted to H is free}. This result will be a crucial ingredient in \cref{theorem: action of Deck is n-free}, where we show that the $\infty$-action of $\Deck(p)$ on the fiber $F$ of an $n$-cover $p\colon E\to X$, which commutes with the  $\Pi_n(X,x)$-action, is $n$-free.
\subsubsection{Commuting $\infty$-actions}\label{subsubsection: commuting actions}
\begin{definition}
    Two $\infty$-actions $G\curvearrowright X$ and $H\curvearrowright X$ \textit{commute} if they are equipped with a map $\B G\times\B H\to\B\Aut(X)$ making the following diagram commute:
\[\begin{tikzcd}
	{\B G\vee\B H} & {\B\Aut(X)} \\
	{\B G\times\B H}
	\arrow[from=1-1, to=1-2]
	\arrow[from=1-1, to=2-1]
	\arrow[from=2-1, to=1-2]
\end{tikzcd},\]
where the left vertical arrow is the map that includes each factor into the product by using the basepoint of the other.
\end{definition}

One useful consequence of commutativity is that each $\infty$-group induces an $\infty$-action on the quotient by the other. For an ordinary group acting on a set, this follows by an elementary manipulation. The case of $\infty$-actions is no harder and follows from a simple diagrammatic argument. Indeed, there is a pasting of Cartesian squares:
\[\begin{tikzcd}
    {X\sslash G} \arrow[r] \arrow[d] \arrow[dr, phantom, "\lrcorner"{anchor=center, pos=0.125}, near start] & \B G \arrow[r] \arrow[d] \arrow[dr, phantom, "\lrcorner"{anchor=center, pos=0.125}, near start] & \ast \arrow[d] \\
    {X\sslash (G\times H)} \arrow[r] & {\B (G\times H)} \arrow[r] & \B H
\end{tikzcd}.\]
The left square is Cartesian by the definition of restricted $\infty$-actions, and the composite rectangle classifies an $H$-action on $X\sslash G$. By symmetry, there is also an induced $G$-action on $X\sslash H$. The quotients of these $\infty$-actions coincide with the total quotient of $G\times H\curvearrowright X$:
\begin{equation}\label{equation: quotients of commuting actions}
    (X\sslash G)\sslash H\simeq X\sslash (G\times H)\simeq(X\sslash H)\sslash G.
\end{equation}
\begin{remark}
    Though not needed in what follows, $X$ can be recovered as the pullback of the span $X\sslash H\to X\sslash (G\times H)\leftarrow X\sslash G$. 
\end{remark}
\subsubsection{Triviality of the $H$-stabilizer action}\label{subsubsection: triviality of the H-stabilizer action}


We now prove the technical result mentioned in the introductory paragraph of \cref{subsection: action from a product}: if the $G$-action is transitive, then the stabilizer of $H$ acts trivially. The argument proceeds in three steps. First, we analyse how stabilizers of the individual $G$- and $H$-actions relate to those of the product $G\times H$. Second, we develop a general criterion that lets us deduce triviality of an action from triviality of the action on a quotient. Finally, we apply these two lemmas to obtain the desired \cref{cor: the action of the stabilizer of H is trivial if the action of G is transitive}.
\begin{lemma}\label{lemma: product of stabilizers}
   Let $G,H$ be $n$-group objects equipped with an $\infty$-action $(G\times H)\curvearrowright X$ on an $(n-1)$-truncated and pointed object $(X,x)$. Then there exists a pointed map
    \begin{equation*}
        L:\B^n \left(\Stab_{\pi_{n-1}(G)}(1_x)\times \Stab_{\pi_{n-1}(H)}(1_x)\right)\to \B^n \Stab_{\pi_{n-1}(G\times H)}(1_x)
    \end{equation*}
    over $\B (G\times H)$.
\end{lemma}
\begin{proof}
First, observe that $\B^n\Stab_{\pi_{n-1}(G)}(1_x)$ and $\B^n\Stab_{\pi_{n-1}(H)}(1_x)$ both admit canonical maps to $\B^n \Stab_{\pi_{n-1}(G\times H)}(1_x)$. Indeed, each fits into a lifting diagram as follows:
\begin{equation*}
    \begin{tikzcd}
	\ast & {} & {\B^n \Stab_{\pi_{n-1}(G\times H)}(1_x)} \\
	{\B^n \Stab_{\pi_{n-1}(G)}(1_x)} & {X\sslash G} & {X\sslash (G\times H)}
	\arrow[from=1-1, to=1-3]
	\arrow[from=1-1, to=2-1]
	\arrow[from=1-3, to=2-3]
	\arrow[dashed, from=2-1, to=1-3]
	\arrow[from=2-1, to=2-2]
	\arrow[from=2-2, to=2-3]
\end{tikzcd}.
\end{equation*}
The left vertical map is $(n-2)$-connected, while the right vertical map is $(n-2)$-truncated; hence the dashed lift exists. The same reasoning applies for $H$. This diagram lives over $\B (G\times H)$ because so does the terminal object of the diagram $X\sslash (G\times H)$. We now combine these two maps to define $L$. Consider the following lifting problem:
\begin{equation*}
\begin{tikzcd}
	{\B^n \Stab_{\pi_{n-1}(G)}(1_x)\vee \B^n \Stab_{\pi_{n-1}(H)}(1_x)} && {\B^n \Stab_{\pi_{n-1}(G\times H)}(1_x)} \\
	{\B^n \Stab_{\pi_{n-1}(G)}(1_x)\times \B^n \Stab_{\pi_{n-1}(H)}(1_x)} & {\B G\times \B H} & {\B (G\times H)}
	\arrow[from=1-1, to=1-3]
	\arrow[from=1-1, to=2-1]
	\arrow[from=1-3, to=2-3]
	\arrow[dashed, "L", from=2-1, to=1-3]
	\arrow[from=2-1, to=2-2]
	\arrow["\simeq"{marking, allow upside down}, draw=none, from=2-2, to=2-3]
\end{tikzcd}.
\end{equation*}
Because $\B^n \Stab_{\pi_{n-1}(G)}(1_x)$ and $\B^n \Stab_{\pi_{n-1}(H)}(1_x)$ are $(n-1)$-connected, their wedge inclusion into the product is $2(n-1)\geq (n-1)$-connected by \cite[Prop.~4.12]{DevalapurkarHaine_James/HiltonMilnorSplittings}. Meanwhile, the right vertical map is $(n-1)$-truncated because $X$ is so (\cref{cor: induced action of pi_n-1 in the truncated setting and classifing object of the stabilizer}). The lift $L$ therefore exists. 
\end{proof}
\begin{remark}
    Though not needed, the map $L$ is $(n-1)$-truncated and thus defines a sub-$n$-group.
\end{remark}
\begin{lemma}\label{lemma: action is trivial on X iff trivial on X//G}
    Let $G,A$ be $n$-groups equipped with an $\infty$-action $(G\times A)\curvearrowright X$. Suppose that 
    \begin{enumerate}
        \item  the $A$-action on $X\sslash G$ is trivial, so that $(X\sslash G)\sslash A\simeq X\sslash G\times\B A$;
        \item under this equivalence, the $(G\times A)$-action on $X$ is classified by \[
        X\sslash G\times \B A\xrightarrow{\chi\times 1_{\B A}} \B G\times \B A,\]
        where $\chi\colon X\sslash G\to\B G$ is the classifying map of the $G$-action on $X$.
    \end{enumerate}
    Then the $\infty$-action $A\curvearrowright X$ is also trivial.
\end{lemma}
\begin{remark}
    The second hypothesis is necessary. For example, if $G=A=X$ is an abelian group endowed with the left multiplication action, we have that the $A$-action on $X\sslash G\simeq \ast$ is necessarily trivial, but the $A$-action on $X$ is not. The $(G\times A)$-action on $X$ is classified by the diagonal $\Delta_{\B G}$, while the map of the second hypothesis is $(\ast,1_{\B G})$.
\end{remark}
\begin{proof}
Consider the diagram of restricted $\infty$-actions on $X$:
\[\begin{tikzcd}
    {X\sslash A} \arrow[r] \arrow[d] \arrow[dr, phantom, "\lrcorner"{anchor=center, pos=0.125}, near start] & {X\sslash (G\times A)} \arrow[d] \\
    {\B A} \arrow[r, "{\ast\times 1_{\B A}}"] & {\B (G\times A)}
\end{tikzcd}.\]
By \eqref{equation: quotients of commuting actions}, the quotient of the $(G\times A)$-action on $X$ is equivalent to the iterated quotient $(X\sslash G)\sslash A$. The first hypothesis tells us that this quotient is $X\sslash G\times \B A$. Using the second hypothesis, the right vertical map is therefore \[X\sslash G\times\B A\xrightarrow{\chi\times1_{\B A}}\B G \times \B A.\] Since $\E$ is cartesian closed, the top left object in the pullback is exactly $X\times \B A$, together with the projection to $\B A$. This means that the $A$-action on $X$ is trivial.
\end{proof}

\begin{corollary}\label{cor: the action of the stabilizer of H is trivial if the action of G is transitive}
      Let $G,H$ be $n$-group objects equipped with an $\infty$-action $(G\times H)\curvearrowright X$ on an $(n-1)$-truncated and pointed object $(X,x)$. If the restricted $\infty$-action $G\curvearrowright X$ is $n$-transitive, the restricted $\infty$-action $\B^{n-1} \Stab_{\pi_{n-1}(H)}(1_x)\curvearrowright X$ is trivial.
\end{corollary}
\begin{proof}
    We apply \cref{lemma: action is trivial on X iff trivial on X//G} for $A=\B^{n-1} \Stab_{\pi_{n-1}(H)}(1_x)$. To do so we must verify the two hypotheses of that lemma.
    
    \textbf{First hypothesis.} We need to show that the $A$-action on $X\sslash G$ is trivial, i.e. that its quotient splits as the product $X\sslash G\times\B A$. This quotient is realized as the restriction of the $H$-action along $A\to H$, hence as the pullback of the span:
    \[\B A\to \B H\leftarrow (X\sslash G)\sslash H.\]
    Recall from \eqref{equation: quotients of commuting actions} that $(X\sslash G)\sslash H \simeq X\sslash (G\times H)$. Furthermore, because the $G$-action is $n$-transitive, $X\sslash G$ is $(n-1)$-connected. Recalling \cref{remark: classifying object of stabilizer is the connected cover of the quotient} we obtain the following equivalences:
    \[X\sslash G \simeq (X\sslash G)\langle n-1\rangle \simeq \B^n \Stab_{\pi_{n-1}(G)}(1_x).\] 
    Consequently, the map $L$ from \cref{lemma: product of stabilizers} becomes:
    \[L \colon X\sslash G \times \B A \to \B^n \Stab_{\pi_{n-1}(G\times H)}(1_x).\]
    Now consider the following commutative diagram:
\[\begin{tikzcd}
    {X\sslash G} \arrow[r] \arrow[dd] \arrow[dr, phantom, "\lrcorner"{anchor=center, pos=0.125}, near start] & {X\sslash G\times \B A} \arrow[r, "L"] \arrow[d, "\chi\times 1_{\B A}"'] & {\B^n \Stab_{\pi_{n-1}(G\times H)}(1_x)} \arrow[r] & {X\sslash (G\times H)} \arrow[d] \\
    & {\B G\times \B A}\arrow[d, "\pi_2"]\arrow[rr]\arrow[dr, phantom, "\lrcorner"{anchor=center, pos=0.125}, near start] & &{\B G\times \B H}\arrow[d, "\pi_2"]\\
    \ast \arrow[r, two heads] & {\B A} \arrow[rr] & {} & \B H
\end{tikzcd}.\]
    The upper right square commutes by the definition of $L$, while the left square, the bottom right square, and the total outer rectangle are Cartesian. Because the map $\ast\twoheadrightarrow \B^n \Stab_{\pi_{n-1}(H)}(1_x)$ is an effective epimorphism, the total right rectangle is Cartesian by \cref{prop: square is Cartesian if fibers are equivalent}.
    
    \textbf{Second hypothesis.} The pasting lemma tells us that the upper right square is Cartesian. This shows that the action of $G\times A$ on $X$ is classified by $\chi\times 1_{\B A}$.

    Having checked both hypotheses, \cref{lemma: action is trivial on X iff trivial on X//G} implies that the $A$-action on $X$ is trivial.
\end{proof}
\subsubsection{Proof of \cref{Theorem: action of a product restricted to H is free}}\label{subsubsection: proof of action of a product restricted to H is free}
We now prove the main theorem of this \cref{subsection: action from a product}, stated in the introductory paragraph.
\begin{theorem*}
    Let $G,H$ be $n$-group objects equipped with an $\infty$-action $(G\times H)\curvearrowright X$ on an $(n-1)$-truncated object $X$. Suppose that the restricted $\infty$-action $G\curvearrowright X$ is $n$-transitive and that the restricted $\infty$-action $H\curvearrowright X$ is $n$-faithful. Then the $\infty$-action $H\curvearrowright X$ is $n$-free.
\end{theorem*}
\begin{proof}
    Assume first that $X$ admits a global point $x\colon1\to X$. We show that the stabilizer $\Stab_{\pi_{n-1}(H)}(1_x)$ of the $\pi_{n-1}(H)$-action on $\pi_{n-1}(X,x)$ is trivial. By \cref{cor: induced action of pi_n-1 in the truncated setting and classifing object of the stabilizer}, $\B^{n-1}\Stab_{\pi_{n-1}(H)}(1_x)\to H$ exhibits an $n$-symmetric sub-$1$-group, which is in particular a sub-$n$-group. By \cref{cor: the action of the stabilizer of H is trivial if the action of G is transitive}, the $n$-transitivity of the $G$-action implies that the restricted $\infty$-action $\B^{n-1}\Stab_{\pi_{n-1}(H)}(1_x)\curvearrowright X$ is trivial. Because the $\infty$-action $H\curvearrowright X$ is $n$-faithful, \cref{cor: faithful action characterization} forces $\B^{n-1}\Stab_{\pi_{n-1}(H)}(1_x)$ to be an $(n-1)$-group. This means it is $(n-2)$-truncated. However it is also $(n-2)$-connected, being an $(n-1)$-fold classifying object. Consequently, it is equivalent to the terminal object $\ast$, proving that the stabilizer group is trivial. 

    \medskip
    
    We upgrade this pointwise triviality to a global statement by working in the slice $\infty$-topos $\slice{\E}{X}$ via the base change functor $-\times X \colon \E \to \slice{\E}{X}$. Because this functor has both adjoints, it preserves truncation levels, $n$-group objects, $\infty$-actions, $n$-transitivity and $n$-th faithfulness. The $(n-1)$-truncated object $(X\times X,\pi_2)\in\slice{\E}{X}$ has a canonical global point given by the diagonal $\Delta\colon X\to X\times X$. By applying the first part of our proof in $\slice{\E}{X}$, we deduce that the stabilizer $\Stab_{\pi_{n-1}(H\times X)}(1_{\Delta_X})$ of the $\pi_{n-1}(H\times X)$-action on $\pi_{n-1}(X\times X,\Delta _X)$ is trivial. Finally, applying the $(1 \Leftrightarrow 4)$ equivalence of \cref{cor: action is n-free iff stabilizers are trivial}, we conclude that the $\infty$-action $H\curvearrowright X$ is $n$-free.
\end{proof}
\subsection{Fixed points}\label{subsection: fixed points}
In the case of an ordinary action of a group $G$ on a set $X$, the homotopy fixed points $X^{hG}$ agree with the fixed points $X^G$ and the canonical map $X^{hG}\to X$ is a monomorphism. We generalize this result in \cref{corollary: inclusion of fixed points in X is (n-1)-truncated} to an $(n-1)$-truncated object $X$. 

\medskip

In the next \cref{section : Coverings}, we will identify the group of deck transformations of a map $p\colon E\to X$ with fiber $F$ as the fixed points of a certain $\infty$-action $\Omega X\curvearrowright\Aut(F)$ (\cref{corollary: deck transformations as Omega X-equivariant autoequivalences}). This will be the appropriate point of view which enables us to show that when $p$ is an $n$-covering map, the $\Deck(p)$-action on $F$ is $n$-faithful (\cref{cor: Deck(p)-action is n-faithful}).

\medskip

Classically, the fixed points of a $G$-action on $X$ correspond exactly to the $G$-equivariant maps from the terminal object into $X$, where the terminal object carries the trivial action. In the $\infty$-categorical setting, this point of view translates directly to the space of sections of the classifying map $X\sslash G\to \B G$.  This motivates the following definition. Recall from \cref{section: preliminaries subsection: higher topoi} that the right adjoint $\prod_X$ to the functor $-\times X\colon\E\to\slice{\E}{X}$ is the functor of \textit{sections}. 
    \begin{definition}[{\cite[Def.~5.1.269]{Schreiber_DifferentialCohomologyCohesiveInfinityTopos}}]\label{definition: fixed points}
        Let $G\curvearrowright X$ be an $\infty$-action in an $\infty$-topos $\E$. The object of \textit{fixed points} or \textit{invariants} $X^G$ is defined as the object of sections of the classifying map $X\sslash G\to \B G$: \[X^G\coloneq\prod_{\B G}X\sslash G.\]
    \end{definition}
    This definition aligns with the slice $\infty$-category equivalence $\Act_G(\E)\simeq \slice{\E}{\B G}$ (\cref{cor: characterization of pointed action}), which identifies the fixed points with the internal hom of maps $\ast\sslash G\to X\sslash G$ over $\B G$, where $\ast\sslash G\simeq \B G$ is the quotient of the trivial action.
    As a fundamental consistency check, the $G$-action on $X$ must induce a trivial $G$-action on $X^G$. This is achieved by pulling back the defining square of the $G$-action along the counit map $X^G\times \B G\to X\sslash G$ of the $\prod_{\B G}\dashv (-\times \B G)$ adjunction:
    \begin{equation}\label{diagram: map from the fixed points to X}
    \begin{tikzcd}
        {X^G} \arrow[r] \arrow[d] \arrow[dr, phantom, "\lrcorner"{anchor=center, pos=0.125}, near start] & X \arrow[r] \arrow[d] \arrow[dr, phantom, "\lrcorner"{anchor=center, pos=0.125}, near start] & \ast \arrow[d] \\
        {X^G\times \B G} \arrow[r, "\varepsilon"'] & {X\sslash G} \arrow[r] & \B G
    \end{tikzcd}.
    \end{equation}
    This diagram yields the desired $G$-action on $X^G$ whose quotient coincides with that of the trivial $\infty$-action, $X^G\sslash G\simeq X^G\times \B G$. Furthermore, it is equipped with a $G$-equivariant map $X^G\to X$, obtained as the pullback of the counit $\varepsilon$ over $\B G$.
    
\begin{lemma}\label{lemma: counit is n-1 truncated}
    Let $\C$ be a presentable locally Cartesian closed $\infty$-category, $B \in\C$ a connected object, and $Y\in (\slice{\C}{B})^{\sleq{n}}$ an $n$-truncated object of the slice. Then the counit map \[\varepsilon_Y:(\prod_BY)\times B\to Y\]
    is $(n-1)$-truncated in $\slice{\C}{B}$.
\end{lemma}
\begin{proof}
    Let $T\in \slice{\C}{B}$ an object over $B$. It suffices to show that the induced map of spaces 
    \begin{equation}\label{lemma: counit is n-1 truncated: eq1}
    \slice{\Map}{B}(T,(\prod_BY)\times B)\to\slice{\Map}{B}(T,Y)
    \end{equation}
    is $(n-1)$-truncated. Equivalently, it suffices to show that its fiber at an arbitrary morphism $s\colon T\to Y$ in $\slice{\C}{B}$ is $(n-1)$-truncated. By adjunction, the domain is identified with $\slice{\Map}{B}(T\times B,Y)$, where $T\times B$ lives over $B$ via the projection. Under this identification, the map \eqref{lemma: counit is n-1 truncated: eq1} is given by precomposition with the section $T\to T\times B$. The fiber of this map at a $s$ is precisely the space $L$ of lifts of the following square:
\[\begin{tikzcd}
    T \arrow[r, "s"] \arrow[d] & Y \arrow[d] \\
    {T\times B} \arrow[r, "\pi_2"'] \arrow[ur, dashed] & B
\end{tikzcd}.\]
Since $B$ is connected, $T\times B\to T$ is connected in $\slice{\E}{T}$, which implies that the left-hand map, being a section, is $(-1)$-connected. Because the right-hand map is $n$-truncated by hypothesis, the space $L$ of lifts is $(n-1)$-truncated by \cite[Prop.~4.2.8]{SchlankYanovski_InfiniteCategoricalEckmanHilton}.
\end{proof}
\begin{corollary}\label{corollary: inclusion of fixed points in X is (n-1)-truncated}
    Let $\E$ be an $\infty$-topos, and $G\curvearrowright X$ an $\infty$-action on an $n$-truncated object $X$. The map $X^G\to X$ from \eqref{diagram: map from the fixed points to X} is $(n-1)$-truncated.
\end{corollary}
\begin{proof}
    By \cref{lemma: counit is n-1 truncated}, the counit map $\varepsilon \colon  X^G\times \B G\to X\sslash G$ is $(n-1)$-truncated. It follows that its base change $X^G\to X$ is $(n-1)$-truncated as well.
\end{proof}


\section{$n$-Covering spaces}\label{section : Coverings}
In this section we finally assemble the main theorems of the paper concerning $n$-covering spaces in an $\infty$-topos $\E$.  All these results are built on the material developed in the preceding sections (the theory of truncations and connectivity, $n$-groups, and $\infty$-actions). We begin in \cref{subsection: universal n-covers} by defining $n$-covering maps and characterizing the universal $n$-cover.  The classification theorems are then proved in \cref{subsection: classification theorems}: they establish an equivalence between $n$-covers over a pointed connected object $X$ and $\infty$-actions of $\Pi_n(X,x)$ on $(n-1)$-truncated objects, and they relate $(k-1)$-connected covers to $k$-symmetric subgroups of $\Pi_n(X,x)$.  Next, \cref{subsection: Group of Deck transformation} introduces the $\infty$-group of deck transformations of an $n$-cover and gives several complementary perspectives on its $\infty$-action on the fiber $F$. This culminates in \cref{subsection: freeness of the deck(p)-action} in the result that this $\infty$-action is $n$-free for $(n-1)$-connected covers. In \cref{subsection: normal coverings} we characterize normal $n$-covers as those corresponding to normal sub‑$1$-groups of $\pi_n(X,x)$ and compute their deck $n$-groups as quotients of $\Pi_n(X,x)$. Finally, \cref{subsection: normalizers} develops the concept of the normalizer of a map of $\infty$-groups and relates it to deck transformations.

For the rest of this section we work in an ambient $\infty$-topos $\E$.
\subsection{Universal $n$-covers}\label{subsection: universal n-covers}
\begin{definition}\label{def: n-covering categories}
    Let $X \in \E$ be an object and let $n \geq 1$.
    \begin{enumerate}
        \item An \textit{$n$-covering object} is an object $E \in \E$ equipped with an $(n-1)$-truncated map $p \colon E \to X$, called an \textit{$n$-covering map}. The $\infty$-category of $n$-covering maps is denoted $\Cov_n(X) \coloneq (\slice{\E}{X})^{\sleq n-1}$. 
        \item If $E$ is additionally $k$-connected, we say $p$ is a \textit{$k$-connected $n$-covering map}, and we denote the corresponding full subcategory by $\Cov_n^{\sg k}(X)$.
        \item If $(X,x) \in \E_\ast$ is pointed, the $\infty$-category of \textit{pointed $n$-covering maps} is defined as $$\Cov_{\ast,n}(X) \coloneq \left(\slice{\E_\ast}{(X,x)}\right)^{\sleq n-1}.$$
    \end{enumerate}
\end{definition}
\begin{remark}
    One could be tempted to define $n$-coverings via local triviality as $F$-fiber bundles for an $(n-1)$-truncated object $F$ \cite[Def.~4.1]{NSS_PrincipalBundlesGeneralTheory}. However, because our framework relates covers to the fundamental $n$-group, we restrict to a pointed and connected base space $X$, where both definitions coincide. This mirrors classical topology ($n=1, \E=\Sp$), where covering spaces are precisely fibrations with discrete fibers \cite[Theorem~10, \S2.4, p.~78]{Spanier_AlgebraicTopology}. Ultimately, what we develop is simply a theory of $(n-1)$-truncated maps. If $X$ lacks a basepoint, one can naturally extend the results of this section by working fiberwise in the slice $\slice{\E}{X}$.
\end{remark}
\begin{prop}\label{prop: characterization of universal cover}
    Let $(X,x) \in \E_\ast$ be a pointed connected object, and let $p \colon (E,e) \to (X,x)$ be a pointed $n$-covering map. The following are equivalent:
    \begin{enumerate}
        \item the total object $E$ is $n$-connected;
        \item the total object $E$ is equivalent to the $n$-connected cover $X\langle n\rangle$ of $X$;
        \item the map $p$ is initial in the $\infty$-category $\Cov_{\ast,n}(X)$ of pointed $n$-covering maps.
    \end{enumerate}
\end{prop}
\begin{proof}
    If $E$ is $n$-connected, the given point inclusion $e\colon\ast\to E$ is $(n-1)$-connected. Therefore, the composition $\ast \xrightarrow{e}E\xrightarrow{p} X$ exhibits an $(n-1)$-connected/$(n-1)$-truncated factorization of the global point $x\colon\ast\to X$. We conclude by \cref{definition: image of a map} that $E\simeq X\langle n\rangle$. This gives $(1\Leftrightarrow 2)$. Because the basepoint $x\colon\ast\to X$ is initial in $\slice{(\E_\ast)}{X}$, its $(n-1)$-truncation $\tau_{n-1}^X(x)$ is initial in $\left(\slice{\E_\ast}{X}\right)^{\sleq n-1}$. By \cref{remark: fiberwise truncation is the n-connected/truncated factorization} this is precisely $X\langle n\rangle\to X$. This proves $(2\Leftrightarrow 3)$.
\end{proof}
\begin{definition}\label{def: universal cover}
    Let $(X,x)$ be a pointed connected object and $n \geq 1$. A \textit{universal $n$-covering map} over $X$ is a pointed $n$-covering map satisfying the equivalent conditions of \cref{prop: characterization of universal cover}.
\end{definition}

Recall from \cref{remark: contruction of the n-connected cover} that the universal $n$-covering map of a pointed and connected object $X$ can be constructed as the following pullback: 
\begin{equation*}
    \begin{tikzcd}
    {X\langle n\rangle} \arrow[r] \arrow[d, "{p_n}"'] \arrow[dr, phantom, "\lrcorner"{anchor=center, pos=0.125}, near start] & \ast \arrow[d] \\
    X \arrow[r] & {\B \Pi_n(X,x)}
    \end{tikzcd}.
\end{equation*}

\subsection{Classification theorems}\label{subsection: classification theorems}
Let $F$ be the fiber of an $n$-covering map. Recall from \cref{example: action of Pi_n(X) on the fiber F} the $\infty$-action $\Pi_n(X,x)\curvearrowright F$, whose quotient is $F\sslash \Pi_n(X,x)\simeq\tau_nE$.
\begin{remark}
    For $n=1$, one could ask whether the action $\pi_1(X,x)\curvearrowright F$ induced by the truncation map $X\to \B \pi_1(X,x)$ is the one defined classically in topology. It turns out that this is the unique such $1$-connected map \cite[Prop.~34]{MimramOleon_ClassifyingCoveringHoTT}.
\end{remark}

\begin{theorem}\label{Theorem: representation theorem for n-covers}
    Let $(X,x)\in\E_\ast$ be a pointed connected object, and let $1\leq n<\infty$ be an integer. Then there is an equivalence of $n$-categories
    \begin{equation*}
        \Cov_n(X)\simeq \Act_{\Pi_n(X,x)}(\E^{\sleq{{n-1}}})
    \end{equation*}
    between $n$-covering objects over $X$ and $\infty$-actions of the fundamental $n$-group of $X$ on $(n-1)$-truncated objects. It restricts to an equivalence \[\Cov_{\ast,n}(X)\simeq \Act_{\Pi_n(X,x)}(\E_\ast^{\sleq n-1})\]
    between pointed $n$-covers and pointed $\infty$-actions. For each $1\leq k\leq n$, these restrict to equivalences between $(k-1)$-connected (pointed) $n$-covers and $k$-transitive (pointed) $\infty$-actions.
\end{theorem}
\begin{proof}
    Let $\C$ denote either $\E$ or $\E_\ast$. By \cref{prop: characterization n-connected maps through pullback functor} and \cref{cor: characterization n-connected maps through pullback functor in the pointed category}, pulling back along the $n$-connected reflector $\eta_X\colon X\to\tau_n X$ induces an equivalence of $\infty$-categories:
    \[(\slice{\C}{\tau_n X})^{\sleq{n-1}}\to(\slice{\C}{X})^{\sleq{n-1}}.\]
    By definition, the target of this equivalence is precisely $\Cov_n(X)$ (resp. $\Cov_{\ast,n}(X)$). Meanwhile, \cref{cor: characterization of pointed action} identifies the source with $\Act_{\Pi_n(X,x)}(\E^{\sleq{{n-1}}})$ (resp. $\Act_{\Pi_n(X,x)}(\E_\ast^{\sleq {n-1}})$). Furthermore, an $n$-covering map $E\to X$ is $(k-1)$-connected if and only if $\tau_nE$ is $(k-1)$-connected. This holds if and only if the quotient $F\sslash \Pi_n(X,x)$ is $(k-1)$-connected, which by \cref{theorem: free transitive regular action characterization} is precisely the condition for the $\infty$-action $\Pi_n(X,x)\curvearrowright F$ to be $k$-transitive. By \cref{prop: basepoint forgetful functor preserves and reflects connected maps}, these connectivity equivalences hold identically in the pointed setting. 
\end{proof}

\begin{corollary}\label{corollary: pointed coverings correspond to subgroups}
    Let $(X,x)\in\E_\ast$ be a pointed connected object and $1\leq k\leq n$ be integers. Then there is an equivalence of $\infty$-categories: 
    \[\Cov^{\sg{k-1}}_{\ast,n}(X)\simeq\Sub_{(n-k+1,k)}\left(\Pi_n(X,x)\right)\]
    between $(k-1)$-connected pointed $n$-covering maps over $X$, and $k$-symmetric sub-$(n-k+1)$-groups of $\Pi_n(X,x)$.
\end{corollary}

\begin{proof}
    We have the following chain of equivalences:
\begin{align*}
    \Cov^{\sg{k-1}}_{\ast,n}(X) &\simeq \Act^{k\text{-trans}}_{\Pi_n(X,x)}(\E^{\sleq{n-1}}_\ast) && \proofstep{by \cref{Theorem: representation theorem for n-covers}}\\
    &\simeq \left(\slice{\E_\ast^{\sg{k-1}}}{\B \Pi_n(X,x)}\right)^{\sleq (n-1)} && \proofstep{by \cref{cor: characterization of pointed action}}\\
    &=\big(\slice{\E_\ast^{\sg{k-1},\sleq n}}{\B \Pi_n(X,x)}\big)^{\sleq (n-1)} && \proofstep{by Proposition~\ref{prop: equivalent characterization of truncated maps}}\\
    &\simeq \big(\slice{\Grp_{(n-k+1,k)}(\E)}{\Pi_n(X,x)}\big)^{\sleq (n-1)} && \proofstep{by \cref{theorem: symmetric group objects are pointed highly connected objects}}\\
    &= \Sub_{(n-k+1,k)}(\Pi_n(X,x)) && \proofstep{by \cref{def: sub n groups}.}
\end{align*}
\end{proof}
When $k=n$, this equivalence simplifies and identifies $(n-1)$-connected $n$-covering maps with sub-$1$-groups of the $n$-th homotopy group $\pi_n(X,x)$. 
\begin{corollary}\label{corollary: pointed n-1 connected coverings corresponds to subgroups of homotopy n-th group}
     Let $(X,x)\in\E_\ast$ be a pointed connected object. Then there is an equivalence of $\infty$-categories: 
    \[\Cov^{\sg{n-1}}_{\ast,n}(X)\simeq\Sub_{1}\left(\pi_n(X,x)\right).\]
\end{corollary}
\begin{proof}
    Apply \cref{corollary: pointed coverings correspond to subgroups} and Theorem~\ref{theorem: n-symmetric subgroups of the fundamental n-group are 1-subgroups of n-th homotopy group 1}.
\end{proof}
\begin{remark}\label{remark: decomposition of n-covering with the universal n-cover}
    Let $A\leq \pi_n(X,x)$ be a sub-$1$-group. The corresponding $(n-1)$-connected $n$-covering map decomposes as \[p_A\colon E_A\to X\langle n-1\rangle \xrightarrow{p_{n-1}} X.\] Indeed it is obtained as the following pullback:
\[
\begin{tikzcd}
	E_A \arrow[r] \arrow[d] \arrow[dr, phantom, "\lrcorner"{anchor=center, pos=0.125}, near start] \arrow[dd, bend right=50, "p_A"'] & \B^n A \arrow[d] \\
    X\langle n-1\rangle \arrow[r] \arrow[d, "p_{n-1}"] \arrow[dr, phantom, "\lrcorner"{anchor=center, pos=0.125}, near start] & \B^n \pi_n(X,x) \arrow[d, "\varepsilon"] \\
	X \arrow[r] & \B\Pi_n(X,x)
\end{tikzcd},
\]
where we used \cref{remark: contruction of the n-connected cover} twice, recalling that $\B\Pi_n(X,x)\langle n-1\rangle\simeq\B^n\pi_n(X,x)$.
\end{remark}
Classically, the orbit category of a group $G$ has subgroups as objects and conjugations as morphisms, and it is equivalent to the category of connected unpointed covers $\Cov_1^{\sg{0}}(X)$, where $\pi_1(X,x)=G$. In the setting of an $n$-group $G$, we recover this by working in the slice $\slice{\Sp^{\sg0}}{\B G}$ and allowing \textit{unpointed} homotopies between classifying maps, which correspond precisely to conjugation by elements of $G$. In the $\infty$-topos $\Sp$ of homotopy types, this perfectly models unpointed covers because every connected space admits a unique global point, which fails in a general $\infty$-topos. The existence of such a global point guarantees we can always extract a corresponding subgroup from a cover, motivating the definition for an arbitrary $n$-group.
\begin{definition}
    Let $G$ be an $n$-group object such that $\B G$ has a unique global point. Its \textit{orbit $\infty$-category} $\mathcal{O}\mathrm{rb}_G$ is the full subcategory of $(\slice{\E^{\sg0}}{\B G})^{\sleq{n-1}}$ spanned by objects $E\to \B G$ where $E$ admits a global point.  
\end{definition}
\begin{corollary}
    Working in the $\infty$-topos $\Sp$ of homotopy types, there is an equivalence: 
    \[
        \Cov_n^{\sg0}(X)\simeq \mathcal{O}\mathrm{rb}_{\Pi_n(X,x)}.
    \]
\end{corollary}

\begin{proof}
    As in the pointed case, we have the following chain of equivalences:
    \begin{align*}
        \Cov_n^{\sg0}(X) &\simeq \Act_{\Pi_n(X,x)}^{1\text{-trans}}(\Sp^{\sleq n-1}) \\
        &\simeq\big(\slice{\Sp^{\sg0}}{\B\Pi_n(X,x)}\big)^{\sleq{n-1}} \\
        &\simeq\mathcal{O}\mathrm{rb}_{\Pi_n(X,x)}.
    \end{align*}
\end{proof}

\subsection{The $\infty$-group of deck transformations}\label{subsection: Group of Deck transformation}
In this subsection we proceed in three steps.  
First, we define the $\infty$-group of deck transformations $\Deck(p)$ of a map $p\colon E\to X$, alongside its canonical fiberwise evaluation $\infty$-action $\Deck(p)\times X\curvearrowright E$ in $
\slice{\E}{X}$. Second, we offer two complementary perspectives: as an $\infty$-action $\Deck(p)\curvearrowright E$ in $\E$, and, when $X$ is pointed and connected, as a pair of commuting $\infty$-action $\Deck(p)\times \Omega X \curvearrowright F$ on the fiber $F$ of $p$. Finally, we present two alternative constructions of $\Deck(p)$ and of its fiberwise evaluation $\infty$-action. We show that these perspectives coincide, with each highlighting different properties of the $\infty$-action.

\begin{definition}
    Let $p\colon E\to X$ be a map. The $\infty$-group $\Deck(p)$ of \textit{deck transformations of $p$} is \[\Deck(p)\coloneq\prod_X\Aut_{/X}(E)\in\Grp(\E).\]
\end{definition}
This is well-defined because $\Aut_{/X}(E)$ is the autoequivalence $\infty$-group (\cref{ex: BAut(F) as the image in the universe}) of $p\in\slice{\E}{X}$, and the right adjoint $\prod_X\colon\slice{\E}{X}\to\E$ preserves finite products, hence $\infty$-group objects (\cref{prop: product preserving functors preserves group objects}). By that same proposition, the classifying object of $\Deck(p)$ is the image 
\begin{equation}\label{in text: eq definition classifying object of Deck(p)}
    \ast\twoheadrightarrow \B \Deck(p)\hookrightarrow \prod_X\B \Aut_{/X}(E)
\end{equation}
of the map adjoint to the basepoint $X\to \B \Aut_{/X}(E)$ in $\slice{\E}{X}$.
\begin{remark}\label{remark: global section of internal deck transformations}
    A global section $\ast\to\Deck(p)$ corresponds to an equivalence $E\xrightarrow{\simeq}E$ over $X$: 
    \begin{align*}
        \E\left(\ast,\Deck(p)\right)\simeq \slice{\E}{X}\left(X,\Aut_{/X}(p)\right)\simeq \slice{\E}{X}^\simeq\left(E,E\right).
    \end{align*}
\end{remark}
By applying \cref{prop: product preserving functors preserves group objects}, we obtain the $\infty$-group object $\Deck(p)\times X\in\Grp(\slice{\E}{X})$, whose classifying object is \[\B (\Deck(p)\times X)\simeq \B \Deck(p)\times X,\] with basepoint $(\ast,1_X)\colon X\to \B \Deck(p)\times X$. By adjunction, the inclusion \eqref{in text: eq definition classifying object of Deck(p)} corresponds to a pointed map \[\varepsilon\colon\B \Deck(p)\times X\to \B \Aut_{/X}(E),\] which in turn classifies a morphism of $\infty$-group $\varepsilon\colon \Deck(p)\times X\to\Aut_{/X}(E)$.
\begin{definition}\label{def: fiberwise evaluation action}
    Let $p\colon E\to X$ be a map. The \textit{fiberwise evaluation} $\infty$-action $\Deck(p)\times X\curvearrowright E$ in $\Act(\slice{\E}{X})$ is the restriction of the evaluation $\infty$-action $\Aut_{/X}(E)\curvearrowright E$ along $\varepsilon$. \end{definition}
    By definition of restricted $\infty$-action (Example~\ref{example: group actions: restricted action}), the quotient of the fiberwise evaluation $\infty$-action is characterized by the left Cartesian square in the following pasting diagram:
    \begin{equation}\label{diagram: defining the action Deck x X on E in the slice}
\begin{tikzcd}
    E \arrow[r] \arrow[d] \arrow[dr, phantom, "\lrcorner"{anchor=center, pos=0.125}] & E\sslash (\Deck(p)\times X) \arrow[r] \arrow[d] \arrow[dr, phantom, "\lrcorner"{anchor=center, pos=0.125}] & E\sslash \Aut_{/X}(E) \arrow[d] \\
    X \arrow[r, "{(\ast,1_X)}"] & {\B \Deck(p)\times X} \arrow[r, "\varepsilon"'] & {\B \Aut_{/X}(E)}
\end{tikzcd}\quad\text{in }\slice{\E}{X}.
\end{equation}

\begin{remark}\label{remark: name of the counit from BDeck(p)}
    We write $\varepsilon$ for the bottom map because it is exactly the restriction of the canonical counit of the $(-\times X)\dashv \prod_X$ adjunction:
    \[\B \Deck(p)\times X\hookrightarrow \Big(\prod_X\B \Aut_{/X}(E)\Big)\times X\xrightarrow{\varepsilon}\B \Aut_{/X}(E).\]
\end{remark}
\subsubsection{Two perspectives on the $\infty$-action}\label{subsection: two persepctives on the action}
We now offer two complementary perspectives on the fiberwise evaluation $\infty$-action from \cref{def: fiberwise evaluation action}. The first perspective views this structure as an \textit{evaluation} $\infty$-action $\Deck(p)\curvearrowright E$ in the ambient $\infty$-category $\E$. By forgetting that the diagram \eqref{diagram: defining the action Deck x X on E in the slice} lives over $X$, and post-composing with the projection $\pi_1\colon\B\Deck(p)\times X\to\B\Deck(p)$, we obtain a map whose fiber is exactly $E$. This establishes the evaluation $\infty$-action, characterized by the following Cartesian square:
\[\begin{tikzcd}[row sep=tiny]
    E \arrow[r] \arrow[dd] \arrow[ddr, phantom, "\lrcorner"{anchor=center, pos=0.125}] & {E\sslash (\Deck(p)\times X)\simeq E\sslash \Deck(p)} \arrow[d] \\
    & {\B\Deck(p)\times X} \arrow[d, "\pi_1"] \\
    \ast \arrow[r] & {\B\Deck(p)}
\end{tikzcd}.\]
 We will use this perspective in \cref{theorem: characterization normal n-covering maps} to show that the quotient of this $\infty$-action is precisely $X$ whenever $p$ is a normal $n$-covering map.

\medskip

The second perspective assumes that $X$ is pointed and connected, meaning $X\simeq \B \Omega X$ serves as the classifying object for the loop $\infty$-group $\Omega X$. By viewing the left square of \eqref{diagram: defining the action Deck x X on E in the slice} in the ambient $\infty$-category $\E$, and pasting the defining pullback of the fiber $F$, we extract an $\infty$-action $\Deck(p)\times\Omega X\curvearrowright F$. This is captured by the following pasting of Cartesian squares:
\[\begin{tikzcd}
    F \arrow[r] \arrow[d] \arrow[dr, "\lrcorner"{anchor=center, pos=0.100}, draw=none] & E \arrow[r] \arrow[d, "p"] \arrow[dr, "\lrcorner"{anchor=center, pos=0.225}, draw=none] & {E\sslash (\Deck(p)\times X)\simeq F\sslash (\Deck(p)\times\Omega X)} \arrow[d] \\
    \ast \arrow[r] & X \arrow[r, "{(\ast,1_X)}"] & {\B \Deck(p)\times \B \Omega X}
\end{tikzcd}.\]
\begin{remark}\label{remark: identification of different quotients of Deck(p)-actions}
    There is a canonical chain of equivalences between quotients:\[E\sslash (\Deck(p)\times X)\simeq E\sslash \Deck(p)\simeq (F\sslash\Omega X)\sslash \Deck(p)\simeq F\sslash (\Deck(p)\times\Omega X).\]
The first one arises from the first perspective, the second from the equivalence $E\simeq F\sslash\Omega X$, while the third follows from \cref{equation: quotients of commuting actions}.
\end{remark}
\subsubsection{Equivalent characterizations of $\Deck(p)$ and its $\infty$-action}  We now provide two alternative definitions for $\Deck(p)$ and its fiberwise evaluation $\infty$-action $\Deck(p)\times X\curvearrowright E$ in $\slice{\E}{X}$. These characterizations prove essential for establishing various properties of the $\Deck(p)$-actions (\cref{prop: group deck of a covering is pi_1-equivariant automorphism of the fiber,cor: Deck(p)-action is n-faithful,prop: image of a map into the fibered universe is deck transformation}).

\medskip

\textbf{I. Via $\Omega X$-equivariant autoequivalences.}
Assume that $X$ is pointed and connected. The $\infty$-action $\Deck(p)\times\Omega X\curvearrowright F$ established above restricts to a pair of commuting $\infty$-actions $\Deck(p)\curvearrowright F$ and $\Omega X\curvearrowright F$, as explained in \cref{subsection: action from a product}.

\medskip

 Viewing $\Aut_{/X}(E)$ as an object of the slice $\slice{\E}{\B (\Omega X)}\simeq \Act_{\Omega X}(\E)$, it classifies an $\infty$-action on its fiber $\Aut(F)$ (\cref{cor: properties of autoequivalence group of p}). This is characterized by the following Cartesian square:
\[
\begin{tikzcd}
    \Aut(F) \arrow[r] \arrow[d] \arrow[dr, "\lrcorner"{anchor=center, pos=0.125}, draw=none] & {\Aut_{/X}(E)\simeq \Aut(F)\sslash \Omega X} \arrow[d] \\
    \ast \arrow[r] & {\B \Omega X}
\end{tikzcd}.
\]
Informally, this $\infty$-action is given by conjugation: \[(\lambda \cdot\varphi)(a)\coloneq\lambda^{-1}\cdot \varphi (\lambda\cdot a)\hspace{0.5cm}\text{ for all }\lambda\in\Omega X, \varphi\in\Aut(F), a\in F.\] This motivates the following definition.
\begin{definition}[{\cite[Prop.~5.1.278]{Schreiber_DifferentialCohomologyCohesiveInfinityTopos}}]\label{def: equivariant autoequivalences objects}
    Let $\Omega X\curvearrowright F$ be an $\infty$-action. The object $\Aut_{\Omega X}(F)$ of \textit{$\Omega X$-equivariant autoequivalences of $F$} is the object of fixed points (see \cref{definition: fixed points}) of the  $\infty$-action $\Omega X\curvearrowright \Aut(F)$: \[\Aut_{\Omega X}(F)\coloneq\Aut(F)^{\Omega X}=\prod_{\B(\Omega X)}\Aut(F)\sslash\Omega X.\]
\end{definition}
\begin{remark}
    By adjunction, a global point $\ast\to \Aut_{\Omega X}(F)$ corresponds exactly to an equivalence $F\sslash\Omega X\to F\sslash \Omega X$ over $\B (\Omega X)$, that pulls back to an $\Omega X$-equivariant equivalence $F\to F$.
\end{remark}
Because $\Aut(F)\sslash\Omega X\simeq \Aut_{/X}(E)$ is an $\infty$-group in $\slice{\E}{X}$, the object of invariants $\Aut_{\Omega X}(F)$ inherits an $\infty$-group structure. By \cref{prop: product preserving functors preserves group objects}, its classifying object is the following image: 
\begin{equation*}
1\twoheadrightarrow \B \Aut_{\Omega X}(F)\hookrightarrow \prod_{\B (\Omega X)}\B (\Aut(F)\sslash\Omega X).
\end{equation*}
Recall from \eqref{diagram: map from the fixed points to X} that an object of fixed points admits a canonical map into its ambient object. In our context, this yields a canonical map 
    \begin{equation}\label{eq: forgetful map from equivariant autoequivalence}
    \Aut_{\Omega X}(F)\to \Aut(F).
    \end{equation}
    This map is, in fact, a morphism of $\infty$-groups. Indeed, one can construct a diagram analogous to \eqref{diagram: map from the fixed points to X} directly at the level of classifying objects. The map \eqref{eq: forgetful map from equivariant autoequivalence} thereby encodes an $\infty$-action $\Aut_{\Omega X}(F)\curvearrowright F$.
\begin{prop}\label{corollary: deck transformations as Omega X-equivariant autoequivalences}
    Let $(X,x)\in\E_\ast$ be a pointed connected object, and let $p\colon E\to X$ be a map with fiber $F$. Then there is an equivalence of $\infty$-groups \[\Deck(p)\simeq \Aut_{\Omega X}(F).\]
    Moreover, this equivalence respects the respective $\infty$-actions on $F$.
\end{prop}
\begin{proof}
    We have simply given different names for the same object, since $\Aut(F)\sslash\Omega X\simeq\Aut_{/X}(E)$ in $\slice{\E}{X}$. The $\infty$-actions are both defined through the counit $\varepsilon$.
    \end{proof}
This identification gives a clear justification to why the $\infty$-actions of $\Deck(p)$ and $\Omega X$ on $F$ commute: they do so by the very definition of $\Aut_{\Omega X}(F)$. Moreover, when $p$ is an $n$-covering map, this point of view offers two useful consequences. Firstly, we can further identify the deck $\infty$-group with $\Pi_n(X,x)$-equivariant autoequivalences of $F$, recovering a standard identification for $n=1$. Secondly, one obtains that the $\Deck(p)$-action on $F$ is $n$-faithful, a fundamental ingredient of \cref{theorem: action of Deck is n-free}. 

\begin{corollary}\label{prop: group deck of a covering is pi_1-equivariant automorphism of the fiber}
    Let $(X,x)\in\E_\ast$ be a pointed connected object, and let $p\colon E\to X$ be an $n$-covering map with fiber $F$. Then there is an equivalence of $n$-groups: \[\Deck(p)\simeq \Aut_{\Pi_n(X,x)}(F).\]
    Moreover, this equivalence respects the respective $\infty$-actions on $F$.
\end{corollary}
\begin{proof}
   Let $\eta_X\colon X\to\B\Pi_n(X,x)$ be the $n$-truncation map, and $\eta_X^\ast\colon\slice{\E}{\B \Pi_n(X,x)}\to\slice{\E}{X}$ its associated base-change functor. Because it preserves objects of equivalences and classifying objects, there is a Cartesian square:
\[\begin{tikzcd}[column sep=large]
    {\B \Aut_{/X}(E)} \arrow[r, "\eta"] \arrow[d] \arrow[dr, phantom, "\lrcorner"{anchor=center, pos=0.125}, near start] & {\B \Aut_{/\B \Pi_n(X,x)}\left(F\sslash \Pi_n(X,x)\right)} \arrow[d] \\
    X \arrow[r, "\eta_X"'] & {\B \Pi_n(X,x)}
\end{tikzcd}.\]
 Both horizontal maps $\eta_X$ and $\eta$ are $n$-connected. Because $p$ is an $n$-covering map, its $n$-truncation $F\sslash\Pi_n(X,x)\xrightarrow{\tau_np}\B\Pi_n(X,x)$ is $(n-1)$-truncated as well. By \cref{truncation level of the object of equivalences}, the target of $\eta$ is $n$-truncated. It follows that $\eta_X$ and $\eta$ are both $n$-truncation maps. Applying \cref{prop: cartesian square of n-truncation implies equivalence on the objects of sections} to this Cartesian square yields an equivalence on the object of sections: 
 \begin{equation}\label{prop: group deck of a covering is pi_1-equivariant automorphism of the fiber eq}
     \prod_X \B \Aut_{/X}(E)\simeq \prod_{\B \Pi_n(X,x)}\B \Aut_{/\B \Pi_n(X,x)}(F\sslash \Pi_n(X,x)).
 \end{equation}
Using \cref{prop: product preserving functors preserves group objects} and the identification \[\Aut_{/\B \Pi_n(X,x)}(F\sslash \Pi_n(X,x))\simeq\Aut(F)\sslash \Pi_n(X,x)\qquad \text{ in }\slice{\E}{\B \Pi_n(X,x)},\]
 we identify $\B \Deck(p)\simeq \B \Aut_{\Pi_n(X,x)}(F)$ as the connected components of the respective basepoints of \eqref{prop: group deck of a covering is pi_1-equivariant automorphism of the fiber eq}. Because the $\infty$-group map $\Omega X\to \Pi_n(X,x)$ preserves the $\infty$-actions on $F$ by \cref{example: action of Pi_n(X) on the fiber F}, so does the equivalence $\Aut_{\Omega X}(F)\simeq \Aut_{\Pi_n(X,x)}(F)$. Because $\Deck(p)\simeq\Aut_{\Omega X}(F)$ by construction (\cref{corollary: deck transformations as Omega X-equivariant autoequivalences}), the second claim follows.
\end{proof}
\begin{corollary}\label{cor: Deck(p)-action is n-faithful}
    Let $(X,x)\in\E_\ast$ be a pointed and connected object, and let $p\colon E\to X$ be an $n$-covering map with fiber $F$. Then the $\infty$-action $\Deck(p)\curvearrowright F$ is $n$-faithful.
\end{corollary}
\begin{proof}
   By \cref{corollary: deck transformations as Omega X-equivariant autoequivalences}, it is equivalent to analyze the $\Aut_{\Omega X}(F)$-action on $F$. It is classified by the map \[\Aut_{\Omega X}(F)\to \Aut(F),\] which by definition is the canonical map from the fixed points of the $\Omega X$-action on $\Aut(F)$. Because $F$ is $(n-1)$-truncated, this map is $(n-2)$-truncated by \cref{corollary: inclusion of fixed points in X is (n-1)-truncated}. This means that the $\Deck(p)$-action is $n$-faithful.
\end{proof}
\medskip

\textbf{II. Via the fibered universe $\U^X$.} This interpretation requires no additional assumptions on $X$. In particular, we will realize the $\infty$-action of $\Deck(p)$ on the fibers directly in the slice $\slice{\E}{X}$. We will later use this interpretation of $\Deck(p)$ to prove that when $p$ is a normal $n$-covering map, the $\infty$-group of deck transformations is identified with a quotient of the fundamental $n$-group (\cref{cor: group deck is the quotient of fundamental n-group}).

\medskip

Heuristically, the construction mirrors the description of $\B\Aut(F)$ inside the universe. Recall from \cref{ex: BAut(F) as the image in the universe} that, by univalence, $\B \Aut(F)$ is the connected component of $\corner{F}\colon\ast\to\U$, with loops $\U(\corner{F},\corner{F})\simeq\Aut(F)$. We seek $\B\Deck(p)$ in the same form: the component of a point in an object whose based loops are the self-equivalences of $E$ \emph{over the fixed base $X$}. That object is the internal hom $\U^X$. Indeed, it represents the core functor $T\mapsto(\slice{\E}{T\times X})^\simeq$, so its points classify families over $X$, and its path objects compute fiberwise equivalences, as we showed in \cite[Lem.~2.8]{Constantin_YonedaInLCCC}:
\[
  \U^X(\corner{P},\corner{Q})\simeq\prod_X\Eq_{/X}(P,Q)
  \qquad\text{for }P,Q\in\slice{\E}{X}\text{ classified by }\U.
\]
So the based loops at $\corner{p}^\sharp\colon\ast\to\U^X$ recover the fiberwise self-equivalences of $E$, that is $\Deck(p)$. The following proposition makes this precise.

\begin{prop}\label{prop: image of a map into the fibered universe is deck transformation}
    Let $p\colon E\to X$ be a map classified by $\corner{p}\colon X\to\U$. Then the image of the adjoint $\corner{p}^\sharp\colon \ast\to \U^X$ is the classifying object of the group of deck transformations of $p$:
\[\begin{tikzcd}
    \ast \arrow[rr, "{\corner{p}^\sharp}"] \arrow[dr, two heads] & & {\U^X} \\
    & {\B \Deck(p).} \arrow[ur, hook, "f"']
\end{tikzcd}.\]
Moreover the adjoint $(f^\flat, \pi_2)\colon \B \Deck(p)\times X\to\U\times X$ 
classifies the $\infty$-action $\Deck(p)\times X\curvearrowright E$ in $\slice{\E}{X}$.
\end{prop}
\begin{proof}
   Recall from \cref{lemma: universe in slice categories} that the universe in $\slice{\E}{X}$ is given by $\U\times X$. By \cref{ex: BAut(F) as the image in the universe}, the image factorization of the classifying map $(\corner{p}, 1_X)\colon X\to \U\times X$ in the slice $\slice{\E}{X}$ is given by:
\[\begin{tikzcd}
    X \arrow[r, two heads, "s"] & {\B \Aut_{/X}(E)} \arrow[r, hook, "\iota"] & {\U\times X}.
\end{tikzcd}\]
    Applying the dependent product functor $\prod_X\colon \slice{\E}{X}\to\E$ to this composition yields the map $\corner{p}^\sharp$. Because $\prod_X$ is a right adjoint, it preserves monomorphisms, meaning $\prod_X(\iota)\colon \prod_X \B \Aut_{/X}(E) \hookrightarrow \U^X$ is a monomorphism. 
    The left-hand map $s$ acts as the canonical basepoint in the slice, so applying $\prod_X$ yields the global section $\prod_X(s)\colon \ast\to \prod_X \B \Aut_{/X}(E)$. By \cref{prop: product preserving functors preserves group objects}, the image of this global section is precisely $\B \Deck(p)$. Because the composition of two monomorphisms is a monomorphism, the image factorization of the total map $\corner{p}^\sharp$ factors exactly through $\B \Deck(p)$:
\[\begin{tikzcd}
    \ast \arrow[r, two heads] & {\B \Deck(p)} \arrow[r, hook] & {\prod_X\B \Aut_{/X}(E)} \arrow[r, hook, "{\prod_X(\iota)}"] & {\U^X,}
\end{tikzcd}\]
    where we use the canonical identification $\prod_X(\U\times X)\simeq \U^X$. This proves the first claim. 

   For the second claim, we must show that the composition
\[\begin{tikzcd}
    {\B \Deck(p)\times X} \arrow[r, hook] & {\prod_X \B \Aut_{/X}(E)\times X} \arrow[rr, "{(\prod_X\iota)^\flat}"] & & \U
\end{tikzcd}\]
    agrees with the defining map \eqref{diagram: defining the action Deck x X on E in the slice} of the $\infty$-action $\Deck(p)\times X\curvearrowright E$ in $\slice{\E}{X}$:
\[\begin{tikzcd}
    {\B \Deck(p)\times X} \arrow[r, hook, "\varepsilon"] & {\B \Aut_{/X}(E)} \arrow[r, hook, "\iota"] & {\U\times X} \arrow[r, "\pi_1"] & \U,
\end{tikzcd}\]
    where $\varepsilon$ is the counit (recall \cref{remark: name of the counit from BDeck(p)}). By \cref{lemma: computation of adjoints maps in universe} below, these two compositions are indeed equivalent, completing the proof.
\end{proof}
\begin{lemma}\label{lemma: computation of adjoints maps in universe}
    Let $p \colon E \to X$ be an object in $\slice{\E}{X}$. The adjoint of the map $\prod_X(\iota) \colon \prod_X \B \Aut_{/X}(E) \to \U^X$ under the Cartesian closure adjunction $(-\times X) \dashv (-)^X$ is given by the composition:
\[\begin{tikzcd}
    {\prod_X\B \Aut_{/X}(E)\times X} \arrow[r, "\varepsilon"] & {\B \Aut_{/X}(E)} \arrow[r, hook, "\iota"] & {\U\times X} \arrow[r, "\pi_1"] & \U,
\end{tikzcd}\]
    where $\varepsilon$ is the counit of the base-change adjunction $X^\ast \dashv \prod_X$.
\end{lemma}
\begin{proof}
    Under the Cartesian closure adjunction $(-\times X) \dashv (-)^X$ in $\E$, the adjoint of any map $f \colon A \to \U^X$ is obtained by applying $(-\times X)$ and composing with the evaluation counit $ev_\U \colon \U^X \times X \to \U$. For $f = \prod_X(\iota)$, this adjoint is exactly the top row of the following diagram:
\[\begin{tikzcd}[column sep=large, row sep=large]
    {\prod_X \B \Aut_{/X}(E) \times X} \arrow[r, "{(\prod_X\iota) \times 1_X}"] \arrow[d, "\varepsilon"'] & 
    {\U^X \times X} \arrow[r, "ev_\U"] \arrow[d, "{\varepsilon = (ev_\U, \pi_2)}"'] & 
    \U \\
    {\B \Aut_{/X}(E)} \arrow[r, hook, "\iota"'] & 
    {\U \times X} \arrow[ur, "\pi_1"']
\end{tikzcd}.\]
    The left square is the standard naturality square for the counit $\varepsilon \colon X^\ast \circ \prod_X \to \Id_{\slice{\E}{X}}$ evaluated on the morphism $\iota$. The right triangle commutes by the following observation: the slice-theoretic counit evaluated on the trivial object $\U \times X$ is exactly the Cartesian evaluation map paired with the projection to $X$.
\end{proof}
\subsection{Freeness of the $\Deck(p)$-action}\label{subsection: freeness of the deck(p)-action}
In classical topology, the group of deck transformations of a connected covering space acts freely on the fiber. This follows cleanly from the fact that the actions of $\pi_1(X,x)$ and $\Deck(p)$ on $F$ commute. If there exists a point $e_0\in F$ and a deck transformation $\varphi \in \Deck(p)$ such that $\varphi(e_0)=e_0$, then $\varphi(e)=e$ for all $e \in F$. To see this, note that by $\pi_1(X,x)$-transitivity, for any $e\in F$ there exists an element $\alpha\in\pi_1(X,x)$ such that $\alpha\cdot e_0 = e$. Then,
\[\varphi(e)=\varphi(\alpha\cdot e_0)=\alpha\cdot\varphi(e_0)=\alpha\cdot e_0=e.\]
Because the $\Deck(p)$-action is faithful (\cref{cor: Deck(p)-action is n-faithful}), $\varphi$ must be the identity transformation, proving that the action is free. Using the main \cref{Theorem: action of a product restricted to H is free} of \cref{section: infinty group actions}, we now establish the higher categorical analogue of this classical fact. This result is central in the classification of normal $n$-covering maps (\cref{theorem: characterization normal n-covering maps}).
\begin{theorem}\label{theorem: action of Deck is n-free}
    Let $(X,x)\in\E_\ast$ be a pointed connected object, and let $p\colon E\to X$ be an $(n-1)$-connected $n$-covering map with fiber $F$. Then the $\infty$-action $\Deck(p)\curvearrowright F$ is $n$-free. 
\end{theorem}
\begin{proof}
   As established previously, the $\infty$-actions of the deck $\infty$-group and the fundamental $n$-group commute, yielding an $\infty$-action $\Deck(p)\times \Pi_n(X,x)\curvearrowright F$. First, consider the restricted $\Pi_n(X,x)$-action on $F$. Because $E$ is $(n-1)$-connected, its $n$-truncation $\tau_nE\simeq F\sslash\Pi_n(X,x)$ is also $(n-1)$-connected. By \cref{Theorem: representation theorem for n-covers}, this implies that the restricted $\Pi_n(X,x)$-action is $n$-transitive. Moreover, the restricted $\Deck(p)$-action on $F$ is $n$-faithful by \cref{cor: Deck(p)-action is n-faithful}. We conclude by \cref{Theorem: action of a product restricted to H is free} that the $\Deck(p)$-action is $n$-free.
\end{proof}

\subsection{Normal coverings}\label{subsection: normal coverings}
In classical topology, a connected covering map $p\colon E\to X$ is normal if the induced $\Deck(p)$-action on the fiber $p^{-1}(x)$ is transitive. The classical classification theorem identifies such covers with normal subgroups $A\leq\pi_1(X,x)$, yielding the standard isomorphism $\Deck(p)\cong\pi_1(X,x)/A$. In our higher categorical setting, we restrict our attention to pointed $n$-covering maps whose total objects are $(n-1)$-connected. By \cref{corollary: pointed n-1 connected coverings corresponds to subgroups of homotopy n-th group}, these covers correspond precisely to sub-$1$-groups $A \leq \pi_n(X,x)$. We will show that the normal $n$-covering maps are those corresponding to normal sub-$1$-groups. Because higher homotopy groups are abelian for $n \geq2$, every $(n-1)$-connected $n$-covering map is automatically normal in these higher degrees (\cref{cor: classification normal n-covering maps}).
\begin{definition}\label{def: n-normal covering map}
    An $(n-1)$-connected $n$-covering map $p \colon E \to X$ in $\E$ is \textit{normal} if the $\infty$-action $\Deck(p)\times X\curvearrowright E$ in $\slice{\E}{X}$ is $n$-transitive. We denote the full subcategory of $\Cov^{\sg{n-1}}_n(X)$ spanned by the normal $n$-covering maps by $\Cov_n^{\mathrm{nor}}(X)$.
\end{definition}
\begin{remark}
    When $X$ is pointed and connected, normality can equivalently be stated in terms of the fiber $F$: $p$ is normal if and only if the $\infty$-action $\Deck(p)\curvearrowright F$ is $n$-transitive. Indeed, there is a Cartesian square
\[\begin{tikzcd}
	{F\sslash\Deck(p)} & {E\sslash (\Deck(p)\times X)} \\
	\ast & X
	\arrow[from=1-1, to=1-2]
	\arrow[from=1-1, to=2-1]
	\arrow[from=1-2, to=2-2]
	\arrow[""{name=0, anchor=center, inner sep=0}, two heads, from=2-1, to=2-2]
	\arrow["\lrcorner"{anchor=center, pos=0.125}, draw=none, from=1-1, to=0]
\end{tikzcd},\]
where the bottom horizontal map is an effective epimorphism. It follows that the left vertical map is $(n-1)$-connected if and only if the right one is.
\end{remark}
The key ingredient to prove all the main results of this section (\cref{cor: group deck is the quotient of fundamental n-group} \& \ref{cor: classification normal n-covering maps}, \cref{theorem: characterization normal n-covering maps}) is the following theorem, whose direct corollary is an equivalence $\Deck(p_A)\simeq \Pi_{n}(X,x)\sslash \B ^{n-1}A$. The idea of the proof is inspired by \cite[Thm.~7.1]{Rijke_TheJoinConstruction}, and relies on the internal Yoneda embedding \cref{Yoneda Embedding}. 
\begin{theorem}\label{theorem: Deck p is loop space of the base B}
    Let $\E$ be an $\infty$-topos, and consider a Cartesian square in $\E_\ast:$
\[\begin{tikzcd}
        E \arrow[r] \arrow[d, "p"'] \arrow[dr, phantom, "\lrcorner"{anchor=center, pos=0.125}, near start] & \ast \arrow[d] \\
        X \arrow[r, "f"'] & {\B G}
    \end{tikzcd},\]
where $p$ is an $(n-1)$-connected $n$-covering map. Then there is an equivalence of $n$-groups \[\Deck(p)\simeq G.\] Moreover, this equivalence respects both canonical $\infty$-actions on $E$.
\end{theorem}
\begin{proof}
    The strategy is to show that $\B G$ and $\B\Deck(p)$ are both realized as the image of the same global point $\ast\to(\U^{\sleq{n-1}})^X$ in the universe of $(n-1)$-truncated maps. For $\B\Deck(p)$ this is \cref{prop: image of a map into the fibered universe is deck transformation}. For $\B G$, we use the Yoneda map $\yo_{\B G}\colon\B G\to(\U^{\sleq{n-1}})^{\B G}$ and the fact that $f\colon X\twoheadrightarrow\B G$ is $(n-1)$-connected to identify the image of $\ast\to X\to(\U^{\sleq{n-1}})^X$ with $\B G$. The uniqueness of images then yields the equivalence $\B\Deck(p)\simeq\B G$.

    We first establish the truncation levels of the objects involved. Because $E$ is $(n-1)$-connected and is the fiber of $f$ at the effective epi $\ast\to \B G$, the map $f$ is $(n-1)$-connected as well. Moreover, the map $f$ is an effective epimorphism and $p$ is $(n-1)$-truncated. This implies that $\ast\to \B G$ is $(n-1)$-truncated as well. It follows that $G$ is $(n-1)$-truncated, making both $G$ and $\Deck(p)$ $n$-groups in $\E$. Because the classifying object $\B G$ is $n$-truncated, its diagonal $\Delta_{\B G}\colon \B G\to\B G\times\B G$ is $(n-1)$-truncated. Thus, its classifying family $\B G(-,-)\colon\B G\times\B G\to\U$ factors through the universe of $(n-1)$-truncated objects $\U^{\sleq{n-1}}$. 
    
    Consider the composition \[X\times X\xrightarrow{f\times f}\B G\times \B G\xrightarrow{\B G(-,-)}\U^{\sleq{n-1}}.\]  
    Passing to the Cartesian adjunctions yields the commutative diagram:
\begin{equation}\label{theorem: deck p is loop space eq1}
    \begin{tikzcd}
	X & {(\U^{\sleq{n-1}})^X} \\
	\B G & {(\U^{\sleq{n-1}})^{\B G}}
	\arrow[from=1-1, to=1-2]
	\arrow["f",two heads, from=1-1, to=2-1]
	\arrow["{\yo_{\B G}}", hook, from=2-1, to=2-2]
	\arrow[hook, from=2-2, to=1-2]
\end{tikzcd}.
\end{equation}
The bottom horizontal map is the Yoneda map $\yo_{\B G}$ by definition, which is a monomorphism by the Yoneda Embedding \cref{Yoneda Embedding}. The right vertical map is a monomorphism by \cref{cor: n-connected map induce mono on universe of n-truncated objects}, since $f$ is $(n-1)$-connected as noted above. It follows that $\B G$ is precisely the image of the composition $X\to(\U^{\sleq{n-1}})^X$. Because the composition $\ast\to X\to\B G$ is an effective epi, we deduce:  \[\B G\simeq im(\ast\to X\to (\U^{\sleq{n-1}})^X).\] On the other hand, by \cref{prop: image of a map into the fibered universe is deck transformation}, this image is also $\B\Deck(p')$, where $p'$ is the map classified by the adjoint $X\to\U^{\sleq{n-1}}$ of the map $\ast\to (\U^{\sleq{n-1}})^X$. We now verify that this map $p'$ is precisely $p$. It is determined by the following pasting of Cartesian squares:
\begin{equation}\label{theorem: deck p is loop space eq2}
    \begin{tikzcd}
        E \arrow[r] \arrow[d, "p"'] \arrow[dr, phantom, "\lrcorner"{anchor=center, pos=0.125}, near start] & P \arrow[r] \arrow[d] \arrow[dr, phantom, "\lrcorner"{anchor=center, pos=0.125}, near start] & \B G \arrow[r] \arrow[d, "\Delta"'] \arrow[dr, phantom, "\lrcorner"{anchor=center, pos=0.125}, near start] & {\U_\ast^{\sleq{n-1}}} \arrow[d] \\
        X \arrow[r, "{(x, 1_X)}"'] & {X\times X} \arrow[r, "{f\times f}"'] & {\B G\times \B G} \arrow[r, "{\B G(-,-)}"'] & {\U^{\sleq{n-1}}}
    \end{tikzcd}.
\end{equation}
We must verify that the leftmost pullback evaluates to $E$. The rightmost square is Cartesian by definition of $\B G(-,-)$, and let $P$ be the pullback of the middle square. On one hand, by a simple Fubini argument, $P$ is the pullback of the cospan $X\xrightarrow{f}\B G\xleftarrow{f}X$. On the other hand, using \ref{item: pullback of a product with X}, the leftmost square can be computed as the pullback of the cospan $\ast\xrightarrow{x}X\xleftarrow{}P$. Together we compute the total pullback by the following diagram:
\[\begin{tikzcd}
        E \arrow[r] \arrow[d] \arrow[dr, phantom, "\lrcorner"{anchor=center, pos=0.125}, near start] & P \arrow[r] \arrow[d] \arrow[dr, phantom, "\lrcorner"{anchor=center, pos=0.125}, near start] & X \arrow[d, "f"] \\
        \ast \arrow[r, "x"'] & X \arrow[r, "f"'] & {\B G}
    \end{tikzcd}.\]
    The right and total squares are Cartesian, which implies that the left square is Cartesian as well. From this we deduce that $p$ is indeed classified by $X\to \U^{\sleq{n-1}}$. By the uniqueness of images in $\E_{\ast/}\simeq\E_\ast$, we obtain the canonical pointed equivalence:
\begin{equation}\label{theorem: deck p is loop space eq3}
    \B \Deck(p)\simeq \B G.
\end{equation}
Finally, we verify that this equivalence preserves the $\infty$-actions on $E$. By construction of image factorization, this equivalence lives over $(\U^{\sleq{n-1}})^X$. By adjunction, this in turn yields an equivalence in the slice over $X$: \[\B \Deck(p)\times X\simeq \B G\times X\hspace{1cm}\text{over }\U^{\sleq{n-1}}\times X.\]
By \cref{prop: image of a map into the fibered universe is deck transformation}, the left-hand map classifies the $\Deck(p)\times X$-action on $E$ in $\slice{\E}{X}$. The first perspective on its defining diagram \eqref{diagram: defining the action Deck x X on E in the slice} confirms that this indeed defines the desired $\Deck(p)$-action on $E$. On the other hand, the right-hand map $\B G\times X\to\U^{\sleq{n-1}}\times X$ factors through $\B G\times \B G$, and as such classifies the $G\times X$-action on $E$ in $\slice{\E}{X}$. As for the left-hand case, this indeed classifies the $G$-action on $E$ given by the theorem statement. Thus, the equivalence \eqref{theorem: deck p is loop space eq3} respects both $\infty$-actions on $E$.
\end{proof}
Recall from \cref{corollary: pointed n-1 connected coverings corresponds to subgroups of homotopy n-th group} that a sub-$1$-group $A\leq \pi_n(X,x)$ corresponds to an $(n-1)$-connected $n$-covering map $p_A\colon E_A\to X$. Moreover, the inclusion $A\to\pi_n(X,x)$ is normal if and only if the induced composition of $n$-group maps $\B^{n-1}A\to\B^{n-1}\pi_n(X,x)\to\Pi_n(X,x)$ is normal, by Theorem~\ref{theorem: n-symmetric subgroups of the fundamental n-group are 1-subgroups of n-th homotopy group 2}.
\begin{corollary}\label{cor: group deck is the quotient of fundamental n-group}
    Let $(X,x)$ be a pointed connected object of an $\infty$-topos $\E$, and let $A\to \pi_n(X,x)$ be a normal sub-$1$-group object corresponding to a pointed $n$-covering map $p_A\colon E_A\to X$. Then there is an equivalence of $n$-groups \[\Deck(p_A)\simeq \Pi_{n}(X,x){\sslash \B ^{n-1}A}.\]
    Moreover this equivalence respects the $\infty$-actions on the fiber $F_A$ of $p_A$.
\end{corollary}
\begin{proof}
     By definition, the $n$-covering map $p_A$ arises as the pullback in the left square of the following diagram:
\[\begin{tikzcd}
    {E_A} \arrow[r] \arrow[d, "p_A"'] \arrow[dr, phantom, "\lrcorner"{anchor=center, pos=0.125}] & {\B^n A} \arrow[r] \arrow[d] \arrow[dr, phantom, "\lrcorner"{anchor=center, pos=0.125}] & \ast \arrow[d] \\
    X \arrow[r, "{\eta_X}"'] & {\B \Pi_n(X,x)} \arrow[r, "{\B \pi}"'] & {\B (\Pi_n(X,x) \sslash \B^{n-1} A)}
\end{tikzcd}.\]
By hypothesis, the map of classifying spaces $\B^n A\to\B\Pi_n(X,x)$ is normal, meaning it is the fiber of some quotient map $\B \pi$. Because both squares are Cartesian, so is the outer square. As noted above, $p_A$ is an $(n-1)$-connected $n$-covering map. The conclusion follows directly from \cref{theorem: Deck p is loop space of the base B}.
\end{proof}
\begin{remark}
    For $n=1$ we recover the classical identification \[\Deck(p_A)\simeq\pi_1(X,x)\sslash A\]
    for $p_A$ the $1$-covering map corresponding to a normal subgroup $A\trianglelefteq\pi_1(X,x)$.
\end{remark}
Recall from \cref{section: normal maps} that a normal $\infty$-group map $f\colon A\to G$ might have multiple normality data; that is, there might be multiple compatible $\infty$-group structures on the quotient $G\sslash A$. However, when $f$ defines a normal $n$-symmetric sub-$1$-group, the previous identification allows us to prove that its normality datum is unique.
\begin{corollary}\label{cor: quotient of n-symmetric sub-1-groups are unique}
Let $\E$ be an $\infty$-topos, let $G$ be an $n$-group object in $\E$, and let $f\colon \B ^{n-1}A\to G$ be a normal $n$-symmetric sub-$1$-group. Then the $\infty$-group structure on the quotient $G\sslash \B ^{n-1}A$ is unique and is given by: \[G\sslash \B ^{n-1}A\simeq \Deck(\B ^nA\xrightarrow{\B f} \B G) \hspace{0.5cm} \text{ in } \Grp_n(\E).\]
\end{corollary}
\begin{proof}
    Suppose that we are given two such $\infty$-group quotient structures, represented by connected objects $\B (G\sslash \B ^{n-1}A)$ and $\B (G\sslash \B ^{n-1}A)'$, together with pointed maps from $\B G$ that fit $\B f$ into a fiber sequence. Applying \cref{theorem: Deck p is loop space of the base B} to the respective fiber sequences yields pointed equivalences: \[\B (G\sslash \B ^{n-1}A)\simeq\B\Deck(\B ^nA\to \B G)\simeq \B (G\sslash \B ^{n-1}A)'.\]
\end{proof}

We can now use this \cref{cor: group deck is the quotient of fundamental n-group} to characterize normal $n$-covering maps.
\begin{theorem}\label{theorem: characterization normal n-covering maps}
Let $\E$ be an $\infty$-topos, let $(X,x)$ be a pointed connected object, and let $p\colon E\to X$ be an $(n-1)$-connected $n$-covering map with fiber $F$. The following are equivalent:
\begin{enumerate}
    \item $p$ is a normal $n$-covering map;
    \item there is an equivalence $E\sslash \Deck(p)\simeq X$ under $E$;
    \item $p$ corresponds via \cref{corollary: pointed n-1 connected coverings corresponds to subgroups of homotopy n-th group} to a normal sub-$1$-group $A\hookrightarrow \pi_n(X,x)$.
\end{enumerate}
If one of these conditions is satisfied, there exists a basepoint $e\colon \ast\to F\to E$ rendering $p$ pointed. Moreover, for any such basepoint, the $1$-group $A$ is identified with both the $n$-th homotopy group $\pi_n(E,e)$, and the stabilizer $n$-symmetric $1$-group of the action $\pi_n(X,x)\curvearrowright\pi_{n-1}(F,e)$: \[A\simeq \pi_n(E,e)\simeq \Stab_{\pi_n(X,x)}(1_e).\]
\end{theorem}
\begin{proof}
    $(1\Leftrightarrow 2)$ Using \cref{theorem: action of Deck is n-free}, we infer that $p$ is normal if and only if the fiberwise evaluation $\infty$-action $\Deck(p)\times X\curvearrowright E$ in $\slice{\E}{X}$ is $n$-free and $n$-transitive, i.e. regular. By \cref{theorem: free transitive regular action characterization}, this happens if and only if $E\sslash (\Deck(p)\times X)\simeq 1_X$ in the slice over $X$. By \cref{remark: identification of different quotients of Deck(p)-actions}, the quotient in the slice corresponds to the quotient $E\sslash\Deck(p)$ in $\E$.
    
    \medskip
$(2\Rightarrow 3)$ By hypothesis, there is a Cartesian square: 
\[\begin{tikzcd}
    E \arrow[r] \arrow[d, "p"'] \arrow[dr, phantom, "\lrcorner"{anchor=center, pos=0.125}, near start] & \ast \arrow[d] \\
    X \arrow[r] & {\B \Deck(p)}
\end{tikzcd}.\]
Since $\B \Deck(p)$ is $n$-truncated, the bottom map factors through the $n$-truncation $\B \Pi_n(X,x)$, yielding a pasting of pullback squares:
\[\begin{tikzcd}
    E \arrow[r, "\eta_E"] \arrow[d, "p"'] \arrow[dr, phantom, "\lrcorner"{anchor=center, pos=0.125}, near start] & {F\sslash\Pi_n(X,x)} \arrow[r] \arrow[d] \arrow[dr, phantom, "\lrcorner"{anchor=center, pos=0.125}, near start] & \ast \arrow[d] \\
    X \arrow[r, "\eta_X"'] & {\B \Pi_n(X,x)} \arrow[r] & {\B \Deck(p)}
\end{tikzcd},\]
where $\eta_X$ and $\eta_E$ are the $n$-truncation maps. Notice that $E$ is canonically pointed by the universal property of the pullback. Since it is $(n-1)$-connected, the quotient $F\sslash \Pi_n(X,x)$ is the $n$-truncation $\tau_nE$ and as such lives in $\E^{\sg{n-1},\sleq n}_\ast\simeq\Grp_{(1,n)}(\E)$. Thus, the quotient is the $n$-fold classifying object of an $n$-symmetric $1$-group: \[A\coloneq\Omega^n (F\sslash \Pi_n(X,x)).\]
The above diagram shows that $A\to \Pi_n(X,x)$ is normal with quotient $\Deck(p)$. By \cref{corollary: pointed n-1 connected coverings corresponds to subgroups of homotopy n-th group}, this pullback construction exactly means that $p$ corresponds to $A$, and by Theorem \ref{theorem: n-symmetric subgroups of the fundamental n-group are 1-subgroups of n-th homotopy group 2}, $A$ is normal in $\pi_n(X,x)$.

\medskip

    $(3 \Rightarrow 2)$
        The correspondence of \cref{corollary: pointed n-1 connected coverings corresponds to subgroups of homotopy n-th group} implies that $p$ is the pullback of the canonical composition $\B^n A\to\B^n\pi_n(X,x)\to\B\Pi_n(X,x)$. Using \cref{cor: group deck is the quotient of fundamental n-group}, there is a pasting of Cartesian squares:
        \[\begin{tikzcd}
    {E_A} \arrow[r, "\eta_{E_A}"] \arrow[d, "p_A"'] \arrow[dr, phantom, "\lrcorner"{anchor=center, pos=0.125}, near start] & {\B ^{n}A} \arrow[r] \arrow[d] \arrow[dr, phantom, "\lrcorner"{anchor=center, pos=0.125}, near start] & \ast \arrow[d] \\
    X \arrow[r, "\eta_X"'] & {\B \Pi_n(X,x)} \arrow[r] & {\B \Deck(p_A)}
\end{tikzcd},\]
        such that the composition $X\to \B \Deck(p_A)$ classifies the usual $\infty$-action $\Deck(p_A)\curvearrowright E_A$. This immediately implies that $E_A\sslash\Deck(p_A)\simeq X$ under $E$. Furthermore, because $\eta_X$ and $\eta_{E_A}$ are $n$-truncation maps, $E_A$ is $(n-1)$-connected, as desired: \[\tau_{n-1}E_A\simeq\tau_{n-1}(\tau_n E_A)\simeq\tau_{n-1}\B^{n}A\simeq\ast.\]
     
\medskip
   
Finally, assuming the third condition, we obtain a basepoint $e\colon \ast\to F\to E$ by the universal property of the pullback. Note that this basepoint is not canonical. By \cref{cor: induced action of pi_n-1 in the truncated setting and classifing object of the stabilizer}, we observe that $\B ^n\Stab_{\pi_n(X,x)}(e)\simeq \B ^nA$ as the former is defined as the unique factoring object of the following $(n-2)$-connected/$(n-2)$-truncated factorization:
\[\begin{tikzcd}
	\ast & F & {\B^nA\simeq F\sslash \Pi_n(X,x)\simeq\B\Pi_n(E,e)} \\
	& {\B ^n\Stab_{\pi_n(X,x)}(1_e)}
	\arrow[from=1-1, to=1-2]
	\arrow[from=1-1, to=2-2]
	\arrow[from=1-2, to=1-3]
	\arrow["\simeq"', from=2-2, to=1-3]
\end{tikzcd}.\]
The second map in the factorization is an equivalence since $\ast\to \B ^nA$ is $(n-2)$-connected. Moreover, $\Pi_n(E,e)\simeq\B^{n-1}\pi_n(E,e)$ since $E$ is $(n-1)$-connected.
\end{proof}
\begin{example}\label{example: universal n-covering is normal}
    The universal $n$-covering map $X\langle n\rangle\to X$ is normal since it corresponds to the trivial sub-$1$-group of $\pi_n(X,x)$.
\end{example}
Using the previous characterization of normal $n$-covering maps given in \cref{theorem: characterization normal n-covering maps}, we can classify them with normal sub-$1$-groups of $\pi_n(X,x)$.
\begin{corollary}\label{cor: classification normal n-covering maps}
    There is an equivalence of $\infty$-categories 
    \begin{equation*}
        \Cov^{\nor}_{\ast,n}(X)\simeq \Sub_1^{\nor}(\pi_n(X,x)).
    \end{equation*}
    Moreover when $n>1$ every pointed $(n-1)$-connected $n$-covering map is normal.
\end{corollary}
\begin{proof}
    By \cref{corollary: pointed n-1 connected coverings corresponds to subgroups of homotopy n-th group}, there is an equivalence of $\infty$-categories 
    \begin{equation*}
        \Cov_{\ast,n}^{\sg{n-1}}(X)\simeq \Sub_1\left(\pi_n\left(X,x\right)\right)
    \end{equation*}
    Moreover this equivalence restricts through the subcategories of normal objects. Indeed by \cref{theorem: characterization normal n-covering maps}, normal $n$-covering maps correspond precisely to normal $n$-symmetric sub-$1$-groups of $\pi_n(X,x)$. The $n$-th homotopy group $\pi_n(X,x)$ is an $n$-symmetric $1$-group, so when $n>1$ its sub-$1$-groups are normal by \cref{prop: sufficiently symmetric sub n groups are normal}, which proves the second claim.
\end{proof}
Lastly, we obtain the identification of the $n$-group of deck transformations of a normal $n$-covering map.
\begin{corollary}\label{corollary: group deck of a normal covering is a quotient of fundamental group}
    Let $(X,x)\in\E_\ast$ be a pointed connected object, and let $p\colon E\to X$ be a normal $n$-covering map. Then $\B^{n-1}\pi_n(E,e)\to \Pi_n(X,x)$ is a normal $n$-group morphism, and there is an equivalence of $n$-groups:
    \[\Deck(p)\simeq \Pi_n(X,x)\sslash \B^{n-1}\pi_n(E,e).\]
\end{corollary}
\begin{proof}
    By \cref{theorem: characterization normal n-covering maps}, $E$ has a basepoint $e\colon \ast\to E$, and $p$ corresponds to the normal sub-$1$-group $\pi_n(E,e)$ of $\pi_n(X,x)$. We conclude by \cref{cor: group deck is the quotient of fundamental n-group}. 
\end{proof}
\subsection{Normalizers}\label{subsection: normalizers}
Let $f\colon A\to G$ be a map of $\infty$-groups. We seek a suitable definition for its normalizer $\N(f)\in\Grp(\E)$ with a factorization \[A\xrightarrow{f'} \N(f)\to G,\] where $f'$ is a normal map. Guided by classical group theory, consider the case of an ordinary subgroup $A\leq G$. There is a canonical isomorphism:
\[\Aut_G(G/A)\cong \N_G(A)/A,\]
identifying the $G$-equivariant bijections of the coset space with the quotient of the normalizer $\N _G(A)$ by $A$. Passing to classifying spaces, this implies that the normalizer sits inside a fiber sequence $\B A\to\B \N_G(A)\to\B\Aut_G(G/A)$, defining an $\Aut_G(G/A)$-action on $\B A$. Consequently, $\B \N_G(A)$ can be recovered completely as the quotient $\B A\sslash\Aut_G(G/A)$. 

To construct the normalizer of $f$, it is enough to define the $\infty$-action $\Aut_G(G\sslash A)\curvearrowright\B A$. Recall from Example~\ref{example: group actions: action of a map A-> G} that the $G$-action on $G\sslash A$ is classified by $\B f \in \slice{\E}{\B G} \simeq \Act_G(\E)$. By \cref{def: equivariant autoequivalences objects}, the $\infty$-group of $G$-equivariant autoequivalences of $G\sslash A$ is precisely the deck transformation $\infty$-group of this classifying map:
\[\Aut_G(G\sslash A)\simeq\prod_{\B G}\left(\Aut(G\sslash A)\sslash G\right)\simeq\prod_{\B G}\Aut_{/\B G}(\B A)\simeq\Deck(\B f).\]
As established in \cref{subsection: two persepctives on the action}, there is a canonical $\infty$-action of $\Deck(\B f) \simeq \Aut_G(G\sslash A)$ on $\B A$. This enables us to define the normalizer directly via the corresponding quotient.
\begin{definition}
    The \textit{normalizer $\infty$-group} of $f$, denoted $\N(f)$, is defined by its classifying object as the quotient: 
    \[\B \N(f)\coloneq\B A\sslash\Deck(\B f),\]
    which canonically inherits the basepoint of $\B A$.
\end{definition}
This object is indeed connected since it is the target of an effective epimorphism from a connected object:
\[\begin{tikzcd}
    \B A \arrow[r, "\B f'", two heads] \arrow[d] \arrow[dr, phantom, "\lrcorner"{anchor=center, pos=0.125}, near start] & {\B\N(f)} \arrow[d] \\
    \ast \arrow[r, two heads] & {\B\Deck(\B f)}
\end{tikzcd}.\]
As desired, this construction provides a normality datum for the induced $\infty$-group map $A\xrightarrow{f'}\N(f)$ with quotient $\Deck(\B f)\simeq\Aut_G(G\sslash A)$. To further justify this definition, we establish the following theorem:
\begin{theorem}\label{theorem: properties normalizers}
Let $f\colon A\to G$ be an $\infty$-group map in an $\infty$-topos $\E$.
\begin{enumerate}
    \item There exists a factorization $A\xrightarrow{f'}\N(f)\xrightarrow{\iota} G$ of $f$ in $\Grp(\E)$, where $f'$ is a normal map.
    \item Let $A\xrightarrow{g} K\to G$ be any factorization of $f$ in $\Grp(\E)$ such that $g$ is a normal map. Then there exists an $\infty$-group map $\varphi\colon K\to\N(f)$ such that the following diagram commutes:
\[\begin{tikzcd}[row sep = tiny]
	&& K \\
	A &&&& G \\
	&& {\N(f)}
	\arrow[from=1-3, to=2-5]
	\arrow[dashed,"\varphi", from=1-3, to=3-3]
	\arrow["g", from=2-1, to=1-3]
	\arrow["f'"', from=2-1, to=3-3]
	\arrow[from=3-3, to=2-5, "\iota"']
\end{tikzcd}.\]
\end{enumerate}
\end{theorem}
\begin{proof}
        $(1)\,$ Recall from \eqref{diagram: defining the action Deck x X on E in the slice} that $\B\N(f)$ is obtained as the quotient of $\B\Deck(\B f)\times\B G\curvearrowright\B A$ in the slice over $\B G$. This yields a quotient map $\B A\to\B\N(f)$ over $\B G$. It has been established in the construction that $f'$ is normal.
        
        $(2)\,$ Fix a choice of normality datum $\B A\xrightarrow{\B g}\B K\to \B (K\sslash A)$. We will first construct a map $\Tilde{\varphi}:K\sslash A\to \N(f)\sslash A$, and subsequently show that it lifts to a map of $\infty$-groups $\varphi\colon K\to \N(f)$. Notice that the canonical map $\B K\to \B (K\sslash A)\times \B G$ is a $\B A$-bundle in the slice over $\B G$, since its fiber is $\B A$:
\[\begin{tikzcd}[column sep=large]
    \B A \arrow[r] \arrow[d, "\B f"'] \arrow[dr, phantom, "\lrcorner"{anchor=center, pos=0.125}, near start] & \B K \arrow[d] \\
    \B G \arrow[r, "{(\ast, 1_{\B G})}"] & {\B (K\sslash A)\times \B G}
\end{tikzcd}.\]
It is hence classified by a uniquely determined map into the universal $\B A$-bundle:
\begin{align*}
    (\B \Tilde{\varphi})^\flat:\B (K\sslash A)\times \B G\to \B \Aut_{/\B G}(\B A)\qquad\text{in } \slice{\E}{\B G}.
\end{align*} 
The adjoint to this classifying map is \[\B \Tilde{\varphi}:\B (K\sslash A)\to \B \Deck(\B f)\simeq \B(\N(f)\sslash A).\] We now assemble these maps to compute the following pasting of pullbacks:
\begin{equation}\label{theorem: properties normalizers eq 1}
    \begin{tikzcd}[column sep=large]
        \B A \arrow[dr, phantom, "\lrcorner"{anchor=center, pos=0.125}, near start]\arrow[r] \arrow[d, "\B f"'] & \B K\arrow[dr, phantom, "\lrcorner"{anchor=center, pos=0.125}, near start] \arrow[r, "\B \varphi"] \arrow[d] & {\B \N(f)}\arrow[dr, phantom, "\lrcorner"{anchor=center, pos=0.125}, near start] \arrow[r] \arrow[d] & {\B A\sslash \Aut_{/\B G}(\B A)} \arrow[d] \\
        \B G \arrow[r, "{\ast\times \B G}"] & {\B (K\sslash A)\times \B G} \arrow[r, "{\B \Tilde{\varphi}\times \B G}"] \arrow[rr, "{(\B \Tilde{\varphi})^\flat}"', curve={height=18pt}] & {\B (\N(f)\sslash A)\times \B G} \arrow[r, "\varepsilon"] & {\B \Aut_{/\B G}(\B A)}
    \end{tikzcd}.
\end{equation}
By definition, the $\infty$-action $\Deck(\B f)\times \B G\curvearrowright \B A$ is classified by the counit $\varepsilon$, whose quotient is $\N(f)$. This confirms that the right-most square is Cartesian. The pasting of the right and middle square is Cartesian by definition of the classifying map $(\B \Tilde{\varphi})^\flat$. By the pasting law for Cartesian squares, the induced map between the pullbacks yields the desired map $\B\varphi:\B K\to\B\N(f)$. This map satisfies the required properties: the commutativity of the right triangle follows from inspecting the central square, while the commutativity of the left triangle follows from the top left composition. 
\end{proof}
\begin{remark}
    In \cite{FarjounSegev_NormalClosureAndInjectiveNormalizers}, they characterize normal $1$-group homomorphisms as those $f\colon A\to G$ such that the morphism $\iota\colon \N(f)\to G$ admits a section $s$. In the context of $\infty$-groups, the forward direction is clear from \cref{theorem: properties normalizers}. However the converse is not. One proof strategy would be to show that the composition \[\B G\xrightarrow{\B s} \B \N(f)\to \B \Aut_\ast(\B A)\simeq \B \Aut_{\Grp}(A)\] (which we can construct), defining an $G$-action by automorphisms on $A$, is part of the structure of a sensible notion of \textit{$\infty$-crossed module structure} on $f\colon A\to G$. This could induce a compatible $\infty$-group structure on the quotient $G\sslash A$ (see Remark~\ref{remark: on normal maps 3}).
\end{remark}
When restricting to $n$-symmetric sub-$1$-groups, normalizers behave exactly as ordinary subgroups.
\begin{corollary}\label{cor: characterization normality of n symmetric sub 1 groups with normalizers}
    Let $f\colon \B ^{n-1}A\to G$ be an $n$-symmetric sub-$1$-group of an $n$-group $G$. Then $f$ is normal if and only if the canonical map $\iota\colon\N(f)\to G$ is an equivalence.
\end{corollary}
\begin{proof}
    Suppose first that $f$ is normal. By the second part of \cref{theorem: properties normalizers}, the map $\iota$ has a section $\varphi\colon G\to \N(f)$, which arises as the pullback of the map \[\B \tilde{\varphi}\times \B G\colon\B (G\sslash \B^{n-1}A)\times \B G\to\B (\N(f)\sslash \B^{n-1}A)\times\B G.\] We claim that $\B\tilde{\varphi}$ is an equivalence, implying that $\B \varphi$, and hence $\B\iota$, is one as well. By \cref{cor: quotient of n-symmetric sub-1-groups are unique}, there is a canonical equivalence \[\B\Tilde{\psi}\colon \B(G\sslash\B^{n-1}A)\xrightarrow{\simeq}\B(\N(f)\sslash \B^{n-1}A),\] uniquely determined by the property that the following diagram commutes:
\[\begin{tikzcd}
	{\B G} & {\B\Deck(\B f)\times\B G} \\
	{\B(G\sslash \B^{n-1}A)\times\B G} & {\U^{\sleq n-1}}
	\arrow["{\ast\times \B G}", from=1-1, to=1-2]
	\arrow["{\ast\times\B G}"', from=1-1, to=2-1]
	\arrow[from=1-2, to=2-2]
	\arrow["\simeq","\B\Tilde{\psi}\times\B G"', dashed, from=2-1, to=1-2]
	\arrow[from=2-1, to=2-2]
\end{tikzcd},\]
remembering that $\B(\N(f)\sslash \B^{n-1}A)\coloneq\B\Deck(\B f)$ by definition. The map $\B\tilde{\varphi}\times\B G$ is indeed a solution to this lifting problem by the construction in \eqref{theorem: properties normalizers eq 1}, noting that \[\B \Aut_{/\B G}(\B A)\hookrightarrow\U^{\sleq {n-1}}\times\B G\to\U^{\sleq {n-1}}.\]
This ensures that $\B\Tilde{\psi}\simeq\B\Tilde{\varphi}$, as desired.

Conversely, suppose that $\iota$ is an equivalence. Then the composition \[f\colon \B^{n-1}A\xrightarrow{f'} \N(f)\xrightarrow[\simeq]{\iota}G\]
is a normal map because $f'$ is normal by construction (\cref{theorem: properties normalizers}).
\end{proof}
The link between $n$-covering maps and normalizers is as follows.
\begin{theorem}\label{theorem: Deck p of a covering is a quotient of the normalizer}
    Let $\E$ be an $\infty$-topos, let $(X,x)$ be a pointed connected object, let $p\colon E\to X$ be a connected pointed $n$-covering map in $\E_\ast$, and write $p_n\colon \Pi_n(E,e)\to\Pi_n(X,x)$ for the induced morphism of $n$-groups. Then there is an equivalence of $n$-group objects: 
    \begin{equation*}
        \Deck(p)\simeq \N(p_n)\sslash \Pi_n(E,e).
    \end{equation*}
\end{theorem}
\begin{proof}
    Let $F$ be the fiber of $p$ over the base point $x$. Since $E$ and $X$ are connected, \cref{corollary: Cartesian square of truncation maps} yields a Cartesian square:
\[
    \begin{tikzcd}
        E \ar[r, "\eta_E"] \ar[d,"p"'] \ar[dr, phantom, "\lrcorner"{anchor=center, pos=0.125}, near start]
        & {\B\Pi_n(E,e)} \ar[d,"\B p_n"] \\
        X \ar[r, "\eta_X"] & {\B\Pi_n(X,x)}
    \end{tikzcd},
    \]
showing that the fiber of $\B p_n$ is again $F$. By \cref{prop: group deck of a covering is pi_1-equivariant automorphism of the fiber} applied to $p$, and \cref{corollary: deck transformations as Omega X-equivariant autoequivalences} applied to $\B p_n$, we obtain a sequence of equivalences of $n$-groups: \[\Deck(p)\simeq \Aut_{\Pi_n(X,x)}(F)\simeq \Deck(\B p_n).\]
    By construction, the latter object is precisely the quotient $\N(p_n)\sslash \Pi_n(E,e)$. 
\end{proof}

\section{Examples}\label{section: examples}
We illustrate \cref{corollary: group deck of a normal covering is a quotient of fundamental group} by computing the internal $n$-group of deck transformations of normal $n$-covering maps across a range of $\infty$-topoi. The main input is \cref{cor: classification normal n-covering maps}: over a pointed base $(X,x)$, normal $n$-coverings correspond to normal subgroups $A\leq\pi_n(X,x)$, the normality being automatic for $n\geq2$. The covering attached to such an $A$ is the pullback of the $n$-truncation $X\to\B\Pi_n(X,x)$ along $\B^nA\to\B^n\pi_n(X,x)\xrightarrow{\varepsilon}\B\Pi_n(X,x)$, and its deck group is $\Deck(p)\simeq\Pi_n(X,x)\sslash\B^{n-1}A$.
We explore examples in homotopy types, presheaf $\infty$-topoi, sheaf $\infty$-topoi, and cohesive homotopy types. One major contrast with spaces arises in sheaf $\infty$-topoi: the internal deck group is computed by the same formula in every $\infty$-topos $\E$, while its external meaning varies with the cohomology of $\E$. To access these examples we use the characterization of $\infty$-topoi as left exact localizations of presheaf $\infty$-categories \cite[Theorem 6.1.0.6]{LurieHTT} rather than \cref{def: infinity topos}. We first record the two computations that recur throughout, then open with a non-example delimiting the scope of the theory.
\medskip

\textbf{Universal $n$-coverings.}
For a pointed object $(X,x)\in\E_\ast$, the universal $n$-covering $p_n\colon X\langle n\rangle\to X$ corresponds to the trivial subgroup $1\leq\pi_n(X,x)$. So its deck group is the full fundamental $n$-group:\[\Deck(p_n)\simeq\Pi_n(X,x).\]
\textbf{Eilenberg--Mac Lane objects.}
Let $\iota\colon A\to G$ be a sub-$1$-group of an $n$-symmetric $1$-group $G$ in $\E$. For $n=1$, the map $\B\iota\colon\B A\to\B G$ is a normal $1$-covering if and only if $\iota$ is normal, with $\Deck(\B\iota)\simeq G/A$. For $n\geq2$ the group $G$ is necessarily abelian, $\B^n\iota$ is automatically normal, and
\begin{equation}\label{example: eilenberg MacLane objects}\Deck(\B^n\iota)\simeq \B^{n-1}(G/A).
\end{equation}
\textbf{A non-example: parametrized spectra.}
The tangent $\infty$-category $p\colon \T_{\C}\to\C$ of a presentable $\infty$-category is the codomain fibration whose fiber at $X\in\C$ is the stabilization $\mathrm{Sp}(\slice{\C}{X})$ of the slice over $X$ \cite[§7.3.1]{LurieHigherAlgebra}. The objects of the latter category are spectrum parametrized over $X$. When $\C=\E$ is an $\infty$-topos, its tangent $\infty$-category $\T_\E$ is itself an $\infty$-topos, called of \textit{parametrized spectra} \cite{Hoyois_TopoiOfParametrizedObjects}.
In this $\infty$-topos, the theory yields nothing new: the stable directions of $\T_\E$ are independent of the truncated part. Concretely, the objects of $\T_\E$ are pairs $(X,E)$ with $X\in\E$ and $E\in\mathrm{Sp}(\slice{\Sp}{X})$ a parametrized spectrum over $X$, as explained in a model categorical framework in \cite{HarpazNuitenPrasma_TangentBundleModelCategory}. The terminal object is $(\ast,0)$, and the fibration $p\colon \T_\E\to\E$, $(X,E)\mapsto X$, is both left and right adjoint to the zero section $s_0\colon X\mapsto(X,0_X)$, where $0_X$ is the zero object of $\mathrm{Sp}(\slice{\E}{X})$. We claim that the tangent $\infty$-category of homotopy types $\T_{\Sp}$ has no new interesting truncated object.
\begin{prop}
    An object $(X,E)\in \T_{\Sp}$ is $n$-truncated if and only if $X$ is an $n$-truncated space and $E\simeq 0_X$.
\end{prop}
\begin{proof}
    If $X$ is $n$-truncated, then so is $(X,0_X)=s_0(X)$ since $s_0$ is a right adjoint. Conversely, suppose that $(X,E)$ is $n$-truncated in $\T_\Sp$. Then $X=p(X,E)$ is $n$-truncated, and so is $(X,0_X)$ by the first part. Hence the map $u\colon (X,E)\to (X,0_X)$, adjoint to the identity $1_X\colon X\to X$, is $n$-truncated by \cref{prop: equivalent characterization of truncated maps}. Its base change along any point $x\colon (\ast,0)\to (X,0_X)$ is therefore an $n$-truncated object $(\ast,E_x)$. Looping lowers truncation levels by \cref{lemma: truncation and connectivity of loop space}, so $\Omega^{n+2}_{(\ast,0)}(\ast,E_x)$ is $(-2)$-truncated, i.e. terminal. As the inclusion of spectra $\mathrm{Sp}\subseteq \T_\Sp$ preserves finite limits, this object is $(\ast,\Omega^{n+2}E_x)$, thus $\Omega^{n+2}E_x\simeq 0$ is the zero spectrum. Since $\Omega$ is an equivalence on spectra with inverse $\Sigma$, we get $E_x\simeq 0$. Finally, $\coprod_{x\in X}\ast\twoheadrightarrow X$ is an effective epimorphism, and since $s_0$ has both a left and right adjoint, it preserves coproducts and effective epimorphisms, so $z\colon\coprod_{x\in X}(\ast,0)\twoheadrightarrow(X,0)$ is one as well. By universality of coproducts the base change of $u$ along $z$ is $\coprod_{x\in X}\left((\ast,E_x)\to(\ast,0)\right)$, an equivalence since each $E_x\simeq0$. As base change along an effective epimorphism is conservative, $u$ is an equivalence, so $E\simeq0_X$.
\end{proof}
\subsection{Homotopy types}
\begin{example}
    The map $p\colon S^{2n+1}\to\mathbb{C}P^n$ with fiber $S^1$ is a universal $2$-cover, so \[\Deck(p)\simeq\Pi_2(\mathbb{C}P^n)\simeq S^1.\]
    The Hopf fibration $S^1\to S^3\xrightarrow{\eta}S^2$ is the case $n=1$.
\end{example}
\begin{example}
    Since $H^3(S^3;\Z)\cong\Z$, a choice of generator classifies a map $S^3\to \B\mathbb{C}P^{\infty}$ whose homotopy fiber is the universal $3$-covering of $S^3$. Consequently,
    \[\Deck(S^3\langle3\rangle\to S^3)\simeq\Pi_3(S^3)\simeq \mathbb{C}P^{\infty}.\]
\end{example}

\textbf{A $2$-covering from an action.}
    The circle $S^1$ acts on $S^2$ by rotation around an axis through the poles. The homotopy quotient is $\mathbb{C}P^\infty\vee\mathbb{C}P^\infty$, and the quotient map $p\colon S^2\to \mathbb{C}P^\infty\vee\mathbb{C}P^\infty$ is a $2$-covering with fiber $S^1$. Its classifying map fits into the Cartesian square
    \[\begin{tikzcd}
        S^2 \ar[r] \ar[d] \ar[dr, phantom, "\lrcorner"{anchor=center, pos=0.125}, near start]
        & \B S^1 \ar[d] \\
        \mathbb{C}P^\infty\vee\mathbb{C}P^\infty \ar[r]
        & \B(S^1\times S^1)
    \end{tikzcd},\]
    corresponding to the diagonal subgroup $\Z\xrightarrow{\Delta}\Z\times\Z\cong\pi_2(\mathbb{C}P^\infty\vee\mathbb{C}P^\infty)$. It follows that:
    \[\Deck(p)\simeq\Pi_2(\mathbb{C}P^\infty\vee\mathbb{C}P^\infty)\sslash\B\Z\simeq(S^1\times S^1)\sslash\B\Z\simeq S^1.\]
\textbf{A deck $3$-group with non-trivial $k$-invariants.}
The examples so far produce Eilenberg--Mac Lane deck groups. Here is the first with non-trivial Postnikov invariant. For $m\geq 0$, the subgroup inclusion $m\Z\hookrightarrow\Z=\pi_3(S^2)$ gives rise to a $2$-connected $3$-covering $p\colon E_m\to S^2$. By \cref{remark: decomposition of n-covering with the universal n-cover}, $p$ factors as $E_m\to S^3\xrightarrow{\eta}S^2,$
where the first map is the pullback:
\[
\begin{tikzcd}
	E_m \arrow[r] \arrow[d]  \ar[dr, phantom, "\lrcorner"{anchor=center, pos=0.125}, near start] & B^3(m\mathbb{Z}) \arrow[d] \\
	S^3 \arrow[r, "\tau_3"] & B^3\mathbb{Z}
\end{tikzcd}.
\]
The first $k$-invariant of $S^2$ is the cup square $\iota^2\in H^4(\B^2\Z;\Z)$, so by \cref{remark: construction of the n group quotient by a sub 1 group}:
\[\B\Deck(p)\simeq \B\left(\Pi_3(S^2)\sslash \B^2(m\Z)\right)\simeq \mathrm{fib}\big(\B^2\Z\xrightarrow{\iota^2}\B^4\Z\xrightarrow{\B^4\pi}\B^4 (\Z/m\Z)\big).\] Thus $\Deck(p)$ is a $3$-group with $\pi_1=\Z$ and $\pi_2=\Z/m\Z$. For $m=0$ one recovers the universal $3$-cover, $\Deck(p_3)=\Pi_3(S^2)$; while for $m=1$, $\Deck(p)\simeq \Pi_{2}(S^2)\simeq S^1$ and $p=\eta$ is the universal $2$-cover viewed as a $3$-covering.

\subsection{Presheaf $\infty$-topoi}
In a presheaf $\infty$-topos $\Psh(\C)$, limits and $n$-truncations are computed objectwise. We spell out the arrow $\infty$-topos. A second instance is $\Psh(\Orb_G)$, presheaves on the orbit category of a topological group $G$: there the deck group of a $G$-equivariant covering carries a $G$-structure compatible, for a normal $n$-covering, with the $G$-structure on the quotient $\Pi_n(X,x)\sslash\B^{n-1}\pi_n(E,e)$.
\medskip

\textbf{The arrow $\infty$-topos.}
Let $\E=\Fun([1],\Sp)$ be the arrow $\infty$-topos, and take the base $\mathbf X=\Id_X$ for a pointed connected $(X,x)$. A morphism $p\colon\mathbf{E}\to\mathbf{X}$ is a triangle
\[\begin{tikzcd}
    E_1 \ar[rr,"e"] \ar[dr,"p_1"] & & E_2 \ar[dl,"p_2"'] \\
    & X &
\end{tikzcd},\]
and $p$ is an $n$-covering if and only if $p_1$ and $p_2$ are so in $\Sp$. The deck transformation $\infty$-group $\Deck(p)$ is the homomorphism of $\infty$-groups:
\[
\Deck(p_1)\times_{\Map(E_1,E_2)}\Deck(p_2)\to\Deck(p_2),\]
whose domain consists of compatible pairs $(\varphi_1,\varphi_2)$, i.e. $e\circ\varphi_1=\varphi_2\circ e$.  When $p_1,p_2$ are normal $n$-coverings, \cref{corollary: group deck of a normal covering is a quotient of fundamental group} computed in $\E$ gives $\Deck(p)\simeq\Pi_n(\mathbf X,x)\sslash\B^{n-1}\pi_n(\mathbf E,e)$, which objectwise is the morphism \[\Pi_n(X,x)\sslash\B^{n-1}\pi_n(E_1,e_1)\to\Pi_n(X,x)\sslash\B^{n-1}\pi_n(E_2,e_2).\] The same corollary in $\Sp$ identifies its two ends with $\Deck(p_1)$ and $\Deck(p_2)$ respectively. Hence \[\Deck(p_1)\times_{\Map(E_1,E_2)}\Deck(p_2)\simeq\Deck(p_1),\] and in particular there is a fiber sequence of $\infty$-groups
\[\Deck(e)\to\Deck(p_1)\to\Deck(p_2).\]
\subsection{Sheaf $\infty$-topoi}\label{subsection: example: sheaf topoi}
In this subsection the deck group $\Deck(p)$ is best read through two lenses: the internal $n$-group that our classification computes, and its external image $\Gamma(\Deck(p))$, where $\Gamma\colon\E\to\Sp$, $F\mapsto\E(\ast,F)$, is the global sections functor (see \cref{remark: global section of internal deck transformations}). For the Eilenberg--Mac Lane coverings of \eqref{example: eilenberg MacLane objects} the internal deck group is connected when $n\geq 2$, so it has a single component, yet its external deck transformations $\pi_0\Gamma(\Deck(p))$ form a possibly large group determined by the cohomology of the base. 
\begin{definition}[{\cite[7.2.2.14]{LurieHTT}}]
    Let $\E$ be an $\infty$-topos, $n\geq 0$ an integer, and $A$ an abelian $1$-group object in $\E$. The \textit{$n$-th cohomology group of $\E$ with coefficients in $A$} is \[H^n(\E;A)\coloneq \pi_0\E(\ast, \B^n A)=\pi_0\Gamma (\B^n A).\]
\end{definition}
For $0\leq i\leq n$, basing at the canonical point of $\Gamma (\B^n A)$, one obtains \[\pi_i(\Gamma (\B^{n} A))\simeq \pi_0\Omega^i(\Gamma(\B^n A))\simeq \pi_0 \Gamma (\Omega ^i (\B^n A))\simeq \pi_0\Gamma (\B^{n-i} A)\simeq H^{n-i}(\E; A),\]
where the second equivalence holds because $\Gamma$, being a right adjoint, commutes with loops. For $i>n$ these groups vanish, since $\B^nA$ is $n$-truncated and $\Gamma$ preserves truncation levels.\begin{remark}\label{remark: cohomology in a 1-site}
    When $\E=\Sh(\C)$ is an $\infty$-category of sheaves on a $1$-site $\C$, then the cohomology of $\E$ is identified with the cohomology of the $1$-topos $\E^{\sleq 0}\simeq \Sh_{\Set}(\C)$ of $0$-truncated objects \cite[7.2.2.17]{LurieHTT}.
\end{remark}
\textbf{Sheaves on a space.}
Let $B$ be a topological space, $\Sh(B)$ the $\infty$-topos of sheaves on $B$, and let $\Z\in \mathscr{A}\mathrm{b}(\Sh_{\Set}(B))$ be the constant sheaf of integers. It is a $2$-symmetric, hence $n$-symmetric by \cref{prop: k-symmetric groups are k+1-symmetric for large k}, $1$-group object, and as such lives in $\Grp_{(1,n)}(\Sh(B))$. Let $A=m\Z\subseteq\Z$ be a constant subsheaf of abelian groups, which is a normal sub-$1$-group. Then $p_A\colon \B^n A\to\B^n \Z$ is a normal $n$-covering map with deck $n$-group $\Deck(p_A)\simeq \B^{n-1}(\Z/A)$, again a constant sheaf of $n$-groups. Externally, \[\pi_0\Gamma(\Deck(p_A))\simeq \pi_0\Gamma(\B^{n-1}(\Z/A))\simeq H^{n-1}\left(\Sh(B);\Z/A\right)\simeq H^{n-1}(B;\Z/A),\]
    where the last identification follows from \cref{remark: cohomology in a 1-site}.
\begin{example}
    Let $B=S^1$, $A=2\Z$, and $n=2$. The internal deck $2$-group of $p_A\colon \B^2(2\Z)\to\B^2\Z$ is $\Deck(p_A)\simeq \B (\Z/2\Z)$, and externally \[\pi_0\Gamma(\Deck(p_A))\simeq\pi_0\Gamma(\B(\Z/2))\simeq H^1(S^1;\Z/2)\simeq\Z/2.\]
    There is a non-trivial external deck transformation, arising from the Möbius $\Z/2$-torsor over $S^1$, even though the internal deck group is connected.
\end{example}

\textbf{Étale $\infty$-topos.}
Let $X$ be a scheme and $X_{\et}\coloneq(\mathscr{S}\mathrm{ch}_{/X})_{\et}$ its small étale $1$-site of étale maps $U\to X$ \cite[§II.1]{Deligne_CohomologieEtale}. We work in the $\infty$-topos $\E=\Sh(X_{\et})$. Let $\bG_{m,X}=\operatorname{Spec}\Z[x,x^{-1}]\times X$ be the multiplicative group of units, an abelian group object in $\Sh_{\Set}(X_{\et})=\E^{\sleq0}$, viewed as a $2$-symmetric $1$-group object in $\E$. Let $k\in\mathbb{N}$ be invertible on $X$. The \textit{Kummer sequence}  \cite[\href{https://stacks.math.columbia.edu/tag/03PK}{Tag 03PK}]{stacks-project} \[1\to \mu_{k,X}\xrightarrow{\iota}\bG_{m,X}\xrightarrow{\cdot k}\bG_{m,X}\to 1\] is exact as étale sheaves, exhibiting $\mu_{k,X}$ (the $k$-th roots of unity) as a normal sub-$1$-group of $\bG_{m,X}$ (\cref{prop: kernel is a sub n group}), with quotient $\bG_{m,X}/\mu_{k,X}\simeq\bG_{m,X}$. The induced map $p\coloneq\B^n\iota\colon\B^n\mu_{k,X}\to\B^n\bG_{m,X}$ is a normal $n$-covering, with \[\Deck(p)\simeq \B^{n-1} (\bG_{m,X}/\mu_{k,X})\simeq \B^{n-1}\bG_{m,X}.\]
Externally, $\Gamma(\B^{n-1}\bG_{m,X})$ has homotopy groups
\[\pi_i(\Gamma(\B^{n-1}\bG_{m,X}))\simeq H^{n-1-i}(\Sh(X_{\et}); \bG_{m,X})\simeq H^{n-1-i}_{\et}(X;\bG_{m,X}),\qquad 0\leq i\leq n-1,\]
where the last identification follows from \cref{remark: cohomology in a 1-site}.
\begin{example}[Picard groupoid as deck $2$-group]\label{example: picard deck group}
For $n=2$, the external deck group of $p$ is the Picard groupoid of $X$  \cite[\href{https://stacks.math.columbia.edu/tag/03P8}{Tag 03P8}]{stacks-project}:
    \[\pi_0\Gamma(\Deck(p))\simeq H^1_{\et}(X;\bG_{m,X})\simeq\operatorname{Pic}(X),\qquad
    \pi_1\Gamma(\Deck(p))\simeq H^0_{\et}(X;\bG_{m,X})\simeq\mathcal{O}_X^\times(X).\]
Thus every line bundle on $X$ arises as a deck transformation of this covering. Over a field $L$, Hilbert's Theorem~90 gives $H^1_{\et}(L,\bG_{m,L})=0$. Over $X=\operatorname{Spec}\Z[\tfrac1k]$ the Picard group vanishes, but the units are $\mathcal{O}_X^\times(X)=\Z[\tfrac1k]^\times$, so the deck $2$-group has non-trivial $\pi_1$. Over $X=\mathbb P^N_L$ with $L$ a field of characteristic prime to $k$:
    \[\pi_0\Gamma(\Deck(p))\simeq\operatorname{Pic}(\mathbb P^N_L)\simeq\Z,\qquad
    \pi_1\Gamma(\Deck(p))\simeq\mathcal{O}^\times(\mathbb P^N_L)\simeq L^\times,\]
where $\pi_0$ is free on the Serre twist $\mathcal{O}(1)$.
\end{example}
\begin{example}\label{example: kummer theory}
    Let $L$ be a field containing a primitive $k$-th root of unity, $X=\operatorname{Spec} L$, and $d\mid k$. The sequence
    \[1\to\mu_{d,X}\to\mu_{k,X}\to\mu_{k/d,X}\to1\]
    is exact as étale sheaves. The $2$-covering $p\colon\B^2\mu_{d,X}\to\B^2\mu_{k,X}$ corresponding to the sub-$1$-group $\mu_{d,X}\hookrightarrow\mu_{k,X}$ is normal, so $\Deck(p)\simeq\B(\mu_{k,X}/\mu_{d,X})\simeq\B(\mu_{k/d,X})$. Externally, by Kummer theory \cite[\href{https://stacks.math.columbia.edu/tag/03PK}{Tag~03PK}]{stacks-project}:
    \[\pi_0\Gamma(\Deck(p))\simeq H^1_{\et}(L,\mu_{k/d,X})\simeq L^\times/(L^\times)^{k/d},\]
    where $(L^\times)^a$ denotes the $a$-th powers. For $L=\mathbb{Q}(\zeta_k)$ the cyclotomic field and any $d\mid k$ with $k/d\geq2$, this is an infinite group: the covering $\B^2\mu_d\to\B^2\mu_k$ admits infinitely many external deck transformations. By contrast, for the corresponding covering $p\colon\B^2(\Z/d)\to\B^2(\Z/k)$ of homotopy types, the external deck
    $2$-group is connected: $\pi_0\Gamma(\Deck(p))\simeq H^1(\ast;\Z/(k/d))=0$.
\end{example}
                                                      
\subsection{Cohesive $\infty$-topoi}
Let $\E$ be an $\infty$-topos and $(\Delta,\Gamma)\colon\E\to\Sp$ its terminal geometric
morphism.
\begin{definition}[{\cite[Def.~4.1.8]{Schreiber_DifferentialCohomologyCohesiveInfinityTopos}}]
    The $\infty$-topos $\E$ is \emph{cohesive} if
    \begin{enumerate}
        \item the constant-object functor $\Delta$ is fully faithful;
        \item $\Gamma$ admits a fully faithful right adjoint $\coDisc\colon\Sp\to\E$;
        \item $\Delta$ admits a finite product preserving left adjoint $\Pi\colon\E\to\Sp$.
    \end{enumerate}
\end{definition}
Objects of a cohesive $\infty$-topos should be thought of as $\infty$-groupoids equipped with cohesion: a topology, a smooth structure, and so on. To reflect this we write $\Disc\coloneq\Delta$ and call $\Disc(S)$ a \emph{discrete} object; for $X\in\E$, we call $\Pi X$ the \emph{fundamental $\infty$-groupoid} of $X$. The adjunctions assemble into the string $\Pi\dashv\Disc\dashv\Gamma\dashv\coDisc$, and we write $\sh\coloneq\Disc\circ\Pi\colon\E\to\E$ for the idempotent \emph{shape} modality, with unit $\eta_X\colon X\to\sh X$. Only the shape will be used explicitly in what follows.

\begin{example}
    The canonical examples are the cohesive $\infty$-topoi $\Sp_{\mathrm{Top}}$ and $\Sp_{\mathrm{Smooth}}$ of topological and smooth $\infty$-groupoids, defined as the $\infty$-categories of sheaves on the site of topological (resp. smooth) Cartesian spaces $\mathbb{R}^n$ \cite[Def.~6.3.3, resp.\ 6.4.7]{Schreiber_DifferentialCohomologyCohesiveInfinityTopos}. Topological (resp.\ smooth) manifolds embed fully faithfully as $0$-truncated objects \cite[6.4.10]{Schreiber_DifferentialCohomologyCohesiveInfinityTopos}, and the shape $\sh M$ of a manifold is its standard underlying homotopy type \cite[6.3.5.2, resp.\ 6.4.5.3]{Schreiber_DifferentialCohomologyCohesiveInfinityTopos}. In what follows we write $\Sp_{\Mfd}$ for either of these cohesive $\infty$-topoi, say \textit{manifold} for either the topological or the smooth case, and denote by $\Mfd\subseteq \Sp_{\Mfd}^{\sleq 0}\subseteq\Sp_{\Mfd}$ the full inclusion of manifolds.
\end{example}

These examples show that the covering theory of \cref{section : Coverings} must be adapted to the cohesive structure: a manifold $M$ is $0$-truncated in $\Sp_{\Mfd}$, hence has trivial fundamental $n$-groups, even when its underlying homotopy type is non-trivial. The remedy, following Myers in cohesive homotopy type theory \cite[Def.~9.15]{Myers_GoodFibrationsThroughTheModalPrism}, is to single out the maps that are determined by their shape.
\begin{definition}\label{def: étale map and cohesive coveings}
    Let $\E$ be a cohesive $\infty$-topos.
    \begin{enumerate}
        \item A map $p\colon E\to X$ is \emph{$\sh$-étale} if the naturality square
        \begin{equation}\label{def: etale map eq}
            \begin{tikzcd}
            E \ar[r,"\eta_E"] \ar[d,"p"'] & \sh E \ar[d,"\sh p"] \\
            X \ar[r,"\eta_X"'] & \sh X
        \end{tikzcd}
        \end{equation}
        is Cartesian.
        \item A map $p\colon E\to X$ is a \emph{cohesive $n$-covering} if it is $(n-1)$-truncated and $\sh$-étale.
    \end{enumerate}
\end{definition}

$\sh$-étale maps are called $\sh$-closed in \cite[§5.2.7]{Schreiber_DifferentialCohomologyCohesiveInfinityTopos}. It is shown there that $\sh$-étale maps over $X$ are precisely the maps classified by $\Disc(\U)$ for a universe $\U\in\Sp$ \cite[Prop.~5.4.42]{Schreiber_DifferentialCohomologyCohesiveInfinityTopos}. \begin{remark}
    Schreiber shows in \cite[Rmk.~5.2.46]{Schreiber_DifferentialCohomologyCohesiveInfinityTopos}  that these are precisely the locally constant maps in $\slice{\E}{X}$ in the sense of \cite[Def.~A.1.12]{LurieHigherAlgebra}: there is an effective epimorphism $\coprod_i U_i\twoheadrightarrow X$ with equivalences $U_i\times_X E\simeq U_i\times\Disc(S_i)$ over each $U_i$, for homotopy types $S_i\in\Sp$. Hoyois shows that this holds when $\E$ is locally contractible \cite{Hoyois_HigherGaloisTheory}.
\end{remark} 
For this statement to interact with our framework, we record that $\Disc$ preserves univalence.

\begin{lemma}\label{lemma: univalent map of discrete objects}
    Let $\E$ be a cohesive $\infty$-topos. A map $p\colon\U_\ast\to\U$ in $\Sp$ is univalent if and only if $\Disc(p)$ is univalent in $\E$.
\end{lemma}
\begin{proof}
    The localization $\Pi\colon\E\to\Sp$ is locally Cartesian: it preserves pullbacks along maps of the form $\Disc A\to\Disc B$ \cite[Prop.~4.1.35]{Schreiber_DifferentialCohomologyCohesiveInfinityTopos}. The claim is then \cite[Thm.~3.12]{GepnerKock_UnivalenceInLCCC}.
\end{proof}

The key computation is that mapping \emph{into} a discrete object only sees the shape of the source.

\begin{lemma}\label{lemma: internal hom in a discrete object}
    Let $\E$ be a cohesive $\infty$-topos, $A\in\E$ be a discrete object, i.e.\ $A\simeq\sh A$, and $X\in\E$ be arbitrary. Then precomposition with the unit $\eta_X\colon X\to\sh X$ induces an equivalence of
    internal mapping objects \[\Map(\eta_X,A)\colon\Map(\sh X,A)\xrightarrow{\ \simeq\ }\Map(X,A).\]
\end{lemma}
\begin{proof}
    For every $T\in\E$ there is a chain of equivalences, natural in $T$:
    \begin{align*}
        \E\big(T,\Map(X,A)\big) &\simeq\E(T\times X,\sh A) &&\proofstep{$-\times X\dashv\Map(X,-)$, and $A\simeq\sh A$}\\
        &\simeq\Sp\big(\Pi(T\times X),\Pi A\big) &&\proofstep{$\Pi\dashv\Disc$}\\
        &\simeq\Sp(\Pi T\times\Pi X,\Pi A) &&\proofstep{$\Pi$ preserves finite products}\\
        &\simeq\Sp\big(\Pi T\times\Pi(\sh X),\Pi A\big) &&\proofstep{$\Pi\sh\simeq\Pi$, as $\Disc$ is fully faithful}\\
        &\simeq\E\big(T\times\sh X,\sh A\big) &&\proofstep{$\Pi$ preserves finite products, $\Pi\dashv\Disc$}\\
        &\simeq\E\big(T,\Map(\sh X,A)\big) &&\proofstep{$-\times\sh X\dashv\Map(\sh X,-)$, and $A\simeq\sh A$.}
    \end{align*}
    Tracing through the adjunctions, the fourth equivalence is precomposition with $\Pi(1_T\times\eta_X)$, so the induced equivalence of representing objects is $\Map(\eta_X,A)$ by the Yoneda lemma.
\end{proof}

\begin{prop}\label{prop: group deck in a cohesive covering is group deck of the shape}
    Let $\E$ be a cohesive $\infty$-topos, and let $p\colon E\to X$ be a $\sh$-étale map in $\E$. Then there is an equivalence of $\infty$-group objects \[\Deck(p)\simeq\Deck(\sh p).\] In particular, the deck $\infty$-group of a $\sh$-étale map is discrete.
\end{prop}
\begin{proof}
    By \cref{theorem: group objects are pointed connected objects} it suffices to produce a pointed equivalence $\B\Deck(p)\simeq\B\Deck(\sh p)$. Since $\sh p$ is $\sh$-étale, it is classified by a map $\corner{\sh p}\colon\sh X\to\Disc(\U)$ for a universe $\U\in\Sp$, and $\Disc(\U)$ is univalent by \cref{lemma: univalent map of discrete objects}. As the naturality square \eqref{def: etale map eq} of a $\sh$-étale map is Cartesian, $p$ is classified by the composite $\corner{p}\colon X\xrightarrow{\eta_X}\sh X\xrightarrow{\corner{\sh p}}\Disc(\U)$. By \cref{prop: image of a map into the fibered universe is deck transformation}, $\B\Deck(\sh p)$ is the image of the point $\ast\xrightarrow{\corner{\sh p}}\Map(\sh X,\Disc\U)$, and $\B\Deck(p)$ the image of its composite with 
    \[\Map(\eta_X,\Disc\U)\colon\Map(\sh X,\Disc\U)\longrightarrow\Map(X,\Disc\U).\]
    The latter is an equivalence by \cref{lemma: internal hom in a discrete object}, since $\Disc(\U)$ is discrete; hence the two images agree.
\end{proof}

Thus the deck $n$-group of a cohesive $n$-covering is computed by its shape, which lives in the full subcategory of discrete objects, i.e.\ in $\Sp$. Combining with the quotient description of deck groups:

\begin{corollary}\label{cor: cohesive deck group}
    Let $\E$ be a cohesive $\infty$-topos, and $p\colon E\to X$ be a pointed cohesive $n$-covering map in $\E$ such that $\sh X$ is connected and $\sh p$ is a normal $n$-covering in $\Sp$. Then \[\Deck(p)\simeq\Pi_n(\sh X,x)\sslash\B^{n-1}\pi_n(\sh E,e).\]
\end{corollary}
\begin{proof}
    Combine \cref{prop: group deck in a cohesive covering is group deck of the shape} with \cref{corollary: group deck of a normal covering is a quotient of fundamental
    group} applied to $\sh p$.
\end{proof}

\medskip
\textbf{Recovering manifolds.}
Somewhat surprisingly, this framework recovers the classical theory of $1$-coverings of manifolds internally. Suppose there is an object $\A\in\E$ such that the localization $\Pi\colon\E\to\Sp$ is the localization at the class $S\coloneq\{c_i\times(\A\to\ast)\}_{i\in I}$, for a small set of generators $\{c_i\}_{i\in I}$ of $\E$; we then say $\A$ \emph{exhibits the cohesion} of $\E$ \cite[Def.~5.2.48]{Schreiber_DifferentialCohomologyCohesiveInfinityTopos}. In the topological (resp.\ smooth) case, $\A=\mathbb{R}$ exhibits the cohesion \cite[Rmk.~5.2.52]{Schreiber_DifferentialCohomologyCohesiveInfinityTopos}.

\begin{definition}[{\cite[Def.~5.2.56]{Schreiber_DifferentialCohomologyCohesiveInfinityTopos}}]
    Let $\A$ exhibit the cohesion of $\E$. An object $M\in\E$ is a \emph{manifold of dimension $n$} if there is a small family of monomorphisms $\phi_j\colon\A^n\hookrightarrow M$ such that
    \begin{enumerate}
        \item the induced map $\phi\colon\coprod_j\A^n\twoheadrightarrow M$ is an effective epimorphism;
        \item the \v{C}ech nerve $\check{C}(\phi)$ of $\phi$ is degreewise a coproduct of copies of $\A^n$.
    \end{enumerate}
\end{definition}

If $p\colon E\to M$ is a cohesive $1$-covering and $M$ is a manifold of dimension $n$, then $E$ is again a manifold of dimension $n$: pulling back the charts $\phi_j$ along $p$ yields a cover of $E$, and since $\sh(\A^n)\simeq\ast$, every locally constant $0$-truncated map over $\A^n$ is equivalent to a projection $\A^n\times\Disc(S)\simeq\coprod_{s\in S}\A^n\to\A^n$ for a $0$-truncated space $S$. Thus the pulled-back charts again satisfy (1) and (2). The theory of cohesive $1$-coverings of manifolds therefore restricts to a theory of $1$-coverings of \emph{manifolds by manifolds}.

\medskip

In $\E=\Sp_{\Mfd}$, this recovers the classical topological (resp.\ smooth) picture since $\Mfd\subseteq\Sp_{\Mfd}$ is fully faithful. Moreover, it follows that the external set of deck transformations of a cohesive $1$-covering $p\colon N\to M$ over a manifold $M$ is the set of homeomorphisms (resp. diffeomorphisms) $N\cong N$ commuting with $p$: \[\Gamma\Deck(p)\simeq \slice{\Sp_{\Mfd}}{M}^\simeq(N,N)\simeq\slice{\Mfd}{M}^\simeq(N,N).\]
Because it is also equivalent to \[\Gamma\Deck(\sh p)\simeq \slice{\Sp}{\Pi M}^{\simeq}(\Pi N, \Pi N),\]
this shows that topological (resp. smooth) deck transformations are equivalently the deck transformations by homotopy equivalences on their underlying homotopy types.

\begin{remark}
    The analogous statement fails for higher coverings, in the sense that a cohesive $n$-covering of a manifold by a manifold is necessarily a $1$-covering. Indeed, let $M$ be a manifold and $p\colon N\to M$ a cohesive $n$-covering with $N$ a manifold. Then $M$ and $N$ are both $0$-truncated, so $p$ is $0$-truncated as well, a cohesive $1$-covering. Because $M\to \sh M$ is an effective epimorphism, $\sh p$ is a $1$-covering map in $\Sp$.
    \end{remark}
    
\printbibliography
\end{document}